\documentclass[10pt, a4paper,reqno]{amsart}
\usepackage{latexsym}
\usepackage{amsmath}
\usepackage[dvips]{graphicx}%
\usepackage{amsfonts}%
\usepackage{amssymb}
\usepackage{mathtools}
\usepackage{mathrsfs}

\usepackage{color}
\definecolor{purple}{rgb}{0.9,0.01,0.3}

\def\R{\mathbb R}
\def\N{\mathbb N}
\def\l{\lambda}
\def\C{\mathbb C}

\def\D{\mathcal{D}}
\def\Sp{\partial B}
\def\cal{\mathcal}

\def\H{{\cal H}^{n-1}}
\def\F{{\cal F}}

\def\QP{\mathcal{QP}}
\def\QV{\mathcal{QV}}

\def\a{\alpha}
\def\b{\beta}

\def\g{\gamma}
\def\de{\delta}
\def\e{\varepsilon}

\def\r{\varrho}
\def\s{\sigma}
\def\vphi{\varphi}
\def\om{\omega}
\newcommand{\medint}{-\kern -,375cm\int}
\newcommand{\medintinrigo}{-\kern -,315cm\int}
\def\G{\Gamma}
\def\per{{\rm Per}}
\def\Om{\Omega}

\def\pa{\partial}
\def\Id{{\rm Id}\,}

\def\Div{{\rm div}}
\newcommand{\dist}{{\rm dist}}
\numberwithin{equation}{section}
\newtheorem{theorem}{Theorem}[section]

\newtheorem{corollary}[theorem]{Corollary}

\newtheorem{lemma}[theorem]{Lemma}

\newtheorem{proposition}[theorem]{Proposition}

\theoremstyle{definition}
\begingroup

\newtheorem{remark}[theorem]{Remark}

\endgroup

\begin{document}

\title[]{Isoperimetry and stability properties of balls \\ with respect to nonlocal energies}

\author{A. Figalli}
\address{Department of Mathematics, The University of Texas at Austin,  2515 Speedway Stop C1200, Austin, Texas 78712-1202, USA}
\email{figalli@math.utexas.edu}

\author{N. Fusco}
\address{Dipartimento di Matematica e Applicazioni ``R. Caccioppoli'',
Universit\`a di Napoli ``Federico II'',
via Cintia, 80126 Naples, Italy}
\email{n.fusco@unina.it}

\author{F. Maggi}
\address{Department of Mathematics, The University of Texas at Austin,  2515 Speedway Stop C1200, Austin, Texas 78712-1202, USA}
\email{maggi@math.utexas.edu}

\author{V. Millot}
\address{ Laboratoire J.-L. Lions (CNRS UMR 7598), Universit\'e Paris Diderot - Paris 7, Paris, France}
\email{millot@ljll.univ-paris-diderot.fr}

\author{M. Morini}
\address{Dipartimento di Matematica, Universit\`a di Parma, viale G. P. Usberti 53/a, Campus,
43100 Parma, Italy}
\email{massimiliano.morini@unipr.it}

\maketitle

\begin{abstract}
We obtain a sharp quantitative isoperimetric inequality for nonlocal $s$-perimeters, uniform with respect to $s$ bounded away from $0$. This allows us to address local and global minimality properties of balls with respect to the volume-constrained minimization of a free energy consisting of a nonlocal $s$-perimeter plus a non-local repulsive interaction term. In the particular case $s=1$ the $s$-perimeter coincides with the classical perimeter,
and our results improve the ones of Kn\"upfer and Muratov \cite{KM2d,KM} concerning minimality of balls of small volume in isoperimetric problems with a competition between perimeter and a nonlocal potential term. More precisely, their result is extended to its maximal range of validity concerning the type of nonlocal potentials considered, and is also generalized to the case where local perimeters are replaced by their nonlocal counterparts.

\end{abstract}

%
%
%
%
%
%
%
%
%


\section{Introduction}

In the recent paper \cite{CRS}, Caffarelli, Roquejoffre, and Savin have initiated the study of Plateau-type problems with respect to a family of nonlocal perimeter functionals. A regularity theory for such nonlocal minimal surfaces has been developed by several authors \cite{CG,BFV,FV,SV,DdPW}, while the relation of nonlocal perimeters with their local counterpart has been investigated in \cite{CV,ADPM}. The isoperimetry of balls in nonlocal isoperimetric problems has been addressed in \cite{FLS}. Precisely, given $s\in(0,1)$ and $n\ge 2$, one defines the $s$-perimeter of a set $E\subset\R^n$ as
\[
P_s(E):=\int_E\int_{E^c}\frac{dx\,dy}{|x-y|^{n+s}}\in[0,\infty]\,.
\]
As proved in \cite{FLS}, if $0<|E|<\infty$ then we have the nonlocal isoperimetric inequality
\begin{equation}
  \label{nonlocal isoperimetric inequality}
  P_s(E)\ge \frac{P_s(B)}{|B|^{(n-s)/n}}\,|E|^{(n-s)/n}\,,
\end{equation}
where $B_r:=\{x\in\R^n:|x|<r\}$, $B:=B_1$, and $|E|$ is the Lebesgue measure of $E$. Notice that the right-hand side of \eqref{nonlocal isoperimetric inequality} is equal to $P_s(B_{r_E})$,  the $s$-perimeter of a ball of radius $r_E=(|E|/|B|)^{1/n}$ -- so that $|E|=|B_{r_E}|$. Moreover, again in \cite{FLS} it is shown that equality holds in \eqref{nonlocal isoperimetric inequality} if and only if $E=x+B_{r_E}$ for some $x\in\R^n$. In \cite{FMM} the  following stronger form of \eqref{nonlocal isoperimetric inequality} was proved:
\begin{equation}\label{nonlocal isoperimetric inequality quantitative FMM}
  P_s(E)\ge\frac{P_s(B)}{|B|^{(n-s)/n}}\,|E|^{(n-s)/n}\Big\{1+\frac{A(E)^{4/s}}{C(n,s)}\Big\}\,,
\end{equation}
where $C(n,s)$ is a non-explicit positive constant depending on $n$ and $s$ only, while
\begin{equation}\label{asymmetry}
  A(E):=\inf\Big\{\frac{|E\Delta (x+B_{r_E})|}{|E|}:x\in\R^n\Big\}
\end{equation}
measures the $L^1$-distance of $E$ from the set of balls of volume $|E|$ and is commonly known as the {\it Fraenkel asymmetry of $E$} (recall that, given two sets $E$ and $F$, $|E\Delta F|:=|E\setminus F|+|F\setminus E|$). Our first main result improves \eqref{nonlocal isoperimetric inequality quantitative FMM} by providing the sharp decay rate for $A(E)$ in \eqref{nonlocal isoperimetric inequality quantitative}. Moreover, we control the constant $C(n,s)$ appearing in \eqref{nonlocal isoperimetric inequality quantitative FMM} and make sure it does not degenerate as long as $s$ stays away from $0$.

\begin{theorem}\label{thm 1}
  For every $n\ge 2$ and $s_0\in(0,1)$ there exists a positive constant $C(n,s_0)$ such that
  \begin{equation}\label{nonlocal isoperimetric inequality quantitative}
  P_s(E)\ge \frac{P_s(B)}{|B|^{(n-s)/n}}\,|E|^{(n-s)/n}\,\Big\{1+\frac{A(E)^2}{C(n,s_0)}\Big\}\,,
  \end{equation}
  whenever $s\in[s_0,1]$ and $0<|E|<\infty$.
\end{theorem}

\begin{remark}\label{remark blowup constant}
  {\rm The constant $C(n,s_0)$ we obtain in \eqref{nonlocal isoperimetric inequality quantitative} is not explicit. It is natural to conjecture that $C(n,s_0)\approx 1/s_0$ as $s_0\to 0^+$, see \eqref{qip1.1-star} below. Letting $s\to 1$
  we recover the sharp stability result for the classical perimeter, that was first proved in \cite{FMP} by symmetrization methods and later extended to anisotropic perimeters in \cite{FiMP} by mass transportation. The latter approach yields an explicit constant $C(n)$ in \eqref{nonlocal isoperimetric inequality quantitative} when $s=1$, that grows polynomially in $n$. It remains an open problem to prove \eqref{nonlocal isoperimetric inequality quantitative} with an explicit constant $C(n,s)$.}
\end{remark}

We next turn to consider nonlocal isoperimetric problems in presence of nonlocal repulsive interaction terms. The starting point is provided by Gamow model for the nucleus, which consists in the volume constraint minimization of the energy $P(E)+V_\a(E)$, where $P(E)$ is the (distributional) perimeter of $E\subset\R^n$ defines as
$$
P(E):=\sup \Bigl\{\int_E {\rm div} X(x)\,dx : X \in C^1_c(\R^n;\R^n),\,|X|\leq 1 \Bigr\}\,,
$$
while, given $\a\in(0,n)$, $V_\a(E)$ is the Riesz potential
\begin{equation}
\label{eq:Va}
V_\a(E):=\int_E\,\int_E\,\frac{dx\,dy}{|x-y|^{n-\a}}\,.
\end{equation}
By minimizing $P(E)+V_\a(E)$ with $|E|=m$ fixed, we observe a competition between the perimeter term, that tries to round up candidate minimizers into balls, and the Riesz potential, that tries to smear them around. (Notice also that, by Riesz inequality, balls are actually the volume constrained {\em maximizers} of $V_\a$.)

It was recently proved by Kn\"upfer and Muratov that:\\
(a) If $n=2$ and $\a \in (0,2)$, then there exists $m_0=m_0(n,\a)$ such that Euclidean balls of volume $m \leq m_0$ are the only minimizers of $P(E)+V_\a(E)$ under the volume constraint $|E|=m$ \cite{KM2d} . \\
(b) If $n=2$ and $\a$ is sufficiently close to 2, then balls are the unique minimizers for $m \leq m_0$
while for $m>m_0$ there are no minimizers \cite{KM2d}.\\
(c) If $3\le n\le 7$ and $\a\in(1,n)$, then the result in (a) holds \cite{KM}.

In \cite{BC}, Bonacini and Cristoferi have recently extended both (b) and (c) above to the case $n\geq 3$,
and have also shown that balls of volume $m$ are volume-constrained $L^1$-local minimizers of $P(E)+V_\a(E)$ if $m<m_\star(n,\a)$, while they are never volume-constrained $L^1$-local minimizers if $m>m_\star(n,\a)$. The constant $m_\star(n,\a)$ is characterized in terms of a minimization problem, that is explicitly solved in the case $n=3$ (in particular, in the physically relevant case $n=3$, $s=1$, and $\a=2$ (Coulomb kernel), one finds $m_\star(3,1,2)=5$,
a result that was actually already known in the physics literature since the 30's \cite{Bohr, Feenberg, Frenkel}).
Let us also mention that, in addition to (b), further nonexistence results are contained 
in \cite{KM,LO}.

We stress that, apart from the special case $n=2$, all these results are limited to the case $\a\in(1,n)$, named the far-field dominated regime by Kn\"upfer and Muratov to mark its contrast to the near-field dominated regime $\a\in(0,1]$.
Our second and third main results extend (a) and (c) above in two directions: first, by covering the full range $\a\in(0,n)$ for all $n \geq 3$, and second, by including the possibility for the dominant perimeter term to be a nonlocal $s$-perimeter. The global minimality threshold $m_0$ is shown to be uniformly positive with respect to $s$ and $\a$ provided they both stay away from zero.

The local minimality threshold $m_\star(n,s,\a)$ is characterized in terms of a minimization problem.
In order to include the classical perimeter as a limiting case when $s \to 1$, we recall that, by combining \cite[Theorem 1]{CV} with \cite[Lemma 9 and Lemma 14]{ADPM}, one finds that
  \begin{equation}
    \label{limit s to 1}
      \lim_{s\to 1^-}(1-s)\,P_s(E)=\om_{n-1}\,P(E)
  \end{equation}
  whenever $E$ is an open set with $C^{1,\gamma}$-boundary for some $\gamma >0$
  (from now on, $\omega_n$ denotes the volume of the $n$-dimensional ball of radius $1$).
  Hence, to recover the classical perimeter we need to suitably renormalize the $s$-perimeter.

\begin{theorem}\label{thm 2}
  For every $n\ge 2$, $s_0\in(0,1)$, and $\a_0\in(0,n)$, there exists $m_0=m_0(n,s_0,\a_0)>0$ such that, if $m\in(0,m_0)$, $s\in(s_0,1)$, and $\a\in(\a_0,n)$, then the variational problems
  \begin{eqnarray*}
  &&\inf\bigg\{\frac{1-s}{\omega_{n-1}}\,P_s(E)+V_\a(E):|E|=m\bigg\}\,,
  \\
  &&\inf\Big\{P(E)+V_\a(E):|E|=m\Big\}\,,
  \end{eqnarray*}
  admit balls of volume $m$ as their (unique up to translations) minimizers.
\end{theorem}

\begin{remark}
  An important open problem is, of course, to provide explicit lower bounds on $m_0$.
\end{remark}

Let us now define a positive constant $m_\star$ by setting
\begin{equation}
  \label{mstar}
  m_\star(n,s,\a):=
  \begin{cases}
  \displaystyle \omega_n\,\bigg(\frac{n+s}{n-\a}\,\frac{s\,(1-s)\,P_s(B)}{\omega_{n-1}\,\a\,V_\a(B)}\bigg)^{n/(\a+s)}\,, &\hspace{1cm} \text{if $s\in(0,1)$}\,,\\[10pt]
  \displaystyle \omega_n\,\bigg(\frac{n+1}{n-\a}\,\frac{P(B)}{\a\,V_\a(B)}\bigg)^{n/(\a+1)}\,, &\hspace{1cm} \text{if $s=1$}\,.
  \end{cases}
\end{equation}
The constant $m_\star(n,s,\a)$ is the threshold for volume-constrained $L^1$-local minimality of balls with respect to the functional $\frac{1-s}{\omega_{n-1}}\,P_s+V_\a$, as shown in the next theorem:

\begin{theorem}
  \label{thm 3}
  For every $n\ge 2$, $s\in(0,1)$, and $\a\in(0,n)$, let $m_\star=m_\star(n,s,\a)$ be as in \eqref{mstar}. For every $m\in(0,m_\star)$ there exists $\e_\star=\e_\star(n,s,\a,m)>0$ such that, if $B[m]$ denotes a ball of volume $m$, then
  \begin{eqnarray}\label{local minimizer ball s alfa}
\frac{1-s}{\omega_{n-1}}\,P_s(B[m])+V_\a(B[m])\le \frac{1-s}{\omega_{n-1}}\,P_s(E)+V_\a(E)\,,
  \end{eqnarray}
  whenever $|E|=m$ and $|E\Delta B[m]|\le \e_\star\,m$. Moreover, if $m>m_\star$, then there exists a sequence of sets $\{E_h\}_{h\in\N}$ with $|E_h|=m$ and $|E_h\Delta B[m]|\to0$ as $h\to\infty$ such that \eqref{local minimizer ball s alfa} fails with $E=E_h$ for every $h\in\N$.
\end{theorem}


Both Theorem \ref{thm 1} and Theorem \ref{thm 2} are obtained by combining a Taylor's expansion of nonlocal perimeters near balls, discussed in section \ref{section fuglede}, with a uniform version of the regularity theory developed in \cite{CRS,CG}, presented in section \ref{section regularity}. In the case of Theorem \ref{thm 1}, these two tools are combined in section \ref{section stability} through a suitable version of Ekeland's variational principle. We implement this approach, that was introduced in the case $s=1$ by Cicalese and Leonardi \cite{CL}, through a penalization argument closer to the one adopted in \cite{AFM}. Due to the nonlocality of $s$-perimeters, the implementation itself will not be straightforward, and will require to develop some lemmas of independent interest, like the nucleation lemma (Lemma \ref{nucleation lemma}) and the truncation lemma (Lemma \ref{lemma truncation}).

Concerning Theorem \ref{thm 2}, our proof is inspired by the strategy used in \cite{FM}
 (see also \cite{FMarma} for a related argument) to show the isoperimetry of balls in isoperimetric problems with log-convex densities. Starting from the results in sections \ref{section fuglede} and \ref{section regularity}, the proof of Theorem \ref{thm 2} is given in section~\ref{section gamow}.

Finally, the proof of Theorem \ref{thm 3} is based on some second variation formulae for nonlocal functionals (discussed in section \ref{section variation formulae}), which are then exploited to characterize the threshold for volume-constrained stability (in the sense of second variation) of balls  in section \ref{section stability threshold}. The passage from stability to $L^1$-local minimality is finally addressed in section \ref{section stability implies local minimality}. The proof of this last result is pretty delicate since we do not know that the ball is a global minimizer,
a fact that usually plays a crucial role in this kind of arguments.

\section*{Acknowledgment} The work of NF and MM was supported by ERC under FP7 Advanced Grant n. 226234. The work of NF was partially carried on at the University of
Jyv\"askyl\"a under the FiDiPro Program.
The work of AF was supported by NSF Grant DMS-1262411. The work of FM was supported by NSF Grant DMS-1265910. The work of VM was supported by ANR Grant 10-JCJC 0106.


\section{A Fuglede-type result for the fractional perimeter}\label{section fuglede}

In this section we are going to prove Theorem \ref{thm 1} on {\it nearly spherical sets}. Precisely, we shall consider bounded open sets $E$ with $|E|=|B|$, $\int_Ex\,dx=0$, and whose boundary satisfies
\begin{equation}\label{nearly}
 \pa E=\{(1+u(x))x:\,x\in\Sp\}\,,\qquad\mbox{where $u\in C^1(\Sp)$}\,,
 \end{equation}
 for some $u$ with $\|u\|_{C^1(\pa B)}$ small. We correspondingly seek for a control on some fractional Sobolev norm of $u$ in terms of $P_s(E)-P_s(B)$. More precisely, we shall control
 \[
[u]_{\frac{1+s}{2}}^2:= [u]_{H^{\frac{1+s}{2}}(\pa B)}^2=\iint_{\Sp\times\Sp}\frac{|u(x)-u(y)|^2}{|x-y|^{n+s}}\,d\H_x\,d\H_y\,,
 \]
 as well as the $L^2$-norm of $u$. This kind of result is well-known in the local case (see Fuglede \cite[Theorem 1.2]{F}), and takes the following form in the nonlocal case.

 \begin{theorem}\label{fuglede}
 There exist constants $\e_0\in(0,1/2)$ and $c_0>0$, depending only on $n$, with the following property: If $E$ is a nearly spherical set as in \eqref{nearly}, with $|E|=|B|$, $\int_Ex\,dx=0$, and $\|u\|_{C^1(\Sp)}<\e_0$, then
  \begin{equation}\label{fug0}
 P_s(E)-P_s(B)\geq c_0\,\Big([u]_{\frac{1+s}{2}}^2+s\,P_s(B)\,\|u\|_{L^2(\Sp)}^2\Big)\,,\qquad\forall s\in(0,1)\,.
 \end{equation}
 \end{theorem}

  \begin{remark}\label{remfug}
 If we multiply by $1-s$ in \eqref{fug0} and then take the limit $s\to1^-$, then by \eqref{limit s to 1} and \eqref{eq:norms} we get
 $P(E)-P(B)\geq c(n)\,\|u\|^2_{H^1}$ whenever $u\in C^{1,\g}(\pa B)$ for some $\g\in(0,1)$ (thus, on every Lipschitz function $u:\pa B\to\R$ by density). Thus Theorem \ref{fuglede} implies \cite[Theorem 1.2(4)]{F}.
 \end{remark}

In order to prove Theorem \ref{fuglede}, we need to premise some facts concerning hypersingular Riesz operators on the sphere. Following \cite[pp. 159--160]{Sa}, one defines the {\it hypersingular Riesz operator on the sphere of order $\g\in(0,1) \cup(1,2)$} as
\begin{equation}\label{defD}
\D^\g u(x):=\frac{\g\,2^{\g-1}}{\pi^{\frac{n-1}{2}}}\,\frac{\G(\frac{n-1+\g}2)}{\G(1-\frac\g2)}\,\,{\rm p.v.}\left(\int_{\pa B}\frac{u(x)-u(y)}{|x-y|^{n-1+\g}}\,d\H_y\right)\,,\qquad x\in\Sp\,,
\end{equation}
cf. \cite[Equations (6.22) and (6.47)]{Sa}.
(Here, $\Gamma$ denotes the usual Euler's Gamma function, and the symbol p.v. means that the integral is taken in the Cauchy principal value sense.)
By \cite[Lemma~6.26]{Sa}, the $k$-th eigenvalue of $\D^\g$ is given by
\begin{equation}
  \label{se1}
  \l_{k}^*(\g):=\frac{\G(k+\frac{n-1+\g}2)}{\G(k+\frac{n-1-\g}2)}-\frac{\G(\frac{n-1+\g}2)}{\G(\frac{n-1-\g}2)}\,,\qquad k\in\N\cup\{0\}\,,
\end{equation}
(so that $\l_{k}^*(\g)\geq 0$, $\l_{k}^*(\g)$ is strictly increasing in $k$, and $\l_{k}^*(\g)\uparrow\infty$ as $k\to\infty$). Moreover, if  we denote by ${\mathcal S}_k$  the finite dimensional subspace of spherical harmonics of degree $k$, and by  $\{Y_k^i\}_{i=1}^{d(k)}$ an orthonormal basis for ${\mathcal S}_k$ in $L^2(\Sp)$, then
\begin{equation}
  \label{se2}
\D^\g Y_k=\l_{k}^*(\g)\,Y_k\,,\qquad\forall k\in\N\cup\{0\}\,.
\end{equation}
 When no confusion arises, we shall often denote by $Y_k$ a generic element in ${\mathcal S}_k$. Given $s\in(0,1)$, let us now introduce the operator
\begin{equation}
  \label{Is definizione}
  {\mathscr I}_su(x):=2\,{\rm p.v.}\left(\int_{\Sp}\frac{u(x)-u(y)}{|x-y|^{n+s}}\,d\H_y\right)\,,\qquad u\in C^2(\Sp)\,,
\end{equation}
so that, for every $u\in C^2(\pa B)$,
\begin{equation}\label{lb000}
{\mathscr I}_su=\frac{2^{1-s}\,\pi^{\frac{n-1}{2}}}{1+s}\,\frac{\G(\frac{1-s}2)}{\G(\frac{n+s}2)}\,\D^{1+s}u\,,
\end{equation}
and
\begin{equation}\label{lb1}
[u]_{\frac{1+s}{2}}^2=\int_{\pa B}u\,{\mathscr I}_su\,d\H\,.
\end{equation}
Let us denote by $\l_k^s$ the $k$-th eigenvalue of ${\mathscr I}_s$. By \eqref{se1}, \eqref{se2}, and \eqref{lb000} we find that $\l_k^s$ satisfies
\begin{equation}
  \label{ok}
  \l_0^s=0\,,\qquad \l_{k+1}^s>\l_k^s\,,\qquad {\mathscr I}_sY_k=\l_k^s\,Y_k\,,\qquad\forall k\in\N\cup\{0\}\,,
\end{equation}
and $\lambda_k^s\uparrow\infty$ as $k\to\infty$. If we denote by
$$
a_k^i(u):=\int_{\Sp}u\,Y_k^i\,d\H
$$
the Fourier coefficient of $u$ corresponding to $Y_k^i$, then we obtain
\begin{equation}\label{fou1}
[u]_{\frac{1+s}{2}}^2=\sum_{k=0}^\infty\sum_{i=1}^{d(k)}\,\l_k^s\,a_k^i(u)^2\,.
\end{equation}
Concerning the value of $\lambda_1^s$ and $\lambda_2^s$, we shall need the following proposition.

\begin{proposition}\label{millot}
One has
\begin{eqnarray}\label{millot1}
 \lambda_1^s&=&s(n-s)\frac{P_s(B)}{P(B)}\,.
\\
 \label{remeige1}
 \lambda_2^s&=&\frac{2n}{n-s}\lambda_1^s\,.
 \end{eqnarray}
 \end{proposition}

 \begin{proof}
 Since each coordinate function $x_i$, $i=1,\dots,n$, belongs to ${\mathcal S}_1$, we have
 ${\mathscr I}_sx_i=\lambda_1^sx_i$. Hence, inserting $x_i$ in \eqref{lb1} and adding up over $i$, yields
 \begin{equation}\label{millot2}
 \lambda_1^s=\frac{1}{P(B)}\iint_{\Sp\times\Sp}\frac{d\H_x\,d\H_y}{|x-y|^{n+s-2}}\,.
 \end{equation}
 For $z\in\R^n\setminus\{0\}$, we now set
 $$
 \mathcal{K}(z):=-\frac{1}{n+s-2}\frac{1}{|z|^{n+s-2}}\,.
$$
Splitting $\nabla\mathcal{K}$ into its tangential and normal components to $\Sp$, we compute for $y\not\in{\overline B}$ the integral
  \begin{align}\label{millot3}
 L(y):=&\int_{\Sp}\frac{(x-y)\cdot(x-y)}{|x-y|^{n+s}}\,d\H_x\\
 =&\int_{\Sp}\nabla_x \mathcal{K}(x-y)\cdot x\,d\H_x-\int_{\Sp}\nabla_x\mathcal{K}(x-y)\cdot y\,d\H_x\nonumber\\
 =&\int_{\Sp}(1-x\cdot y)\frac{\pa \mathcal{K}}{\pa \nu(x)}(x-y)\,d\H_x-\int_{\Sp}\nabla_{\tau}\mathcal{K}(x-y)\nabla_{\tau}(x\cdot y)\,d\H_x\nonumber\\
 =&\!:\mathcal{A}(y)-\mathcal{B}(y)\,.\nonumber
     \end{align}
 We now evaluate separately $\mathcal{A}(y)$ and $\mathcal{B}(y)$. Noticing that $\Delta{\mathcal K}(z)=-s/|z|^{n+s}$, we first integrate  $\mathcal{A}(y)$ by parts to obtain
  \begin{align*}
  \mathcal{A}(y)&=\int_B\Delta_x \mathcal{K}(x-y)(1-x\cdot y)\,dx+\int_B\nabla_x \mathcal{K}(x-y)\nabla_x(1-x\cdot y)\,dx\\
  &=-s\int_B\frac{1-x\cdot y}{|x-y|^{n+s}}\,dx+\int_B\frac{|y|^2-x\cdot y}{|x-y|^{n+s}}\,dx\\
  &=(1-s)\int_B\frac{1-x\cdot y}{|x-y|^{n+s}}\,dx+\int_B\frac{|y|^2-1}{|x-y|^{n+s}}\,dx\,.
  \end{align*}
 We now denote by $\Delta_{{\mathbb S}^{n-1}}$ the standard Laplace--Beltrami operator on the sphere and recall that $-\Delta_{{\mathbb S}^{n-1}}x_i=(n-1)x_i$ for $i=1,\dots,n$. Integrating $\mathcal{B}(y)$ by parts leads to
    \begin{align*}
\mathcal{B}(y)&=-\int_{\Sp}\mathcal{K}(x-y)\Delta_{{\mathbb S}^{n-1}}(x\cdot y)\,d\H_x=(n-1)\int_{\Sp}\mathcal{K}(x-y)x\cdot y\,d\H_x\\
&=
-\frac{n-1}{n+s-2}\int_{\Sp}\frac{x\cdot y}{|x-y|^{n+s-2}}\,d\H_x\,.
     \end{align*}
 From the above expressions of $\mathcal A$ and $\mathcal B$, we can let $y$ converge to a point on $\Sp$ to find
  \begin{equation}\label{millot4}
 L(y)=(1-s)\int_B\frac{1-x\cdot y}{|x-y|^{n+s}}\,dx+\frac{n-1}{n+s-2}\int_{\Sp}\frac{x\cdot y}{|x-y|^{n+s-2}}\,d\H_x\,,\quad y\in\partial B\,.
 \end{equation}
Integrating over $\Sp$ the first  integral on the right hand side of the previous equality, and using the divergence theorem again, we get
 \begin{align*}
 \int_Bdx\int_{\Sp}\frac{1-x\cdot y}{|x-y|^{n+s}}\,d\H_y&= \int_Bdx\int_{\Sp}\frac{(y-x)\cdot y}{|x-y|^{n+s}}\,d\H_y\,dx\\
 &=\int_Bdx\int_{\Sp}\frac{\pa \mathcal{K}}{\pa \nu}(y-x)\,d\H_y=-\int_Bdx\int_{B^c}\Delta_y\mathcal{K}(y-x)\,dy\\
 &=s\int_B\int_{B^c}\frac{1}{|x-y|^{n+s}}\,dx\,dy=sP_s(B)\,.
\end{align*}
From this formula, integrating both sides of \eqref{millot4} and recalling \eqref{millot2} and \eqref{millot3}, we obtain
 \begin{equation}\label{millot5}
 \lambda_1^s=s(1-s)\frac{P_s(B)}{P(B)}+\frac{n-1}{(n+s-2)P(B)}\iint_{\Sp\times\Sp}\frac{x\cdot y}{|x-y|^{n+s-2}}\,d\H_x\,d\H_y\,.
 \end{equation}
 To deal with the last integral of the previous equality we need to rewrite $P_s(B)$ as follows
 \begin{align*}
 P_s(B)&=\int_{B^c}dy\int_B\frac{(x-y)\cdot(x-y)}{|x-y|^{n+s+2}}\,dx=-\frac{1}{n+s}\int_{B^c}dy\int_B\nabla_x\Bigl(\frac{1}{|x-y|^{n+s}}\Bigr)\cdot(x-y)\,dx\\
 &=-\frac{1}{n+s}\int_{B^c}\biggl(-n\int_B\frac{dx}{|x-y|^{n+s}}+\int_{\Sp}\frac{(x-y)\cdot x}{|x-y|^{n+s}}\,d\H_x\biggr)\,dy\\
 &=\frac{n}{n+s}P_s(B)-\frac{1}{n+s}\int_{B^c}dy\int_{\Sp}\frac{(x-y)\cdot x}{|x-y|^{n+s}}\,d\H_x\,.
 \end{align*}
 Therefore
  \begin{align*}
  P_s(B)&=\frac{1}{s}\int_{\Sp}d\H_x\int_{B^c}\frac{(y-x)\cdot x}{|x-y|^{n+s}}\,dy \\
  &=-\frac{1}{s(n+s-2)}\int_{\Sp}d\H_x\int_{B^c}\nabla_y\Bigl(\frac{1}{|x-y|^{n+s-2}}\Bigr)\cdot x\,dy\\
  &=\frac{1}{s(n+s-2)}\iint_{\Sp\times\Sp}\frac{x\cdot y}{|x-y|^{n+s-2}}\,d\H_x\,d\H_y\,.
     \end{align*}
     Combining this last equality with \eqref{millot5} leads to the proof of \eqref{millot1}.

  Finally, using \eqref{se1} and exploiting the factorial property of the Gamma function $\G(z+1)=\G(z)\,z$ for every $z\in\C\setminus\{-k:k\in\N\cup\{0\}\}$, we see that
  \begin{equation}\label{ober}
 \l_1^*(\a)=\frac{\alpha}{\kappa}\, \frac{\Gamma(\alpha+\kappa)}{\Gamma(\kappa)} \,,\quad \l_2^*(\a)=\frac{1+\alpha+2\kappa}{1+\kappa}\,\l_1^*(\a)\,,\quad\kappa:=\frac{n-1-\alpha}{2}\,.
\end{equation}
 Since $\a=1+s$,  we infer from \eqref{lb000} and \eqref{ober} that $\l_2^s/\l_1^s=\l_2^*(\a)/\l_1^*(\a)=\frac{2n}{n-s}$ which is precisely identity \eqref{remeige1}.
 \end{proof}

 \begin{proof}[Proof of Theorem \ref{fuglede}] {\it Step 1.} We start by slightly rephrasing the assumption. Precisely, we consider a function $u\in C^1(\Sp)$ with $\|u\|_{C^1(\Sp)}\le1/2$ such that there exists $t\in(0,2\e_0)$ with the property that the bounded open  set $F_t$ whose boundary is given by
 $$
 \pa F_t=\{(1+tu(x))x:\,x\in\Sp\}\,,
 $$
 satisfies
$$
 |F_t|=|B|\,,\qquad \int_{F_t}x\,dx=0\,.
$$
 We thus aim to prove that, if $\e_0$ and $c_0$ are small enough, then
 \begin{equation}\label{fug0star}
 P_s(F_t)-P_s(B)\geq c_0\,t^2\,\Big([u]_{\frac{1+s}{2}}^2+s\,P_s(B)\,\|u\|_{L^2}^2\Big)\,,\qquad\forall s\in(0,1)\,.
 \end{equation}
 Changing to polar coordinates, we first rewrite
 $$ P_s(F_t)=\iint_{\Sp\times\Sp}\biggl(\int_{0}^{1+tu(x)}\int_{1+tu(y)}^{+\infty}\frac{r^{n-1}\r^{n-1}}{(|r-\r|^2+r\r|x-y|^2)^{\frac{n+s}{2}}}\,dr\,d\r\biggr)\,d\H_x\,d\H_y\,.$$
Then, symmetrizing this formula leads to
\begin{multline*}
   P_s(F_t)=\frac12\iint_{\Sp\times\Sp}\biggl(\int_{0}^{1+tu(x)}\int_{1+tu(y)}^{+\infty} f_{|x-y|}(r,\r)\,dr\,d\r \\
   +\int_{0}^{1+tu(y)}\int_{1+tu(x)}^{+\infty} f_{|x-y|}(r,\r)\,dr\,d\r\biggr)\,d\H_x\,d\H_y\,,
  \end{multline*}
  where, for $r,\r,\theta>0$, we have set
 $$
  f_{\theta}(r,\r):=\frac{r^{n-1}\r^{n-1}}{(|r-\r|^2+r\r\,\theta^2)^{\frac{n+s}{2}}}\,.
 $$
Using the convention $\int_a^b=-\int_b^a$, we formally have
  $$
  \int_0^b\int_a^{+\infty}+ \int_0^a\int_b^{+\infty} = \int_a^b\int_a^{b} + \int_0^a\int_a^{+\infty} + \int_0^b\int_b^{+\infty}\,,
  $$
  so that
 \begin{multline}\label{fug5}
   P_s(F_t)=\frac12\iint_{\Sp\times\Sp}\biggl(\int_{1+tu(y)}^{1+tu(x)}\int_{1+tu(y)}^{1+tu(x)} f_{|x-y|}(r,\r)\,dr\,d\r \biggr)\,d\H_x\,d\H_y\\
   +\iint_{\Sp\times\Sp}\biggl(\int_{0}^{1+tu(x)}\int_{1+tu(x)}^{+\infty} f_{|x-y|}(r,\r)\,dr\,d\r\biggr)\,d\H_x\,d\H_y\,.
  \end{multline}
  Rescaling variables, we find that
  \begin{multline*}
 \int_{\pa B}\biggl(\int_0^{1+tu(x)}\int_{1+tu(x)}^{+\infty}f_{|x-y|}(r,\r)\,dr\,d\r\biggr)\,d\H_y\\
  =(1+tu(x))^{n-s}\int_{\pa B}\int_0^1\int_1^{+\infty}f_{|x-y|}(r,\r)\,dr\,d\r\,d\H_y\,,\qquad \forall x\in\Sp\,.
  \end{multline*}
 By symmetry, the triple integral on the right hand side of this identity does not depend on $x\in\Sp$.
 Its constant value is easily deduced by evaluating \eqref{fug5} at $t=0$ and yields
$$ P_s(B)=P(B)\, \int_{\pa B}\int_0^1\int_1^{+\infty}f_{|x-y|}(r,\r)\,dr\,d\r\,d\H_y\,,\qquad\forall x\in\Sp\,.$$
  Combining the last two identities with \eqref{fug5}, we conclude that
   \begin{multline*}
   P_s(F_t)=\frac12\iint_{\Sp\times\Sp}\biggl(\int_{1+tu(y)}^{1+tu(x)}\int_{1+tu(y)}^{1+tu(x)} f_{|x-y|}(r,\r)\,dr\,d\r \biggr)\,d\H_x\,d\H_y\\
   +\frac{P_s(B)}{P(B)}\int_{\Sp}(1+tu(x))^{n-s}\,d\H_x\,.
  \end{multline*}
  With a last change of variable in the first term on the right hand side of this identity, we reach the following formula for $P_s(F_t)$:
  \begin{equation}\label{fug6}
   P_s(F_t)=\frac{t^2}{2}\,g(t)+\frac{P_s(B)}{P(B)}\,h(t)\,,
  \end{equation}
  where we have set
$$
  g(t):=\iint_{\Sp\times\Sp}\biggl(\int_{u(y)}^{u(x)}\int_{u(y)}^{u(x)} f_{|x-y|}(1+tr,1+t\r)\,dr\,d\r \biggr)\,d\H_x\,d\H_y\,,
 $$
   and
$$
  h(t):=\int_{\Sp}(1+tu(x))^{n-s}\,d\H_x\,.
$$
  Since $g$ depends smoothly on $t$, we can find $\tau\in(0,t)$ such that $g(t)=g(0)+t\,g'(\tau)$.
 In addition, observing that
 \[
 \Big|r\,\frac{\pa f_\theta}{\pa r}(1+\tau\,r,1+\tau\,\r)+\r\,\frac{\pa f_\theta}{\pa\r}(1+\tau\,r,1+\tau\,\r)\Big|\le \frac{C(n)}{\theta^{n+s}}\,,\qquad\forall r,\r\in\Big(-\frac12,\frac12\Big)\,,
 \]
for a suitable dimensional constant $C(n)$ (whose value is allowed to change from line to line),  one can estimate
 \[
 |g'(\tau)|\le C(n)\iint_{\pa B\times\pa B}\frac{|u(x)-u(y)|^2}{|x-y|^{n+s}}\,d\H_x\,d\H_y=C(n)\,[u]_{\frac{1+s}{2}}^2\,.
 \]
 Taking into account that $g(0)=[u]_{\frac{1+s}{2}}^2$ and $h(0)=P(B)$, we then infer from \eqref{fug6} that
 \begin{equation}\label{fug9}
 P_s(F_t)-P_s(B)\ge\frac{t^2}2[u]_{\frac{1+s}{2}}^2+\frac{P_s(B)}{P(B)}\,\big(h(t)-h(0)\big)-C(n)\,t^3\,[u]_{\frac{1+s}{2}}^2\,.
 \end{equation}
 We now exploit the volume constraint $|F_t|=|B|$ to deduce that
 \[
 \int_{\Sp}(1+t\,u)^n\,d\H=n\,|F_{t}|=n\,|B|=P(B)=h(0)\,,
 \]
 so that
$$
 h(t)-h(0)=\int_{\Sp}(1+t\,u)^n\big((1+t\,u)^{-s}-1\big)\,d\H_x\,.
$$
By a Taylor expansion, we find that for every $|z|\leq1/2$,
$$
\big((1+z)^{-s}-1\big)(1+z)^n=\Bigl(-sz+\frac{s(s+1)}{2}z^2+sR_1(z)\Bigr)\Bigl(1+nz+\frac{n(n-1)}{2}z^2+R_2(z)\Bigr)\,,
$$
with $|R_1(z)|+|R_2(z)|\leq C(n)|z|^3$.
Thus
\begin{equation}\label{fug9.6}
h(t)-h(0)\geq-s\int_{\Sp}\Bigl[t\,u+\Bigl(n-\frac{s+1}{2}\Bigr)t^2\,u^2\Bigr]\,d\H-C(n)s\,t^3\|u\|^2_{L^2}\,.
\end{equation}
Exploiting the volume constraint again, i.e., $\int_{\Sp}\big((1+t\,u)^n-1\big)=0$, and expanding the term $(1+t\,u)^n$, we get
 \begin{equation}\label{fug3}
 -\int_{\Sp}t\,u\,d\H\ge\frac{(n-1)}{2}\int_{\Sp}t^2\,u^2\,d\H-C(n)\,t^3\,\|u\|^2_{L^2}\,.
 \end{equation}
 We may now combine \eqref{fug3} with \eqref{fug9.6} and \eqref{millot1} to obtain
 \begin{align*}
\frac{P_s(B)}{P(B)}\big(h(t)-h(0)\big)&\geq-\frac{t^2}{2}\,\frac{s(n-s)P_s(B)}{P(B)}\,\int_{\Sp}u^2\,d\H-C(n)\,\frac{s\,P_s(B)}{P(B)}\,t^3\|u\|^2_{L^2}\\
&=-\frac{t^2}{2}\,\l_1^s\,\int_{\Sp}u^2\,d\H-\frac{C(n)}{n-s}\,\l_1^s\,t^3\|u\|^2_{L^2}\,.
\end{align*}
We plug this last inequality into \eqref{fug9} to find that
\begin{eqnarray}\label{ober2}
  P_s(F_t)-P_s(B)
  &\ge&\frac{t^2}2\Big([u]_{\frac{1+s}{2}}^2-\l_1^s\,\|u\|_{L^2}^2\Big)
  -C(n)\,t^3\,\Big([u]_{\frac{1+s}{2}}^2+\l_1^s\,\|u\|_{L^2}^2\Big)\,.
\end{eqnarray}
Setting for brevity $a^i_k:=a^i_k(u)$, we now apply \eqref{fou1} to deduce that, for every $\eta\in(0,1)$,
\begin{align}\nonumber
[u]_{\frac{1+s}{2}}^2-\l_1^s\,\|u\|_{L^2}^2
&\geq \sum_{k=1}^{\infty}\sum_{i=1}^{d(k)}\lambda_k^s|a^i_k|^2-\lambda_1^s\sum_{k=0}^{\infty}\sum_{i=1}^{d(k)}|a^i_k|^2
\\
&=
\frac{1}{4}\,\sum_{k=2}^{\infty}\sum_{i=1}^{d(k)}\lambda_k^s|a^i_k|^2
+
\sum_{k=2}^{\infty}\sum_{i=1}^{d(k)}\bigg(\frac{3}{4}\lambda_k^s-\lambda_1^s\bigg)|a^i_k|^2-\lambda_1^s|a_0|^2
\nonumber
\\
&\ge\frac{1}{4}\,[u]_{\frac{1+s}{2}}^2
+
\sum_{k=2}^{\infty}\sum_{i=1}^{d(k)}\bigg(\frac{3}{4}\lambda_k^s-\lambda_1^s\bigg)|a^i_k|^2-\lambda_1^s\sum_{i=1}^n|a^i_1|^2-\lambda_1^s|a_0|^2
\nonumber\,.
\end{align}
Thanks to \eqref{ok} and \eqref{remeige1},
$\frac{3}{4}\lambda_k^s-\lambda_1^s\ge \l_1^s/2$ for every $k\ge 2$. Hence,
\begin{equation}\label{fug10}
[u]_{\frac{1+s}{2}}^2-\l_1^s\,\|u\|_{L^2}^2\ge
\frac{1}{4}\,[u]_{\frac{1+s}{2}}^2+
\l_1^s\bigg(\frac12\,\sum_{k=2}^{\infty}\sum_{i=1}^{d(k)}|a^i_k|^2-\sum_{i=1}^n|a^i_1|^2-|a_0|^2\bigg)\,.
\end{equation}
Using the volume constraint again and taking into account that $a_0=P(B)^{-1/2}\int_{\pa B}u\,$, one easily estimates for a suitably small value of $\e_0$,
\begin{equation}\label{fug11}
|a_0|\leq C(n)\,t\,\|u\|^2_{L^2}\,.
\end{equation}
Similarly,  the barycenter constraint $0=\int_{\Sp}x_i\,(1+t\,u)^{n+1}\,d\H$ yields
\[
\Big|\int_{\pa B}x_i\,u\,d\H\Big|\le C(n)\,t\,\|u\|_{L^2}^2\,,
\]
so that,  taking into account that $Y_1^i=c(n)\,x_i$ for some constant $c(n)$ depending on $n$ only,
\begin{equation}\label{fug11-bis}
|a_1^i|\leq C(n)\,t\|u\|_{2}^2\,,\qquad i=1,...,n\,.
\end{equation}
We can now combine \eqref{fug11} and \eqref{fug11-bis} with $\|u\|_{L^2}^2=\sum_{k=0}^\infty\sum_{i=1}^{d(k)}|a_k^i|^2$, to conclude that
\[
|a_0|^2+\sum_{i=1}^n|a^i_1|^2\le C(n)\,t\,\sum_{k=2}^\infty\sum_{i=1}^{d(k)}|a_k^i|^2\,.
\]
This last inequality implies of course that, for $\e_0$ small,
\begin{eqnarray}\label{ober3}
  \frac12\,\sum_{k=2}^{\infty}\sum_{i=1}^{d(k)}|a^i_k|^2-\sum_{i=1}^n|a^i_1|^2-|a_0|^2\ge
  \frac{\|u\|_{L^2}^2}{4}\,.
\end{eqnarray}
By \eqref{ober2}, \eqref{fug10}, and \eqref{ober3} we thus find
\begin{eqnarray*}
P_s(F_t)-P_s(B)&\ge& \frac{t^2}8\Big([u]_{\frac{1+s}{2}}^2+\l_1^s\,\|u\|_{L^2}^2\Big)-C(n)\,t^3\,\Big([u]_{\frac{1+s}{2}}^2+\l_1^s\,\|u\|_{L^2}^2\Big)
\\
&\ge&
\frac{t^2}{16}\,\Big([u]_{\frac{1+s}{2}}^2+\l_1^s\,\|u\|_{L^2}^2\Big)\,,
\end{eqnarray*}
provided $\e_0$, hence $t$, is small enough with respect to $n$. Since $\l_1^s\ge s\,P_s(B)$, we have completed the proof of \eqref{fug0star}, thus of Theorem \ref{fuglede}.
\end{proof}


\section{Uniform estimates for almost-minimizers of nonlocal perimeters}\label{section regularity}

A crucial step in our proof of Theorem \ref{thm 1} and Theorem \ref{thm 2} is the application of the regularity theory for nonlocal perimeter minimizers: indeed, this is the step where we reduce to consider small normal deformations of balls, and thus become able to apply Theorem \ref{fuglede}. The parts of the regularity theory for nonlocal perimeter minimizers that are relevant to us have been developed in \cite{CRS,CG} with the parameter $s$ fixed. In other words, there is no explicit discussion on how the regularity estimates should behave as $s$ approaches the limit values $0$ or $1$, although it is pretty clear \cite{CV,ADPM,DFPV} that they should degenerate when $s\to 0^+$, and that they should be stable, after scaling $s$-perimeter by the factor $(1-s)$, in the limit $s\to 1^-$. Since we shall need to exploit these natural uniformity properties, in this section we explain how to deduce these results from the results contained in \cite{CRS,CG}, with the aim of proving Corollary \ref{corollary cruciale} below. In order to minimize the amount of technicalities, we shall discuss these issues working with a rather special notion of almost-minimality, that we now introduce. It goes without saying, the results we present should hold true in the more general class of almost-minimizers considered in \cite{CG}.

We thus introduce the special class of almost-minimizers we shall consider. Given $\Lambda\ge0$, $s\in(0,1)$, and a {\it bounded} Borel set $E\subset\R^n$, we say that $E$ is a (global) $\Lambda$-minimizer of the $s$-perimeter if
\begin{equation}
  \label{austin1}
  P_s(E)\le P_s(F)+\frac{\Lambda}{1-s}\,|E\Delta F|\,,
\end{equation}
for every {\it bounded} set $F\subset\R^n$. Since the validity of \eqref{austin1} is not affected if we replace $E$ with some $E'$ with $|E\Delta E'|=0$, we shall always assume that a $\Lambda$-minimizer of the $s$-perimeter has been normalized so that
\begin{equation}
  \label{austin2}
  \mbox{$E$ is Borel, with}\,\, \pa E=\Big\{x\in\R^n:\mbox{$0<|E\cap B(x,r)|<\om_n\,r^n$ for every $r>0$}\Big\}
\end{equation}
(as show for instance in  \cite[Proposition 12.19, step two]{Ma}, this can always be done). As explained, we shall need some regularity estimates for $\Lambda$-minimizers of the $s$-perimeter to be uniform with respect to $s\in[s_0,1]$, for $s_0\in(0,1)$ fixed. We start with the following uniform density estimates. (The proof is classical, compare with \cite[Theorem 21.11]{Ma} for the local case, and with \cite[Theorem 4.1]{CRS} for the nonlocal case, but we give the details here in order to keep track of the constants.)

\begin{lemma}\label{lemma density}
  If $s\in(0,1)$, $\Lambda\ge0$, and $E$ satisfies the minimality property \eqref{austin1} and the normalization condition \eqref{austin2}, then we have
  \begin{equation}
    \label{density estimate}
    |B|\,(1-c_0)\,r^n\ge |E\cap B(x_0,r)|\ge|B|\,c_0\,r^n\,,
  \end{equation}
  whenever $x_0\in\pa E$ and $r\le r_0$, where
  \[
  c_0=\Big(\frac{s}{8\,|B|\,2^{n/s}}\frac{(1-s)P_s(B)}{P(B)}\Big)^{n/s}\,,\qquad r_0=\Big(\frac{(1-s)\,P_s(B)}{2\,\Lambda\,|B|}\Big)^{1/s}\,.
  \]
\end{lemma}

The following elementary lemma (De Giorgi iteration) is needed in the proof.

\begin{lemma}\label{lemma degiorgi}
  Let $\a\in(0,1)$, $N>1$, $M>0$, and $\{u_k\}_{k\in\N}$ be a decreasing sequence of positive numbers such that
  \begin{equation}
    \label{degiorgi}
    u_{k+1}^{1-\a}\le N^k\,M\,u_k\,,\qquad\forall k\in\N\,.
  \end{equation}
  If
  \begin{equation}
    \label{degiorgi small}
    u_0\leq\frac1{N^{(1-\a)/\a^2}\,M^{1/\a}}\,,
  \end{equation}
  then $u_k\to0$ as $k\to\infty$.
\end{lemma}

\begin{proof}[Proof of Lemma \ref{lemma degiorgi}]
  By \eqref{degiorgi} and \eqref{degiorgi small}, induction proves that $u_k\le N^{-k/\a}\,u_0$ for every $k\in\N$.
\end{proof}


\begin{proof}[Proof of Lemma \ref{lemma density}]
Being the two proofs analogous, we only prove the lower bound in \eqref{density estimate}. Up to a translation we may also assume that $x_0=0$. We fix $r>0$, set $u(r):=|E\cap B_r|$, and apply \eqref{austin1} with $F=E\setminus B_r$ to find
  $$
  (1-s)\int_E\int_{E^c}\frac{dx\,dy}{|x-y|^{n+s}}\le(1-s)\int_{E\setminus B_r}\int_{E^c\cup (E\cap B_r)}\frac{dx\,dy}{|x-y|^{n+s}}+\Lambda u(r)\,,
  $$
   As a consequence
   $$
   (1-s)\int_{E\cap B_r}\int_{E^c}\frac{dx\,dy}{|x-y|^{n+s}}\le (1-s)\int_{E\setminus B_r}\int_{E\cap B_r}\frac{dx\,dy}{|x-y|^{n+s}}+\Lambda u(r)\,,
   $$ hence, by adding up $(1-s)\int_{E\setminus B_r}\int_{E\cap B_r}\frac{dx\,dy}{|x-y|^{n+s}}$ to both sides we immediately get, for every $r>0$,
  \begin{equation}
    \label{austin3}
      P_s(E\cap B_r)\le 2\,\int_{E\setminus B_r}\int_{E\cap B_r}\frac{dx\,dy}{|x-y|^{n+s}}+\frac{\Lambda}{1-s}\,u(r)\,.
  \end{equation}
  On the one hand,  $P_s(E\cap B_r)\ge P_s(B)\,(u(r)/|B|)^{(n-s)/n}$ by the isoperimetric inequality \eqref{nonlocal isoperimetric inequality}; on the other hand, by the coarea formula
  \begin{eqnarray}\nonumber
    \int_{E\setminus B_r}\int_{E\cap B_r}\frac{dx\,dy}{|x-y|^{n+s}}&\le& \int_{E\cap B_r}dx\int_{B(x,r-|x|)^c}\frac{dy}{|x-y|^{n+s}}
    \\\label{caffa}
    &=&
    \frac{P(B)}{s}\int_{E\cap B_r}\frac{dx}{(r-|x|)^s}=\frac{P(B)}s\,\int_0^r\frac{u'(t)}{(r-t)^s}\,dt\,,
  \end{eqnarray}
  where we have also taken into account that $u'(t)=\H(E\cap\pa B_t)$ for a.e. $t>0$. By combining these two facts with \eqref{austin3} we find
  \begin{equation}
    \label{austin4}
  \frac{P_s(B)}{|B|^{(n-s)/n}}\,u(r)^{(n-s)/n}\le \frac{2\,P(B)}{s}\int_0^r\,\frac{u'(t)}{(r-t)^s}\,dt+\frac{\Lambda}{1-s}\,u(r)\,,\qquad\forall r>0\,.
  \end{equation}
  Since $u(r)\le |B|\,r^n$ for every $r>0$, our choice of $r_0$ implies that
  \[
  \frac{\Lambda}{1-s}\,u(r)\le\frac{P_s(B)}{2\,|B|^{(n-s)/n}}\,u(r)^{(n-s)/n}\,,\qquad\forall r\le r_0\,,
  \]
  and enables us to deduce from \eqref{austin4} that
  \begin{equation}
    \label{austin5}
    u(r)^{(n-s)/n}\le \frac{4\,P(B)\,|B|^{(n-s)/n}}{s\,P_s(B)}\int_0^r\,\frac{u'(t)}{(r-t)^s}\,dt\,,\qquad\forall r\le r_0\,.
  \end{equation}
  By integrating \eqref{austin5} on $(0,\ell)\subset(0,r_0)$ and by Fubini's theorem, we thus obtain
  \begin{equation}
    \label{austin6}
    \int_0^\ell\,u(r)^{(n-s)/n}\,dr\le \frac{4\,P(B)\,|B|^{(n-s)/n}}{s\,(1-s)\,P_s(B)}\,\ell^{1-s}\,u(\ell),\qquad\forall \ell\le r_0\,.
  \end{equation}
  We now argue by contradiction, and assume the existence of $\ell_0\le r_0$ such that $u(\ell_0)\leq c_0\,|B|\,\ell_0^n$. Correspondingly we set
  \[
  \ell_k:=\frac{\ell_0}2+\frac{\ell_0}{2^{k+1}}\,,\qquad u_k:=u(\ell_k)\,,\qquad C_1:=\frac{4\,P(B)\,|B|^{(n-s)/n}}{s\,(1-s)\,P_s(B)}\,,
  \]
  and notice that \eqref{austin6} implies
  \[
  \frac{\ell_0}{2^{k+2}}\,u_{k+1}^{(n-s)/n}= (\ell_k-\ell_{k+1})u_{k+1}^{(n-s)/n}\le \int_{\ell_{k+1}}^{\ell_k}\,u^{(n-s)/n}\le C_1\,\ell_k^{1-s}\,u_k\le C_1\ell_0^{1-s}\,u_k\,,
  \]
  that is, $u_{k+1}^{1-\a}\le 2^k\,M\,u_k$ for $M:=4\,C_1\,\ell_0^{-s}$ and $\a=s/n$. Since $u_k\to u(\ell_0/2)=|E\cap B_{\ell_0/2}|>0$ (indeed, $0\in\pa E$ and \eqref{austin2} is in force), by Lemma \ref{lemma degiorgi} we deduce that
  \[
  u(\ell_0)=u_0>\frac1{2^{(1-\a)/\a^2}\,M^{1/\a}}=\frac{2^{n/s}\,\ell_0^n}{2^{(n/s)^2}\,(4\,C_1)^{n/s}}=c_0\,|B|\,\ell_0^n\,.
  \]
  However, this is a contradiction to $u(\ell_0)\leq c_0\,|B|\,\ell_0^n$, and the lemma is proved.
\end{proof}

Introducing a further bit of special terminology, we say that a bounded Borel set $E\subset\R^n$ is a $\Lambda$-minimizer of the $1$-perimeter if
\[
P(E)\le P(F)+\frac{\Lambda}{\om_{n-1}}\,|E\Delta F|\,,
\]
for every bounded $F\subset\R^n$, and if \eqref{austin2} holds true. We have the following compactness theorem.

\begin{theorem}\label{thm compactness}
  If $R>0$, $s_0\in(0,1)$, and $E_h$ ($h\in\N$) is a $\Lambda$-minimizer of the $s_h$-perimeter with $s_h\in[s_0,1)$ and $E_h\subset B_R$ for every $h\in\N$, then there exist $s_*\in[s_0,1]$ and a $\Lambda$-minimizer of the $s_*$-perimeter $E$ such that, up to extracting subsequences, $s_h\to s_*$, $|E_h\Delta E|\to 0$ and $\pa E_h$ converges to $\pa E$ in Hausdorff distance as $h\to\infty$.
\end{theorem}

\begin{proof} Up to extracting subsequences we may obviously assume that $s_h\to s_*$ as $h\to\infty$, where $s_*\in[s_0,1]$. By exploiting \eqref{austin1} with $F=B_R$ we see that
\begin{equation}
  \label{compatt austin}
  \sup_{h\in\N}(1-s_h)\,P_{s_h}(E_h)\le 2\Lambda\,|B_R|+\sup_{h\in\N}(1-s_h)\,P_{s_h}(B_R)<\infty\,,
\end{equation}
where we have used the fact that $(1-s)\,P_s(B)\to \om_{n-1}P(B)$ as $s\to 1^+$
(recall \eqref{limit s to 1}).

\medskip

\noindent {\it Step one:} We prove the theorem in the case $s_*=1$. By \eqref{compatt austin} and by \cite[Theorem 1]{ADPM}, we find that, up to extracting subsequences, $|E_h\Delta E|\to 0$ as $h\to\infty$ for some set $E\subset B_R$ with finite perimeter. By \cite[Theorem 2]{ADPM},
\[
\om_{n-1}\,P(E)\le\liminf_{h\to\infty}(1-s_h)\,P_{s_h}(E_h)\,,
\]
and, if $F\subset\R^n$ is bounded, then we can find bounded set $F_h$ ($h\in\N$) such that $|F_h\Delta F|\to 0$ as $h\to\infty$ and
\[
\om_{n-1}\,P(F)=\liminf_{h\to\infty}(1-s_h)\,P_{s_h}(F_h)\,.
\]
By \eqref{austin1}, $(1-s_h)\,P_{s_h}(E_h)\le(1-s_h)\,P_{s_h}(F_h)+\Lambda\,|E_h\Delta F_h|$; by letting $h\to\infty$, we find that $E$ is a $\Lambda$-minimizer of the $1$-perimeter. The fact that $\pa E_h$ converges to $\pa E$ in Hausdorff distance as $h\to\infty$ is now a standard consequences of the uniform density estimates proved in Lemma \ref{lemma density}.

\medskip

\noindent {\it Step two:} We address the case $s_*<1$. In this case we may notice that \eqref{compatt austin} together with the assumption that $E_h\subset B_R$ allows us to say that $\{P_s(E_h)\}_{h\in\N}$ is bounded in $\R$ for some $s\in(0,1)$. By compactness of the embedding of $H^{s/2}$ in $L^1_{loc}$ and by the assumption $E_h\subset B_R$ we find a set $E\subset B_R$ such that, up to extracting subsequences, $|E_h\Delta E|\to 0$ as $h\to\infty$. If we pick any bounded set $F\subset\R^n$, then by Appendix \ref{appendix gamma} there exists a sequence of bounded sets $\{F_h\}_{h\in\N}$ such that
\begin{equation}
  \label{gamma}
\lim_{h\to\infty}|F_h\Delta F|=0\,,\qquad\limsup_{h\to\infty}P_{s_h}(F_h)\le P_{s_*}(F)\,.
\end{equation}
By applying \eqref{austin1} to $E_h$ and $F_h$, and then by letting $h\to\infty$, we find that
\[
P_{s_*}(E)\le\liminf_{h\to\infty}P_{s_h}(E_h)\le \limsup_{h\to\infty}P_{s_h}(F)+\frac{\Lambda}{1-s_h}\,|E_h\Delta F_h|
\le P_{s_*}(F)+\frac{\Lambda}{1-s_*}\,|E\Delta F|\,,
\]
where the first inequality follows by Fatou's lemma, and the last one by \eqref{gamma}. Since the Hausdorff convergence of $\pa E_h$ to $\pa E$ is again consequence of Lemma \ref{lemma density}, the proof is complete.
\end{proof}

The next result is a uniform (with respect to $s$) version of the classical ``improvement of flatness'' statement.

\begin{theorem}\label{thm flatness}
  Given $n\ge 2$, $\Lambda\ge0$, and $s_0\in(0,1)$, there exist $\tau,\eta,q\in (0,1)$, depending on $n$, $\Lambda$ and $s_0$ only, with the following property: If $E$ is a $\Lambda$-minimizer of the $s$-perimeter for some $s\in[s_0,1]$ with $0\in\pa E$ and
  \[
  B\cap\pa E\subset\Big\{y\in\R^n:|(y-x)\cdot e|<\tau\Big\}
  \]
  for some $e\in S^{n-1}$, then there exists $e_0\in S^{n-1}$ such that
  \[
  B_{\eta}\cap\pa E\subset\Big\{y\in\R^n:|(y-x)\cdot e_0|<q\,\tau\,\eta\Big\}\,.
  \]
\end{theorem}

\begin{proof}
  {\it Step one:} We prove that if $\bar{s}\in(0,1]$, then there exist $\de>0$ and $\bar \tau,\bar\eta,\bar q\in (0,1)$  (depending on $n$, $\bar s$ and $\Lambda$ only), such that if $s\in(\bar{s}-\de,\bar{s}+\de)\cap(0,1]$ and $E$ is a $\Lambda$-minimizer of the $s$-perimeter with $0\in\pa E$ and
  \[
  B\cap\pa E\subset\Big\{y\in\R^n:|(y-x)\cdot e|<\bar\tau\Big\}
  \]
  for some $e\in S^{n-1}$, then there exists $e_0\in S^{n-1}$ such that
  \[
  B_{\bar\eta}\cap\pa E\subset\Big\{y\in\R^n:|(y-x)\cdot e_0|<\bar q\,\bar\tau\,\bar\eta\Big\}\,.
  \]
  Indeed, it follows from \cite[Theorems 24.1 and 26.3]{Ma} in the case $\bar{s}=1$, and from \cite[Theorem 1.1]{CG} if $\bar{s}<1$, that there exist $\bar\tau,\bar\eta,\bar q\in (0,1/2)$ (depending on $n$, $\bar s$ and $\Lambda$ only) such that if $F$ is a $\Lambda$-minimizer of the $\bar s$-perimeter with
  \begin{equation}
    \label{sbar1}
     0\in\pa F\,,\qquad  B\cap\pa F\subset\Big\{y\in\R^n:|(y-x)\cdot e|<2\,\bar\tau\Big\}
  \end{equation}
  for some $e\in S^{n-1}$, then there exists $e_0\in S^{n-1}$ such that
  \begin{equation}
  \label{sbar2}
  B_{\bar\eta}\cap\pa F\subset\Big\{y\in\R^n:|(y-x)\cdot e_0|<\frac{\bar q}4\,(2\,\bar\tau)\,\bar\eta\,\Big\}\,.
  \end{equation}
  Let us now assume by contradiction that our claim is false. Then we can find a sequence $s_h\to\bar s$ as $h\to\infty$, and, for every $h\in\N$, $E_h$ $\Lambda$-minimizer of the $s_h$-perimeter such that, for some $e_h\in S^{n-1}$,
  \begin{equation}\label{sbar4}
  0\in\pa E_h\,,\qquad B\cap\pa E_h\subset\Big\{y\in\R^n:|(y-x)\cdot e_h|<\bar{\tau}\Big\}\,,\qquad\forall h\in\N\,,
  \end{equation}
  but
  \begin{equation}
    \label{sbar3}
      B_{\bar\eta}\cap\pa E_h\not\subset\Big\{y\in\R^n:|(y-x)\cdot e_0|<\bar q\,\bar\tau\,\bar{\eta}\Big\}\,,\qquad\forall h\in\N\,,\forall e_0\in S^{n-1}\,.
  \end{equation}
  By the compactness theorem, there exists a $\Lambda$-minimizer of the $\bar{s}$-perimeter $F$ such that $\pa E_h$ converges to $\pa F$ with respect to the Hausdorff distance on compact sets. By the latter information we have $0\in\pa F$, and we find from \eqref{sbar4} that $F$ is a $\Lambda$-minimizer of the $\bar{s}$-perimeter such that \eqref{sbar1} holds true. In particular, there exists $e_0\in S^{n-1}$ such that \eqref{sbar2} holds true. By exploiting the local Hausdorff convergence of $\pa E_h$ to $\pa F$ one more time, we thus find that, if $h$ is large enough, then
  \[
  B_{\bar\eta}\cap\pa E_h\subset\Big\{y\in\R^n:|(y-x)\cdot e_0|<\bar q\,\bar\tau\,\bar\eta\Big\}\,,
  \]
  a contradiction to \eqref{sbar3}. We have completed the proof of step one.

  \medskip

  \noindent {\it Step two:} We complete the proof of the theorem by covering $[s_0,1]$ with a finite number of intervals $(\bar{s}_i-\de_i,\bar{s}_i+\de_i)$ of the form constructed in step one.
\end{proof}

Improvement of flatness implies $C^{1,\a}$-regularity by a standard argument. By exploiting the uniformity of the constants obtained in Theorem \ref{thm flatness} one thus gets the following uniform regularity criterion.

\begin{corollary}
  If $n\ge 2$, $\Lambda\ge0$ and $s_0\in(0,1)$, then there exist positive constants $\e_0<1$, $C_0>0$, and $\a<1$, depending on $n$, $\Lambda$ and $s_0$ only, with the following property: If $E$ is a $\Lambda$-minimizer of the $s$-perimeter for some $s\in[s_0,1)$ and
  \begin{equation}\label{regularity criterion}
  0\in\pa E\,,\qquad B\cap\pa E\subset\Big\{y\in\R^n:|(y-x)\cdot e|<\e_0\Big\}
  \end{equation}
  for some $e\in S^{n-1}$, then $B_{1/2}\cap\pa E$ is the graph of a function with $C^{1,\a}$-norm bounded by $C_0$.
\end{corollary}

Finally, by  Hausdorff convergence of sequences of minimizers, we can exploit the regularity criterion \eqref{regularity criterion} and the smoothness of the limit set $B$ via a standard argument (see, e.g., \cite[Theorem 26.6]{Ma}) in order to obtain the following result, that plays a crucial role in the proof of our main results.

\begin{corollary}\label{corollary cruciale}
  If $n\ge 2$, $\Lambda\ge0$, $s_0\in(0,1)$, $E_h$ ($h\in\N$) is a $\Lambda$-minimizer of the $s_h$-perimeter for some $s_h\in[s_0,1)$, and $E_h$ converges in measure to $B$, then there exists a bounded sequence $\{u_h\}_{h\in\N}\subset C^{1,\a}(\pa B)$ (for some $\a\in(0,1)$ independent of $h$) such that
  $$
    \pa E_h=\Big\{(1+u_h(x))x:x\in\pa B\Big\}\,,\qquad\lim_{h\to\infty}\|u_h\|_{C^1(\pa B)}=0\,.
$$
\end{corollary}


\section{Proof of Theorem \ref{thm 1}}\label{section stability}
Given $s\in(0,1]$, we introduce the {\it fractional isoperimetric gap} of $E\subset\R^n$ (with $0<|E|<\infty$)
$$
 D_s(E):=\frac{P_s(E)}{P_s(B_{r_E})}-1\,,
$$
 where $r_E=(|E|/|B|)^{1/n}$ and $P_1(E)=P(E)$ denotes the distributional perimeter of $E$. We shall also set
 \[
 \de_{s_0}(E):=\inf_{s_0\le s< 1}D_s(E)\,.
 \]
 With this notation at hand, the quantitative isoperimetric inequality \eqref{nonlocal isoperimetric inequality quantitative} takes the form
 \begin{equation}\label{qip1}
 A(E)^2\leq C(n,s_0)\,\de_{s_0}(E)\,.
 \end{equation}
 We begin by noticing that we can easily obtain \eqref{qip1} in the case of nearly spherical sets as a consequence of Theorem \ref{fuglede}.

 \begin{remark}\label{remark costanti}
   Starting from Corollary \ref{corollary fuglede}, we shall coherently enumerate the constants appearing in the various statements of this section. For example, thorough this section, the symbol $C_0$ will always denote the constant appearing in \eqref{qip1.1}. No confusion will arise as we shall not need to refer to constants defined in other sections of the paper. Symbols like $C(n,s)$ shall be used to denote generic constants (depending on $n$ and $s$ only) whose precise value shall be inessential to us.
 \end{remark}

 \begin{corollary}\label{corollary fuglede}
 For every $n\ge 2$ there exist positive constants $C_0(n)$ and $\e_0(n)$ such that
  \begin{equation}\label{qip1.1}
  \frac{C_0(n)}s\,D_s(E)\ge A(E)^2
  \end{equation}
  whenever $s\in(0,1)$ and $E$ is a nearly spherical set as in \eqref{nearly}, with $|E|=|B|$, $\int_Exdx=0$, and $\|u\|_{C^1(\Sp)}\leq\e_0(n)$. In particular, under these assumptions on $E$, we have that
  \begin{equation}\label{qip1.1-star}
  \frac{C_0(n)}{s_0}\,\de_{s_0}(E)\ge A(E)^2\,,\qquad\forall s_0\in(0,1)\,.
  \end{equation}
 \end{corollary}

 \begin{proof}
  This follows immediately by \eqref{fug0} since
  $$
  A(E)\le C(n)\int_{\pa B}|u|\,d\H
  \leq C(n)\sqrt{\int_{\pa B}|u|^2\,d\H}\,.
  $$
 \end{proof}

 The proof of Theorem \ref{thm 1} is thus based on a reduction argument to the case considered in Corollary \ref{corollary fuglede}, much as in the spirit of what done \cite{CL} in the case $s=1$. To this end, we argue by contradiction and assume \eqref{qip1} to fail. This gives us a sequence $\{E_h\}_{h\in\N}$ of almost-isoperimetric sets (that is, $D_{s_h}(E_h)\to 0$ as $h\to\infty$ for some $s_h\in[s_0,1)$) with $|E_h|=|B|$ such that $D_{s_h}(E_h)< M\,A(E_h)^2$, for a constant $M$ as large as we want. By Lemma \ref{lemma soft} below, the first information allows us to deduce that, up to translations, $|E_h\Delta B|\to 0$ as $h\to\infty$. We next ``round-up'' our sets $E_h$ by solving a penalized isoperimetric problem, see Lemma \ref{lemma code}, to obtain a new sequence $\{F_h\}_{h\in\N}$ -- with the same properties of $\{E_h\}_{h\in\N}$ concerning isoperimetric gaps and asymmetry -- but with the additional feature of being nearly spherical sets associated to functions $\{u_h\}_{h\in\N}\subset C^1(\pa B)$ with $\|u_h\|_{C^1(\pa B)}\to 0$ as $h\to\infty$. By \eqref{qip1.1-star} this means that $C_0(n)/s_0\ge M$, which gives a contradiction if we started the argument with $M$ large enough.

 In order to make this argument rigorous we need to premise a series of remarks that seem interesting in their own. The first one is a nucleation lemma for nonlocal perimeters in the spirit of \cite[VI.13]{A}, see also \cite[Lemma 29.10]{Ma}. Here, $E^{(1)}$ stands for the set of points of density $1$ of a measurable set $E$.

 \begin{lemma}\label{nucleation lemma}
   If $n\ge 2$, $s\in(0,1)$, $P_s(E)<\infty$, and $0<|E|<\infty$, then there exists $x\in E^{(1)}$ such that
   \begin{equation}
    \label{rutgers tesi}
      |E\cap B(x,1)|\ge\min\Big\{\frac{\chi_1\,|E|}{(1-s)\,P_s(E)},\frac{1}{\chi_2}\Big\}^{n/s}\,,
  \end{equation}
  where
  $$
    \chi_1(n,s):=\frac{(1-s)\,P_s(B)}{4\,|B|^{(n-s)/n}\,\xi(n)}\,,
$$
$$
    \chi_2(n,s):=\frac{2^{3+(n/s)}\,|B|^{(n-s)/n}\,P(B)}{s(1-s)\,P_s(B)}\,,
$$
  and where $\xi(n)$ is Besicovitch's covering constant (see for instance \cite[Theorem 5.1]{Ma}). In particular, $0<\inf\{\chi_1(n,s),\chi_2(n,s)^{-1}:s\in[s_0,1)\}<\infty$ for every $s_0\in(0,1)$.
 \end{lemma}

 \begin{proof}
   {\it Step one:} We show that if $x\in E^{(1)}$ with
   \begin{equation}
     \label{rutgers1}
     |E\cap B(x,1)|\le\Big(\frac{(1-s)\,P_s(B)}{2\,|B|^{(n-s)/n}\,\a} \Big)^{n/s}
   \end{equation}
   for some $\a$ satisfying
   \begin{equation}
     \label{alpha}
     \a\geq\frac{2^{2+(n/s)}\,P(B)}s\,,
   \end{equation}
   then there exists $r_x\in(0,1]$ such that
   \begin{equation}\label{rutgers2}
   |E\cap B(x,r_x)|\le\frac{(1-s)}\a\,\int_{E\cap B(x,r_x)}\int_{E^c}\frac{dz\,dy}{|z-y|^{n+s}}\,.
   \end{equation}
   Indeed, if not, setting for brevity $u(r):=|E\cap B(x,r)|$ we have $(1-s)\int_{E\cap B(x,r)}\int_{E^c}\frac{dz\,dy}{|z-y|^{n+s}}\le \a\,u(r)$ for every $r\le1$. By adding up $(1-s)\int_{E\setminus B(x,r)}\int_{E\cap B(x,r)}\frac{dz\,dy}{|z-y|^{n+s}}$ to both sides, we get
 $$
      P_s(E\cap B(x,r))\le \int_{E\setminus B(x,r)}\int_{E\cap B(x,r)}\frac{dz\,dy}{|z-y|^{n+s}}+\frac{\a}{1-s}\,u(r)
$$
  for every $r\le1$ so that, arguing as in the proof of Lemma \ref{lemma density}, we get
  \begin{equation}
    \label{austin4x}
  \frac{P_s(B)}{|B|^{(n-s)/n}}\,u(r)^{(n-s)/n}\le \frac{P(B)}{s}\int_0^r\,\frac{u'(t)}{(r-t)^s}\,dt+\frac{\a}{1-s}\,u(r)\,,\qquad\forall r\le1\,,
  \end{equation}
  cf. with \eqref{austin4}. By \eqref{rutgers1} we have
  \[
  \frac{\a}{1-s}\,u(r)\le\frac\a{1-s}\,u(1)^{s/n}\,u(r)^{(n-s)/n}\le \frac{P_s(B)}{2\,|B|^{(n-s)/n}}\,u(r)^{(n-s)/n}\,,
  \]
  so that \eqref{austin4x} gives
  \begin{equation}
    \label{austin4xx}
  u(r)^{(n-s)/n}\le \frac{2\,P(B)\,|B|^{(n-s)/n}}{s\,P_s(B)}\int_0^r\,\frac{u'(t)}{(r-t)^s}\,dt\,,\qquad\forall r\le1\,.
  \end{equation}
  Notice that \eqref{austin4xx} implies \eqref{austin5} with $1$ in place of $r_0$. Moreover, \eqref{alpha} implies that $u(1)\leq c_0|B|$, where
  \[
  c_0=\Big(\frac{s}{8\,|B|\,2^{n/s}}\frac{(1-s)P_s(B)}{P(B)}\Big)^{n/s}\,,
  \]
  is the constant defined in Lemma \ref{lemma density}. Therefore, by repeating the very same iteration argument seen in the proof of Lemma \ref{lemma density} (notice that $u(r)>0$ for every $r>0$ since $x\in E^{(1)}$), we see that $u(1)> c_0|B|$, and thus find a contradiction. This completes the proof of step one.

  \medskip

  \noindent {\it Step two:} We complete the proof of the lemma. We argue by contradiction, and assume that for every $x\in E^{(1)}$ we have
  \begin{equation}
    \label{rutgers assurdo}
      |E\cap B(x,1)|\le\min\Big\{\frac{\chi_1\,|E|}{(1-s)\,P_s(E)},\frac{1}{\chi_2}\Big\}^{n/s}\,.
  \end{equation}
  If we set
  \begin{equation}
    \label{alpha2}
  \a:=\frac{(1-s)\,P_s(B)}{2\,|B|^{(n-s)/n}}\,\min\Big\{\frac{\chi_1\,|E|}{(1-s)\,P_s(E)},\frac{1}{\chi_2}\Big\}^{-1}\,,
  \end{equation}
  then \eqref{rutgers assurdo} takes the form of \eqref{rutgers1} for a value of $\a$ that (by definition of $\chi_2$) satisfies \eqref{alpha}. Hence, by step one, for every $x\in E^{(1)}$ there exists $r_x\in(0,1]$ such that \eqref{rutgers2} holds true with $\a$ as in \eqref{alpha2}. By applying Besicovitch covering theorem, see \cite[Corollary 5.2]{Ma}, we find a countable disjoint family of balls $\{B(x_h,r_h)\}_{h\in\N}$ such that $x_h\in E^{(1)}$, $r_h=r_{x_h}$ is such that \eqref{rutgers2} holds true with $x=x_h$, and thus
  \begin{eqnarray*}
    |E|&\le&\xi(n)\sum_{h\in\N}|E\cap B(x_h,r_h)|
    \le\frac{\xi(n)(1-s)}\a \sum_{h\in\N}\int_{E\cap B(x_h,r_h)}\int_{E^c}\frac{dz\,dy}{|z-y|^{n+s}}
    \\
    &\le&\frac{\xi(n)(1-s)P_s(E)}\a\leq
    \chi_1\,\frac{\xi(n)\,2\,|B|^{(n-s)/n}}{(1-s)\,P_s(B)}\,|E|=\frac{|E|}2\,,
  \end{eqnarray*}
  by definition of $\chi_1$. This is a contradiction, and the lemma is proved.
 \end{proof}

 Next, we prove the following ``soft'' stability lemma. An analogous statement was proved in \cite[Lemma 3.1]{FMM} in the case one works with $D_{s_0}(E)$ in place of $\de_{s_0}(E)$, and under the additional assumption that $A(E)\le 3/2$. This last assumption was not a real restriction in \cite{FMM}, as the soft stability lemma was applied to sets enjoying certain symmetry properties that, in turn, were granting that $A(E)\le 3/2$.  We avoid here the use of symmetrization arguments by exploiting the more general tool provided us by the nucleation lemma, Lemma \ref{nucleation lemma}.

 \begin{lemma}\label{lemma soft}
   If $n\ge 2$ and $s_0\in(0,1)$, then for every $\e>0$ there exists $\de>0$ (depending on $n$, $s_0$, and $\e$) such that if $\de_{s_0}(E)<\de$ then $A(E)<\e$.
 \end{lemma}

 \begin{proof} By contradiction, we assume the existence of a sequence of sets $E_h\subset\R^n$, $h\in\N$, such that
 \begin{equation}
   \label{soft0}
 |E_h|=|B|\,,\qquad A(E_h)\ge\e\,,\qquad\lim_{h\to\infty}\de_{s_0}(E_h)=0\,,
 \end{equation}
 where $\e$ is a positive constant. In particular there exist $s_h\in[s_0,1)$, $h\in\N$, such that
 \begin{equation}
   \label{soft1}
    \lim_{h\to\infty}\frac{P_{s_h}(E_h)}{P_{s_h}(B)}=1\,.
 \end{equation}
 Without loss of generality, we assume that $s_h\to s_*\in[s_0,1]$ as $h\to\infty$. Since $(1-s)\,P_s(B)\to \om_{n-1}\,P(B)$ as $s\to 1^-$, we find that
 \begin{equation}
   \label{soft2}
    \sup_{h\in\N}(1-s_h)\,P_{s_h}(E_h)<\infty\,.
 \end{equation}
 By Lemma \ref{nucleation lemma}, see \eqref{rutgers tesi}, we find that, up to translations,
 \begin{equation}
    \label{rutgers tesih}
      |E_h\cap B|\ge\min\Big\{\frac{\chi_1(n,s_h)\,|B|}{(1-s_h)\,P_{s_h}(E_h)},\frac{1}{\chi_2(n,s_h)}\Big\}^{n/s_h}\ge \kappa_*\,,
 \end{equation}
 for some positive constant $\kappa_*$ independent of $h$. By compactness of the embedding of $H^{s/2}(\R^n)$ into $L^1_{loc}(\R^n)$ when $s_*<1$, or by \cite[Theorem 1]{ADPM} in case $s_*=1$, we exploit \eqref{soft2} to deduce that, up to extracting subsequences, there exists a measurable set $E$ such that for every $K\subset\subset \R^n$ we have $|(E_h\Delta E)\cap K|\to 0$ as $h\to\infty$. By local convergence of $E_h$ to $E$ and by \eqref{soft0}, we have $|E|\le |B|$. If $s_*=1$, then by \cite[Theorem 2]{ADPM} and by \eqref{soft1} we find
 \[
 \om_{n-1}\,P(E)\le\liminf_{h\to\infty}(1-s_h)\,P_{s_h}(E_h)=\liminf_{h\to\infty}(1-s_h)\,P_{s_h}(B)=\om_{n-1}\,P(B)\,,
 \]
 that is, $P(E)\le P(B)$. If, instead, $s_*<1$, then \eqref{soft1} gives
 \[
 P_{s_*}(B)=\lim_{h\to\infty}P_{s_h}(E_h)=\lim_{h\to\infty}\int_{\R^n}\int_{\R^n}\frac{1_{E_h}(x)1_{E_h^c}(y)}{|x-y|^{n+s_h}}\,dxdy\ge P_{s_*}(E)\,,
 \]
 where the last inequality follows by Fatou's lemma. In both cases, $P_{s_*}(E)\le P_{s_*}(B)$. Should it be $|E|=|B|$, then, by the (nonlocal, if $s_*<1$) isoperimetric theorem, we would be able to conclude that $A(E)=0$, against $A(E_h)\ge\e$ for every $h\in\N$. Should it be $|E|=0$, then we would get a contradiction with \eqref{rutgers tesih}. Therefore, it must be $0<|E|<|B|$. By a standard application of the concentration-compactness lemma (see, e.g., \cite[Lemma 3.1]{FMM}), $0<|E|<|B|$ can happen only if there exists $\l\in(0,1)$ such that for every $\s>0$ and $h$ large enough there exist $F_h^\s, G_h^\s\subset E_h$ with the property that
 \[
 |E_h\setminus(F_h^\s\cup G_h^\s)|<\s\,,\quad ||F_h^\s|-\l\,|B||<\s\,,\quad ||G_h^\s|-(1-\l)\,|B||<\s\,,
 \]
 and $\dist(F_h^\s,G_h^\s)\to+\infty$ as $h\to\infty$. Let us now set
 \[
 K_{s,\eta}(z):=\frac{1_{\{\eta<|z|<\eta^{-1}\}}}{|z|^{n+s}}+\frac{1_{\{|z|<\eta\}}}{\eta^{n+s}}\,,\qquad z\in\R^n\,,
 \]
 so that $K_{s,\eta}(x-y)\le|x-y|^{-(n+s)}$, and thus
 \begin{eqnarray*}
 P_{s_h}(E_h)&\ge& \int_{F_h^\s}\int_{E_h^c}K_{s_h,\eta}(x-y)\,dxdy+\int_{G_h^\s}\int_{E_h^c}K_{s_h,\eta}(x-y)\,dxdy
   \\
   &\ge& \int_{F_h^\s}\int_{(F_h^\s)^c}K_{s_h,\eta}(x-y)\,dxdy+\int_{G_h^\s}\int_{(G_h^\s)^c}K_{s_h,\eta}(x-y)\,dxdy-\frac{C(n)\s}{\eta^{n+s_h}}
   \\
   &\ge& \int_{B_{a_h^\s}}\int_{(B_{a_h^\s})^c}K_{s_h,\eta}(x-y)\,dxdy+\int_{B_{b_h^\s}}\int_{(B_{b_h^\s})^c}K_{s_h,\eta}(x-y)\,dxdy-\frac{C(n)\s}{\eta^{n+s_h}}\,,
 \end{eqnarray*}
 where in the last inequality we have used \cite[Lemma A.2]{FS} and we have chosen $a_h^\s,b_h^\s>0$ in such a way that $|B_{a_h^\s}|=|F_h^\s|$ and
 $|B_{b_h^\s}|=|G_h^\s|$. We now {\it first} let $\s\to 0^+$, to obtain
 \begin{eqnarray*}
 P_{s_h}(E_h)\ge \int_{B_{a}}\int_{(B_{a})^c}K_{s_h,\eta}(x-y)\,dxdy+\int_{B_{b}}\int_{(B_{b})^c}K_{s_h,\eta}(x-y)\,dxdy\,,
 \end{eqnarray*}
 where $a$ and $b$ are such that $|B_a|=\l\,|B|$ and $|B_b|=(1-\l)|B|$. Next we let $\eta\to0^+$, divide by $P_{s_h}(B)$, and then let $h\to\infty$ to reach the contradiction
 \[
 1\ge \frac{P_{s_*}(B_a)}{P_{s_*}(B)}+\frac{P_{s_*}(B_b)}{P_{s_*}(B)}=\l^{(n-s_*)/n}+(1-\l)^{(n-s_*)/n}>1\,.
 \]
 This completes the proof of the lemma.
 \end{proof}

 Next, we introduce the variational problems with penalization needed to round-up the nearly-isoperimetric sets $E_h$ into nearly-spherical sets $F_h$. Precisely, we shall consider the problems
 \begin{equation}\label{variational}
   \inf\Big\{(1-s)\,P_s(E)+\Lambda\,\bigl| |E|-|B|\bigr|+|{\boldsymbol\alpha}(E)-{\boldsymbol\alpha}|:E\subset\R^n\Big\}\,,
 \end{equation}
 where $s\in(0,1)$, $\Lambda\ge0$, ${\boldsymbol\alpha}>0$, and
 \[
 {\boldsymbol\alpha}(E):=\inf\Big\{|E\Delta(x+B)|:x\in\R^n\Big\}\,,\qquad E\subset\R^n\,.
 \]
 Notice that the existence of minimizers in \eqref{variational} is a non-trivial issue. Indeed, minimizing sequences, in general, are compact only with respect to local convergence in measure, with respect to which $\Lambda\,\bigl| |E|-|B|\bigr|$ is just upper semicontinuous if $|E| \leq |B|$. In addition, we cannot obtain global convergence through the isoperimetric argument used in the proof of Lemma \ref{lemma soft}, since (as we shall see in the proof of  Lemma \ref{lemma code}) a minimizing sequence in \eqref{variational} will not be in general a sequence with vanishing isoperimetric gap
 (because $ {\boldsymbol\alpha}(E)$ has to stay close to $ {\boldsymbol\alpha}$). Therefore we have to resort to a finer argument, and show how to modify an arbitrary minimizing sequence into a uniformly bounded minimizing sequence. We base our argument on the following truncation lemma: the proof by contradiction is inspired by \cite[VI.14]{A}, see also \cite[Lemma 29.12]{Ma}.

 \begin{lemma} \label{lemma truncation}
 Let $n\ge 2$, $s\in(0,1)$, and $E\subset\R^n$. If $|E\setminus B|\le\eta<1$, then there exists $1\leq r_E\le 1+ C_1(n,s)\,\eta^{1/n}$ such that
 \begin{equation}
   \label{tr1}
 (1-s)\,P_s(E\cap B_{r_E})\le (1-s)\,P_s(E)-\frac{|E\setminus B_{r_E}|}{C_2(n,s)\,\eta^{s/n}}\,,
 \end{equation}
 where
 \begin{equation}
   \label{c1c2}
    C_1(n,s):=2^{1+(n-s)/s}\,\Big(\frac{4\,|B|^{(n-s)/n}\,P(B)}{s\,(1-s)\,P_s(B)}\Big)^{1/s} \,,\qquad C_2(n,s):=\frac{2|B|^{(n-s)/n}}{(1-s)\,P_s(B)}\,.
 \end{equation}
 In particular, $\sup\{C_1(n,s)+C_2(n,s):s_0\le s<1\}<\infty$.
 \end{lemma}

 \begin{proof}
 Without loss of generality we consider a set $E$ with $|E\setminus B|\le\eta<1$ and $|E\setminus B_{1+C_1\,\eta^{1/n}}|>0$. Correspondingly, if we set $u(r):=|E\setminus B_r|$, $r>0$, then $u$ is a decreasing function with
 \begin{equation}
     \label{golden2}
        [0,1+c_1\,\eta^{1/n}]\subset{\rm spt}\,u\,\qquad u(1)\le \eta\,,\qquad u'(r)=-\H(E\cap\pa B_r)\quad\mbox{for a.e. $r>0$}\,.
 \end{equation}
 Arguing by contradiction, we now assume that
 \begin{equation}
 \label{tr3}
 (1-s)\,P_s(E)\le (1-s)\,P_s(E\cap B_r)+\frac{u(r)}{C_2\,\eta^{s/n}}\,,\qquad\forall r\in(1,1+C_1\,\eta^{1/n})\,.
 \end{equation}
 First, we notice that we have the identity
 \[
 P_s(E\cap B_r)-P_s(E)=2\,\int_{E\cap B_r}\int_{E\cap B_r^c}\frac{dx\,dy}{|x-y|^{n+s}}-P_s(E\setminus B_r)\,,\qquad\forall r>0\,;
 \]
 second, by arguing as in the proof of \eqref{caffa}, and by \eqref{golden2}, we see that
 \[
 \int_{E\cap B_r}\int_{E\cap B_r^c}\frac{dx\,dy}{|x-y|^{n+s}}\le\frac{P(B)}s\,\int_r^\infty\,\frac{-u'(t)}{(t-r)^s}\,dt\,,\qquad\forall r>0\,;
 \]
 finally, by \eqref{nonlocal isoperimetric inequality}, $P_s(E\setminus B_r)\ge P_s(B)|B|^{(s-n)/n}\,u(r)^{(n-s)/n}$.
 We may thus combine these three remarks with \eqref{tr3} to conclude that, if $r\in(1,1+C_1\,\eta^{1/n})$, then
 \begin{eqnarray}\nonumber
     0&\le& \frac{2\,P(B)}s\,\int_r^\infty\frac{-u'(t)}{(t-r)^s}\,dt-\frac{P_s(B)}{|B|^{(n-s)/n}}\,u(r)^{(n-s)/n}+\frac{u(r)}{(1-s)\,C_2\,\eta^{s/n}}
     \\\label{golden4}
     &\le&\frac{2\,P(B)}s\,\int_r^\infty\frac{-u'(t)}{(t-r)^s}\,dt-\frac{P_s(B)}{2|B|^{(n-s)/n}}\,u(r)^{(n-s)/n}\,,
 \end{eqnarray}
 where in the last inequality we have used our choice of $C_2$ and the fact that $u(r)\le\eta$ for every $r>1$. We rewrite \eqref{golden4} in the more convenient form
 \begin{equation}
   \label{golden8}
   u(r)^{(n-s)/n}\le C_3\,\int_r^\infty\frac{-u'(t)}{(t-r)^s}\,dt\,,\qquad \forall r\in(1,1+c\,\eta^{1/n})\,,
 \end{equation}
 where we have set
 \[
 C_3(n,s):=\frac{4\,|B|^{(n-s)/n}\,P(B)}{s\,P_s(B)}\,.
 \]
 Let us set $r_k:=1+(1-2^{-k})\,C_1\,\eta^{1/n}$, so that $r_0=1$, $r_k<r_{k+1}$, and $r_\infty=1+C_1\,\eta^{1/n}$. Correspondingly, if we set $u_k=u(r_k)$, then by \eqref{golden2} we find that $u_0\le \eta$, $u_k\ge u_{k+1}$, and $u_\infty=\lim_{k\to\infty}u_k>0$. We are now going to show that \eqref{golden8} implies $u_\infty=0$, thus obtaining a contradiction and proving the lemma. Indeed, if we integrate \eqref{golden8} on $(r_k,r_{k+1})$ we get
 \begin{eqnarray}\label{plug1}
 (r_{k+1}-r_k)\,u_{k+1}^{(n-s)/n}&\le&C_3 \int_{r_k}^{r_{k+1}}\,dr\int_r^\infty\frac{-u'(t)}{(t-r)^s}\,dt
 \\\nonumber
 &=&C_3\int_{r_k}^{r_{k+1}}(-u'(t))\,dt\int_{r_k}^t\,\frac{dr}{(t-r)^s}+C_3\int_{r_{k+1}}^{\infty}(-u'(t))\,dt\int_{r_k}^{r_{k+1}}\,\frac{dr}{(t-r)^s}\,.
 \end{eqnarray}
 On the one hand we easily find that
 \begin{equation}
   \label{plug2}
    \int_{r_k}^{r_{k+1}}(-u'(t))\,dt\int_{r_k}^t\,\frac{dr}{(t-r)^s}\le \frac{(r_{k+1}-r_k)^{1-s}}{1-s}\,(u_k-u_{k+1})\,;
 \end{equation}
 on the other hand we notice that, for every $t>r_{k+1}$, since $|b^{1-s}-a^{1-s}|\le|b-a|^{1-s}$ for $a,b\ge0$,
 \[
 \int_{r_k}^{r_{k+1}}\,\frac{dr}{(t-r)^s}=\frac{(t-r_k)^{1-s}-(t-r_{k+1})^{1-s}}{1-s}\le \frac{(r_{k+1}-r_k)^{1-s}}{1-s}\,.
 \]
 Hence, since $|E|<\infty$ implies $\lim_{r\to\infty}u(r)=0$,
 \begin{equation}\label{plug3}
   \int_{r_{k+1}}^{\infty}(-u'(t))\,dt\int_{r_k}^{r_{k+1}}\,\frac{dr}{(t-r)^s}\le \frac{(r_{k+1}-r_k)^{1-s}}{1-s}\,u_{k+1}\,.
 \end{equation}
 We combine \eqref{plug1}, \eqref{plug2}, and \eqref{plug3} to find
 \[
  (r_{k+1}-r_k)\,u_{k+1}^{(n-s)/n}\le \frac{C_3}{1-s}\,(r_{k+1}-r_k)^{1-s}\,u_k\,.
 \]
 Since $r_{k+1}-r_k=C_1\,\eta^{1/n}\,2^{-k-1}$, we conclude that $u_{k+1}^{1-\a}\le N^k\,M\,u_k$, where
 \[
 \a=\frac{s}n\,,\qquad N=2^s\,,\qquad M=\Big(\frac{2}{C_1\,\eta^{1/n}}\Big)^s\,\frac{C_3}{1-s}\,.
 \]
 We notice that, since $u_0\le\eta<1$, we have $u_0\leq(N^{(1-\a)/a^2}M^{1/\a})^{-1}$ thanks to our choice of $C_1$. We are thus in the position to apply Lemma \ref{lemma degiorgi} to get $u_\infty=0$ and obtain the required contradiction.
 \end{proof}

 Given $n\ge 2$, $s\in(0,1)$, $\a>0$, and $E\subset\R^n$, let us set for the sake of brevity
 \[
 \F_{s,\Lambda,{\boldsymbol\alpha}}(E):=(1-s)\,P_s(E)+\Lambda\,\Big||E|-|B|\Big|+|{\boldsymbol\alpha}(E)-{\boldsymbol\alpha}|\,.
 \]
 We now prove the existence of global minimizers of $\F_{s,\Lambda,{\boldsymbol\alpha}}$.

 \begin{lemma}\label{lemma code}
   If $n\ge 2$, $s\in(0,1)$, $\Lambda>\Lambda_0(n,s)$ and ${\boldsymbol\alpha}<\e_1(n,s)$, then there exists a minimizer $E$ in the variational problem \eqref{variational}, that is, $\F_{s,\Lambda,{\boldsymbol\alpha}}(E)\le\F_{s,\Lambda,{\boldsymbol\alpha}}(F)$ for every $F\subset\R^n$. Moreover,
   up to a translation, this minimizer satisfies
   \[
   E\subset B_{C_4(n,s)}\,.
   \]
   Here we have set
   \begin{eqnarray*}
   \Lambda_0(n,s)&:=&\frac{(1-s)\, P_s(B)}{|B|}\,,
   \\
   \e_1(n,s)&:=&\frac12\min\Big\{1,\Big(\frac{1}{(\Lambda+1)C_2(n,s)}\Big)^{n/s},4|B|\Big\}\,,
   \\
   C_4(n,s)&:=&1+C_1(n,s)\,(2\e_1(n,s))^{1/n}\,.
   \end{eqnarray*}
   In particular, $\inf\{\e_1(n,s):s_0\le s<1\}>0$ and $\sup\{\Lambda_0(n,s)+C_4(n,s):s_0\le s<1\}<\infty$.
 \end{lemma}

 \begin{proof} {\it Step one:} We first show that, since $s\in(0,1)$ and $\Lambda>(1-s)\, P_s(B)/|B|$, then the unit ball $B$ is the unique solution, up to a translation, of the minimization problem
 \begin{equation}
   \label{min prob}
    \min\bigl\{(1-s)\,P_s(E)+\Lambda\bigl||E|-|B|\bigr|:\,E\subset\R^n\bigr\}\,.
 \end{equation}
 Indeed, by comparing any set $E$ with a ball having its same volume and thanks to \eqref{nonlocal isoperimetric inequality}, we immediately reduce the competition class in \eqref{min prob} to the family of balls in $\R^n$. Note that, if $r>1$, then $P_s(B)<P_s(B_r)$, so that only balls with radius $r\le 1$ have to be considered. At the same time, if  $\Lambda> (1-s)\,P_s(B)/\omega_n$, then one immediately gets that
 $$
 (1-s)\,P_s(B_r)+\Lambda\bigl||B_r|-|B|\bigr|=r^{n-s}(1-s)\,P_s(B)+\Lambda\omega_n(1-r^n)
 $$
 as a function of $r\in[0,1]$ attains its minimum at $r=1$.
\vskip5pt

 \noindent {\it Step two:} Let us denote by $\g$ the infimum value in \eqref{variational}, and let us consider sets $E_h$ ($h\in\N$) with $\F_{s,\Lambda,{\boldsymbol\alpha}}(E_h)\le \g+h^{-1}\,{\boldsymbol\alpha}$. Since ${\boldsymbol\alpha}<\e_1\leq2|B|$, we immediately get that $\g\le(1-s)P_s(B)$. Therefore, since by step one $(1-s)\,P_s(B)\le (1-s)\,P_s(E_h)+\Lambda\,||E_h|-|B||$, we conclude that $|{\boldsymbol\alpha}(E_h)-{\boldsymbol\alpha}|\le\,h^{-1}{\boldsymbol\alpha}$. Hence, up to translations, we obtain that
$$
        |E_h\setminus B|\le|E_h\Delta B|\le2\,{\boldsymbol\alpha}<2\e_1<1\,,\qquad\forall h\in\N\,.
  $$
 If we set $\eta:=2\,{\boldsymbol\alpha}$, then by Lemma \ref{lemma truncation} we can find $1\leq r_h\le 1+ C_1(n,s)\,\eta^{1/n}$ such that
 \begin{equation}
   \label{tr1h}
 (1-s)\,P_s(E_h\cap B_{r_h})\le (1-s)\,P_s(E_h)-\frac{|E_h\setminus B_{r_h}|}{C_2(n,s)\,\eta^{s/n}}\,.
 \end{equation}
 Since $|{\boldsymbol\alpha}(I)-{\boldsymbol\alpha}(J)|\le|I\Delta J|$ for every $I,J\subset\R^n$, if we set $F_h:=E_h\cap B_{r_h}$ then
   \[
   \Lambda\,||F_h|-|B||+|{\boldsymbol\alpha}(F_h)-{\boldsymbol\alpha}|\le \Lambda\,||E_h|-|B||+|{\boldsymbol\alpha}(E_h)-{\boldsymbol\alpha}|+(\Lambda+1)\,|E_h\setminus B_{r_h}|\,,
   \]
   so that \eqref{tr1h} implies (by our choice of $\e_1>\eta/2$)
   \[
   \F_{s,\Lambda,{\boldsymbol\alpha}}(F_h)\le \F_{s,\Lambda,{\boldsymbol\alpha}}(E_h)+\bigg((\Lambda+1)-\frac{1}{C_2(n,s)\,\eta^{s/n}}\bigg)|E_h\setminus B_{r_h}|\leq \F_{s,\Lambda,{\boldsymbol\alpha}}(E_h)\,.
   \]
   From this we conclude that $\F_{s,\Lambda,{\boldsymbol\alpha}}(F_h)\to\g$ as $h\to\infty$, that is, $\{F_h\}_{h\in\N}$ is still a minimizing sequence for \eqref{variational} with the additional feature that, by construction,
   \[
   F_h\subset B_{1+C_1\,(2\e_1)^{1/n}}\,,\qquad\forall h\in\N\,.
   \]
   It is now easy to prove the existence of a minimizer in \eqref{variational}.
 \end{proof}

 \begin{proof}[Proof of Theorem~\ref{thm 1}.]
 Since both sides of \eqref{qip1} are scaling invariant, we may assume that $|E|=|B|$. We want to show the existence of $\delta_0=\delta_0(n,s_0)>0$ such that, if $M>0$ is large enough, then
 \begin{equation}\label{qip2}
 A(E)^2\leq M\,\de_{s_0}(E)\,,\qquad\mbox{whenever $\de_{s_0}(E)\leq\delta_0$}\,.
 \end{equation}
 (Notice that, since we always have $A(E)\leq 2$, then $A(E)^2\leq (4/\de_0)\de_{s_0}(E)$ whenever $\de_{s_0}(E)\geq\delta_0$: in other words, \eqref{qip2} immediately implies \eqref{qip1}.) To prove \eqref{qip2} we argue by contradiction, assuming that there exists a sequence $E_h$ of sets with $|E_h|=|B|$, $\de_{s_0}(E_h)\to0$ as $h\to\infty$, but
 \begin{equation}\label{qip3}
 \de_{s_0}(E_h)< \frac{A(E_h)^2}M\,.
 \end{equation}
 By Lemma \ref{lemma soft} (and since $|E_h|=|B|$) we can thus find $s_h\in[s_0,1)$ and $h\in\N$ such that
 \begin{equation}\label{qip4}
 \lim_{h\to\infty}\frac{P_{s_h}(E_h)}{P_{s_h}(B)}=1\,,\qquad D_{s_h}(E_h)\le\frac{|E_h\Delta B|^2}{M|B|^2}\,,\qquad \lim_{h\to\infty}{\boldsymbol\alpha}(E_h)=0\,.
 \end{equation}
 We set ${\boldsymbol\alpha}_h:={\boldsymbol\alpha}(E_h)$ (so that, up to translations, ${\boldsymbol\alpha}_h=|E_h\Delta B|$) and consider the minimization problems
 \begin{equation}\label{qip5}
  \inf\Big\{(1-s_h)\,P_{s_h}(E)+\Lambda\bigl||E|-|B|\bigr|+|{\boldsymbol\alpha}(E)-{\boldsymbol\alpha}_h|:\,E\subset\R^n\Big\}\,,
 \end{equation}
 where $\Lambda$ is chosen so that
 \begin{equation}
   \label{choice of L}
    \Lambda>\sup_{s\in[s_0,1)}\frac{(1-s)\,P_{s}(B)}{|B|}\,;
 \end{equation}
 notice that the right-hand side of \eqref{choice of L} is finite since $(1-s)\,P_s(B)\to \om_{n-1}\,P(B)$ as $s\to 1^-$. For the same reason, $\inf_{s\in[s_0,1)}\e_1(n,s)>0$, and thus for every $h$ large enough we may entail that
$$
   {\boldsymbol\alpha}_h<\inf_{s\in[s_0,1)}\e_1(n,s)\,.
$$
 We can thus apply Lemma \ref{lemma code} to prove the existence of minimizers $F_h$ in \eqref{qip5} with
 \begin{equation}
   \label{uniformmm}
    F_h\subset B_{C_4(n,s_h)}\subset B_{C_5(n,s_0)}\,,\qquad\mbox{with}\quad C_5(n,s_0):=\sup_{s\in[s_0,1)}C_4(n,s)<\infty\,.
 \end{equation}
 We shall assume (as we can do up to translations) that
 \begin{equation}
   \label{baribari}
 \int_{F_h}x\,dx=0\,,\qquad\forall h\in\N\,.
 \end{equation}
 By the minimality of each $F_h$, recalling \eqref{qip3} and \eqref{qip4} we have that
 \begin{align}
 \label{qip6}
\F_{s_h,\Lambda,{\boldsymbol\alpha}_h}(F_h)&\leq F_{s_h,\Lambda,{\boldsymbol\alpha}_h}(E_h)=(1-s_h)\, P_{s_h}(E_h)
 \leq (1-s_h)\,P_{s_h}(B)+\frac{(1-s_h){\boldsymbol\alpha}_h^2P_{s_h}(B)}{M|B|^2}
 \\
 &\leq (1-s_h)\,P_{s_h}(F_h)+\Lambda\bigl||F_h|-|B|\bigr|+\frac{(1-s_h){\boldsymbol\alpha}_h^2P_{s_h}(B)}{M|B|^2}\,, \nonumber
 \end{align}
 where in the last inequality we used step one in the proof of Lemma \ref{lemma code}. Since ${\boldsymbol\alpha}_h\to0$, we infer that ${\boldsymbol\alpha}(F_h)/{\boldsymbol\alpha}_h\to 1$ as $h\to\infty$. By taking into account \eqref{baribari}, this implies in particular that
 \begin{equation}\label{qip7}
 \lim_{h\to\infty}|F_h\Delta B|=0\,.
 \end{equation}
 If we now exploit the minimality property of each $F_h$ together with the Lipschitz properties of $t\mapsto |t-|B||$, $t\mapsto|t-{\boldsymbol\alpha}_h|$, and the inequality $|{\boldsymbol\alpha}(I)-{\boldsymbol\alpha}(J)|\le |I\Delta J|$ for every $I,J\subset\R^n$, then we find that each $F_h$ enjoy a uniform global almost-minimality property of the form
 \begin{equation}\label{qip8}
  (1-s_h)P_{s_h}(F_h)\leq (1-s_h)P_{s_h}(G)+(\Lambda+1)|F_h\triangle G|\,,\qquad \forall G\subset\R^n\,.
  \end{equation}
  By \eqref{uniformmm}, \eqref{qip7}, \eqref{qip8}, and Corollary \ref{corollary cruciale}, we find that $F_h$ is nearly spherical, in the sense that
  $\pa F_h=\{x\,(1+u_h(x)):\,x\in\Sp\}$, where $\|u_h\|_{C^1(\pa B)}\to0$ as $h\to\infty$. Let now $\l_h>0$ be such that $|\l_h\,F_h|=|B|$, and set $G_h=\l_h\,F_h$. We notice that, by \eqref{qip6},
  \begin{align*}
    (1-s_h)\,\bigl(P_{s_h}(G_h)-P_{s_h}(B)\bigr)&=
    (1-s_h)\,P_{s_h}(F_h)\,(\l_h^{n-s}-1)+(1-s_h)\big(P_{s_h}(F_h)-P_{s_h}(B)\big)
    \\
    &\le
    (1-s_h)\,P_{s_h}(F_h)\,(\l_h^{n-s}-1)-\Lambda\,||F_h|-|B||+\frac{(1-s_h){\boldsymbol\alpha}_h^2P_{s_h}(B)}{M|B|^2}\,.
  \end{align*}
  Again by \eqref{qip6}, we have $(1-s_h)\,P_{s_h}(F_h)\le (1-s_h)\,P_{s_h}(B)+(1-s_h){\boldsymbol\alpha}_h^2P_{s_h}(B)/(M|B|^2)\le C_6$, provided we set
  \[
  C_6(n,s_0):=\sup_{s\in[s_0,1)}(1-s)\,P_s(B)\bigl(1+|B|^{-2}\inf_{s\in[s_0,1)}\e_1(n,s)^2\bigr)\,,
  \]
  and assume $M\ge 1$. Thus, by taking into account that $\l^{n-s}-1\le|\l^n-1|$ for every $\l>0$ and that $\lambda_h\to1$, we get
  \begin{eqnarray*}
    (1-s_h)\,\big(P_{s_h}(G_h)-P_{s_h}(B)\big)
    &\le&
    C_6\,(\l_h^{n-s}-1)-\frac\Lambda2\,|B|\,|\l_h^n-1|+\frac{(1-s_h){\boldsymbol\alpha}_h^2P_{s_h}(B)}{M|B|^2}
    \\
    &\le&
    \Big(C_6\,-\frac\Lambda2\,|B|\Big)|\l_h^n-1|+\frac{(1-s_h){\boldsymbol\alpha}_h^2P_{s_h}(B)}{M|B|^2}\,.
  \end{eqnarray*}
  We thus strengthen \eqref{choice of L}  into $\Lambda>C_6/|B|$ to find that $P_{s_h}(G_h)-P_{s_h}(B)\le{\boldsymbol\alpha}_h^2P_{s_h}(B)/(M|B|^2)$, that is
  \[
  D_{s_h}(G_h)\le\frac{{\boldsymbol\alpha}_h^2}{M|B|^2}\,,
  \]
  that we combine with Corollary \ref{corollary fuglede} to get
  \[
  A(G_h)^2\le \frac{C_0(n)}{s_0}\,D_{s_h}(G_h)\le \frac{C_0}{s_0\,M|B|^2}\,{\boldsymbol\alpha}_h^2\,.
  \]
  Now, by scale invariance $A(G_h)=A(F_h)$; moreover, by \eqref{qip7}, $|F_h|\to |B|$ as $h\to\infty$, and thus $A(F_h)^2\ge{\boldsymbol\alpha}(F_h)^2/(2|B|^2)$ for $h$ large enough; finally, as noticed in proving \eqref{qip7}, ${\boldsymbol\alpha}(F_h)/{\boldsymbol\alpha}_h\to 1$ as $h\to\infty$, so that $A(F_h)^2\ge{\boldsymbol\alpha}_h/(4|B|^2)$ for every $h$ large enough, and we conclude that
  \[
  \frac{{\boldsymbol\alpha}_h^2}4\le \frac{C_0}{s_0\,M}\,{\boldsymbol\alpha}_h^2\,.
  \]
  We may thus choose
  \[
  M>\max\Big\{1,\frac{4\,C_0(n)}{s_0}\Big\}\,,
  \]
  in order to find a contradiction. This completes the proof of Theorem \ref{thm 1}.
\end{proof}


\section{Proof of Theorem \ref{thm 2}}\label{section gamow} This section is devoted to the proof of Theorem \ref{thm 2}. We shall continue the enumeration of constants that we started in section \ref{section stability}, working with the same convention set in Remark \ref{remark costanti}. We begin with an existence result. In the following, given a set $E\subset\R^n$ we shall set
\[
\per_s(E)
:=
\left\{
\begin{array}{l l}
  \frac{1-s}{\omega_{n-1}}\,P_s(E)\,,&\mbox{if $s\in(0,1)$}\,,
  \\
  P(E)\,,&\mbox{if $s=1$}\,.
\end{array}
\right .
\]
Notice that, by \eqref{limit s to 1}, at least on smooth sets $\per_s$ is continuous as a function of $s\in (0,1]$.
Recall that $V_\a$ denotes the Riesz potential defined in \eqref{eq:Va}.

\begin{lemma}\label{lemma esistenza}
  If $n\ge 2$, $s\in(0,1]$, and $\a\in(0,n)$, then there exist positive constants $m_1(n,\a,s)$ and $R_0(n,s)$ with the following property: For every $m<m_1$, the variational problem
  \begin{equation}
    \label{km1}
      \inf\Big\{\per_s(E)+V_\a(E):|E|=m\Big\}
  \end{equation}
  admits minimizers, and every minimizer $E$ in \eqref{km1} satisfies (up to a translation) the uniform bound
  \[
  E\subset B_{(m/|B|)^{1/n}\,R_0}\,.
  \]
  Moreover,
  \begin{equation}
    \label{uni1}
      \sup\Big\{\frac1{m_1(n,\a,s)}+R_0(n,s):\a\in[\a_0,n)\,, s\in[s_0,1]\Big\}<\infty\,,\qquad\forall s_0\in(0,1),\,\a_0\in(0,n)\,.
  \end{equation}
\end{lemma}

\begin{proof}[Proof of Lemma \ref{lemma esistenza}]
  We first notice that, as expected, the truncation lemma for nonlocal perimeters, namely Lemma \ref{lemma truncation}, holds true as well for classical perimeters. This can be seen either by adapting the argument of Lemma \ref{lemma truncation} to the local case, or can be inferred as a particular case of \cite[Lemma 29.12]{Ma}. Either ways, one ends up showing that if $n\ge 2$ and $E\subset\R^n$ is such that $|E\setminus B|\le\eta<1$, then there exists $1\leq r_E\le 1+ C_1^*\,\eta^{1/n}$ such that
  \[
  P(E\cap B_{r_E})\le P(E)-\frac{|E\setminus B_{r_E}|}{C_2^*\,\eta^{1/n}}\,,
  \]
  where $C_1^*$ and $C_2^*$ are positive constants that depend on the dimension $n$ only. We then extend the definition of $C_1(n,s)$ and $C_2(n,s)$ given in \eqref{c1c2} to the case $s=1$ by setting $C_1(n,1)=C_1^*$ and $C_2(n,1)=C_2^*$. In conclusion, this shows that for every $n\ge 2$, $s\in(0,1]$ and $E\subset\R^n$ is such that $|E\setminus B|\le\eta<1$, there exists $1\leq r_E\le 1+ C_1(n,s)\,\eta^{1/n}$ such that
$$
  \per_s(E\cap B_{r_E})\le \per_s(E)-\frac{|E\setminus B_{r_E}|}{C_2(n,s)\,\eta^{1/n}}\,,
$$
  where $C_1(n,s)$ a $C_2(n,s)$ are such that
  \[
  \sup\Big\{C_1(n,s)+C_2(n,s):s\in[s_0,1]\Big\}<\infty\,,\qquad\forall s_0\in(0,1)\,.
  \]
  With this tool at hand, we now pick $n\ge 2$, $\a\in(0,n)$, $s\in(0,1]$, and denote by $\g$ the infimum in \eqref{km1}. We claim that for every $m<m_1$,
  \begin{equation}
    \label{km0}
      \g=\inf\Big\{\per_s(E)+V_\a(E):|E|=m\,,E\subset B_{(m/|B|)^{1/n}\,R_0}\Big\}\,,
  \end{equation}
where
$$
      m_1=m_1(n,s,\a):=|B|\,\min\bigg\{1,\frac{\per_s(B)}{8|B|^2C(n,s)\,V_\a(B)},
     \frac{\per_s(B)}{2|B|^2C(n,s)\,V_\a(B)}\,\Big(\frac{|B|}{8\,C_2\,C_7}\Big)^{2n/s}\bigg\}^{n/(\a+s)} \,,
$$
$$
      R_0(n,s):=3(1+C_1)\,,
$$
   $C(n,s)$ is a constant such that \eqref{nonlocal isoperimetric inequality quantitative}
holds, and $C_7$ is defined as
  \[
  C_7(n,s,\a):=2\,\Big(\per_s(B)+V_\a(B)\Big)\,.
  \]
  (Note that \eqref{uni1} follows immediately from $(1-s)\,P_s(B)\to \om_{n-1}\,P(B)$ as $s\to 1^+$ and from the fact that $C(n,s)\leq C(n,s_0)$ if $s\geq s_0$.)   We start noting that if $B[m]$ denotes the ball of volume $m$ then, since $m\le|B|$,
  \begin{eqnarray}
  \nonumber
  \g&\le& \per_s(B[m])+V_\a(B[m])
  \\\label{km2}
  &=&\Big(\frac{m}{|B|}\Big)^{(n-s)/n}\,\per_s(B)+\Big(\frac{m}{|B|}\Big)^{(n+\a)/n}\,V_\a(B)\hspace{0.4cm}
  \\\nonumber
  &\le&  C_7\,\Big(\frac{m}{|B|}\Big)^{(n-s)/n}\,,
  \end{eqnarray}
  where in the last inequality we have used the definition of $C_7$. If $E$ is a generic set with
  \begin{equation}
    \label{def of E}
      |E|=m\,,\qquad \per_s(E)+V_\a(E)\le \g+V_\a(B)\,\Big(\frac{m}{|B|}\Big)^{(n+\a)/n}\,,
  \end{equation}
  then by \eqref{km2} we find
  \begin{equation}
    \label{heyjoe}
      D_s(E)\le\frac{2\,(m/|B|)^{(n+\a)/n}\,V_\a(B)}{(m/|B|)^{(n-s)/n}\,\per_s(B)}
  =\frac{2\,V_\a(B)}{\per_s(B)}\,\Big(\frac{m}{|B|}\Big)^{(\a+s)/n}\,.
  \end{equation}
  Let us set $E_*:=\l\,E$ where $\l:=(|B|/m)^{1/n}$, so that $|E_*|=|B|$. Since $D_s(E)=D_s(E_*)$, up to a translation we have, recalling \eqref{nonlocal isoperimetric inequality quantitative},
$$
      |E_*\Delta B|\le|B|\bigg(C(n,s)\Big(\frac{m}{|B|}\Big)^{(\a+s)/n}\frac{\,2\,V_\a(B)}{\per_s(B)}\bigg)^{1/2}=:\eta\,.
 $$
  By Lemma \ref{lemma truncation} we can find $r_*\le 1+C_1\,\eta^{1/n}$ such that
  \[
  \per_s(E_*\cap B_{r_*})\le \per_s(E_*)-\frac{|E_*\setminus B_{r_*}|}{C_2\,\eta^{s/n}}\,.
  \]
  In particular, scaling back to $E$ and setting $r_m=r_*/\l$, we find
  \[
  \per_s(E\cap B_{r_m})\le \per_s(E)-\Big(\frac{m}{|B|}\Big)^{(n-s)/n}\frac{|B|}{C_2\,\eta^{s/n}}\,
  \frac{|E\setminus B_{r_m}|}{m}\,.
  \]
  Since trivially $V_\a(E\cap B_{r_m})\le V_\a(E)$, we conclude that
  \begin{eqnarray}\label{km4}
  \per_s(E\cap B_{r_m})+V_\a(E\cap B_{r_m})\le \per_s(E)+V_\a(E)
  -\Big(\frac{m}{|B|}\Big)^{(n-s)/n}\frac{u|B|}{C_2\,\eta^{s/n}}\,,
  \end{eqnarray}
  where we have set $u:=|E\setminus B_{r_m}|/m$. Let us now consider $F:=\mu(E\cap B_{r_m})$ for $\mu>0$ such that $|F|=m$. Since $\mu=(1-u)^{-1/n}$ with $u<\eta$, if we assume that $\eta\le1/2$, and take into account that
  \[
  \frac1{(1-u)^p}\le 1+2^{p+1}\,u\qquad\forall u\in[0,1/2]\,,
  \]
  then, by $\max\{\mu^{n-s},\mu^{n+\a}\}=\mu^{n+\a}\le 1+8\,u$ and by \eqref{km4}, we conclude that
  \begin{eqnarray*}
    \per_s(F)+V_\a(F)&=&\mu^{n-s}\per_s(E\cap B_{r_m})+\mu^{n+\a}\,V_\a(E\cap B_{r_m})
    \\
    &\le&(1+8\,u)\Big(\per_s(E\cap B_{r_m})+V_\a(E\cap B_{r_m})\Big)
    \\
    &\le&\per_s(E)+V_\a(E)+\Big(8\,C_7-\,\frac{|B|}{C_2\,\eta^{s/n}}\Big)\,\Big(\frac{m}{|B|}\Big)^{(n-s)/n}\,u\,,
  \end{eqnarray*}
  where we have also taken into account that, by \eqref{km4}, \eqref{def of E}, \eqref{km2}, and $m\le |B|$,
  \begin{eqnarray*}
    \per_s(E\cap B_{r_m})+V_\a(E\cap B_{r_m})&\le&\Big(\frac{m}{|B|}\Big)^{(n-s)/n}\,\per_s(B)+2\Big(\frac{m}{|B|}\Big)^{(n+\a)/n}\,V_\a(B)
    \\
    &\le&C_7\,\Big(\frac{m}{|B|}\Big)^{(n-s)/n}\,.
  \end{eqnarray*}
  Since the definition of $m_1$ implies that $\eta^{s/n}\le|B|/(8\,C_2\,C_7)$, we have proved that for every set $E$ as in \eqref{def of E} we can find a set $F$ with $|F|=m$ and $F\subset B_{\mu r_m}$ such that $\per_s(F)+V_\a(F)\le \per_s(E)+V_\a(E)$. This implies \eqref{km0} and completes the proof of the lemma by observing that $\mu\leq1+2^{1+1/n}u<3$ and $r_m=r_*/\l\leq(1+C_1)(m/|B|)^\frac1n$.
\end{proof}

Next, we want to show that minimizers in \eqref{km1}, once rescaled to have the volume of the unit ball, are $\Lambda$-minimizers of the $s$-perimeter for some uniform value of $\Lambda$.

\begin{lemma}\label{corollary minimi}
  If $n\ge 2$, $s\in(0,1]$, $\a\in(0,n)$, $E$ is a minimizer in \eqref{km1} for $m<m_1$, and $E_*=\l\,E$ for $\l>0$ such that $|E_*|=|B|$, then $E_*\subset B_{R_0}$ and
  \begin{equation}
    \label{austin1-mod}
      \per_s(E_*)\le \per_s(F)+ \Lambda_1\,|E_*\Delta F|\,,
  \end{equation}
  for every $F\subset \R^n$. Here,
  \begin{eqnarray*}
  \Lambda_1(n,\a,s)&:=&\frac{4\,C_7}{|B|}+\frac{6\,|B|\,(1+C_8)\,C_8^{\a/n}}{\a}\,,
  \\
  C_8(n,\a,s)&:=&\Big(1+\frac{V_\a(B)}{\per_s(B)}\Big)^{n/(n-s)}\,.
  \end{eqnarray*}
  In particular,
  \[
  \sup_{s\in[s_0,1],\a\in[\a_0,n)}\Lambda_1(n,s,\a)<\infty\,,\qquad\forall s_0\in (0,1)\,,\a_0\in(0,n)\,.
  \]
\end{lemma}

\begin{proof} We first notice that, if $F,G\subset \R^n$ with $|F|<\infty$, then
\begin{equation}
  \label{rigot}
  V_\a(F)-V_\a(G)\le\frac{2\,P(B)}\a\,\Big(\frac{|F|}{|B|}\Big)^{\a/n}\,|F\setminus G|\,.
\end{equation}
(This is a more precise version of \cite[Lemma 5.2.1]{R}.) Indeed, if $r_F=(|F|/|B|)^{1/n}$ is the radius of the ball of volume $|F|$, then
\begin{eqnarray*}
V_\a(F)-V_\a(G)\le2\int_F\int_{F\setminus G}\frac{dx\,dy}{|x-y|^{n-\a}}=2\int_{F\setminus G}dx\,\int_F\,\frac{dy}{|x-y|^{n-\a}}
\le2|F\setminus G|\int_{B_{r_F}}\,\frac{dz}{|z|^{n-\a}}\,,
\end{eqnarray*}
that is \eqref{rigot}. We now prove that $E_*$ satisfies \eqref{austin1-mod}. Of course, we may directly assume that $\per_s(F)\le \per_s(E_*)$. We also claim that we can reduce to prove \eqref{austin1-mod} in the case that
\begin{equation}\label{FB upper}
\frac12\le\frac{|F|}{|B|}\le C_8\,.
\end{equation}
Indeed, if we compare $E$ with a ball of volume $m$ (see \eqref{km2}) and then multiply the resulting inequality by $\l^{n-s}$, we find
\begin{equation} \label{beckner}
\per_s(E_*)+\frac{V_\a(E_*)}{\l^{\a+s}}\le   \per_s(B)+\frac{V_\a(B)}{\l^{\a+s}}\le  \per_s(B)+V_\a(B)\,,
\end{equation}
where in the last inequality we have taken into account that $\l\ge 1$ (because $m\le m_1\le |B|$). If now $F$ is such that $|F|\le |B|/2$, then $|E_*\Delta F|\ge |B|/2$, and thus \eqref{austin1-mod} trivially holds true by \eqref{beckner} and our definition of $\Lambda_1$. If instead the upper bound in \eqref{FB upper} does not hold, then we obtain a contradiction by combining $\per_s(F)\le \per_s(E_*)$, \eqref{nonlocal isoperimetric inequality} (or the classical isoperimetric inequality if $s=1$), and \eqref{beckner}. We have thus reduced to prove \eqref{austin1-mod} in the case that \eqref{FB upper} holds true. If we now set $\mu=(m/|F|)^{1/n}$, then $|\mu\,F|=m$, and by minimality of $E$ in \eqref{km1} and by \eqref{rigot} we find that
\begin{eqnarray*}
\per_s(E)&\le&\per_s(\mu\,F)+V_\a(\mu\,F)-V_\a(E)
\\
&=&\per_s(\mu\,F)+\mu^{n+\a}\Big(V_\a(F)-V_\a(E_*)\Big)+\Big((\l\,\mu)^{n+\a}-1\Big)\,V_\a(E)\,,
\end{eqnarray*}
where in the last identity we have added and subtracted $V_\a(\l\,\mu\,E)$. We multiply this inequality by $\l^{n-s}$,  apply \eqref{rigot} and \eqref{FB upper} to the second term on the right-hand side, and take into account that $\l^{n-s}\,V_\a(E)=\l^{-s-\a}\,V_\a(E_*)$, to find that
\begin{eqnarray}\label{taco0}
\per_s(E_*)&\le&(\l\,\mu)^{n-s}\,\per_s(F)+\l^{n-s}\mu^{n+\a}\,\frac{2\,P(B)\,C_8^{\a/n}}{\a}\,|F\setminus E_*|
\\\nonumber
&&+\Big((\l\,\mu)^{n+\a}-1\Big)\,\frac{V_\a(E_*)}{\l^{\a+s}}\,.
\end{eqnarray}
We now estimate the various terms on the right-hand side of \eqref{taco0}. Since $|F|\ge |B|/2$ and $|B|-|F|=|E_*|-|F|\le|E_*\Delta F|$ give
\begin{equation}
  \label{lmu}
  (\l\,\mu)^{n-s}=\Big(1+\frac{|B|-|F|}{|F|}\Big)^{(n-s)/n}\le1+\frac{n-s}{n}\,\frac{|E_*\Delta F|}{|B|/2}\le 1+\frac{2}{|B|}|E_*\Delta F|\,,
\end{equation}
by $\per_s(F)\le \per_s(E_*)$ and \eqref{beckner} we find
\begin{eqnarray}\label{taco1}
(\l\,\mu)^{n-s}\,\per_s(F)\le \per_s(F)+\frac{C_7}{|B|}\,|E_*\Delta F|\,.
\end{eqnarray}
Since \eqref{FB upper} also gives $|E_*\Delta F|\le (1+C_8)\,|B|$, by \eqref{lmu} and $m\le |B|$ we  have
\begin{eqnarray}\nonumber
\l^{n-s}\,\mu^{n+\a}&=&\mu^{\a+s}\,(\l\mu)^{n-s}\le \Big(\frac{m}{|B|}\Big)^{(\a+s)/n}
\Big(1+\frac{2(n-s)}{P(B)}\,|E_*\Delta F|\Big)
\\\label{taco3}
&\le& 1+\frac{2(n-s)}{n}\,(1+C_8)
\le 3\,(1+C_8)\,.
\end{eqnarray}
Finally, by $|F|\ge |B|/2$ we find that
\begin{equation}\label{taco4}
(\l\,\mu)^{n+\a}=\Big(1+\frac{|B|-|F|}{|F|}\Big)^{1+(\a/n)}\le 1+(2^{1+(\a/n)}-1)\,\Big|\frac{|B|-|F|}{|F|}\Big|\le 1+\frac{6}{|B|}\,|E_*\Delta F|\,,
\end{equation}
that combined with \eqref{beckner} gives
\[
\Big((\l\,\mu)^{n+\a}-1\Big)\,\frac{V_\a(E_*)}{\l^{\a+s}}\le \frac{3\,C_7}{|B|}\,|E_*\Delta F|\,.
\]
We now plug \eqref{taco1}, \eqref{taco3}, \eqref{taco4}, and \eqref{beckner} into \eqref{taco0} to complete the proof of \eqref{austin1-mod}.
\end{proof}

\begin{proof}
  [Proof of Theorem \ref{thm 2}] Let us fix $s_0\in(0,1)$ and $\a_0\in(0,n)$, and let
  \[
  \bar{m}_1:=\inf\Big\{m_1(n,\a,s):\a\in[\a_0,n)\,,s\in[s_0,1)\Big\}\,,
  \]
  so that, by Lemma \ref{lemma esistenza} and Lemma \ref{corollary minimi}, $\bar{m}_1>0$ and for every $m<\bar{m}_1$, $\a\in[\a_0,n)$, and $s\in[s_0,1)$, there exists a minimizer $E_{m,\a,s}$ of
  \[
  \inf\Big\{\per_s(E)+V_\a(E):|E|=m\Big\}
  \]
  such that
  \[
  \per_s(E_{m,\a,s})\le\per_s(F)+\bar\Lambda_1\,|E_{m,\a,s}\Delta F|\,,\qquad\forall F\subset\R^n\,,
  \]
  where
  \[
  \bar\Lambda_1:=\sup\Big\{\Lambda_1(n,\a,s):\a\in[\a_0,n)\,,s\in[s_0,1)\Big\}<\infty\,.
  \]
  We now want to show the existence of $m_0\le \bar{m}_1$ such that  $A(E_{m,\a,s})=0$
  for $m<m_0$, which implies that $E_{m,\a,s}$ is a ball
  (recall \eqref{asymmetry}).

  We argue by contradiction and construct sequences $\{s_h\}_{h\in\N}\subset[s_0,1]$, $\{\a_h\}_{h\in\N}\subset[\a_0,n)$, and $\{E_h\}_{h\in\N}$ minimizers of $\per_{s_h}+V_{\a_h}$ at volume $m_h$, such that $m_h\to 0^+$ as $h\to\infty$ and, if we set $\l_h=(|B|/m_h)^{1/n}$, then $E_{h,*}=\l_h\,E_h$ is a $\bar\Lambda_1$-minimizers of the $s_h$-perimeter with
$$
      |E_{h,*}|=|B|\,,\qquad A(E_{h,*})=A(E_h)>0\,,\qquad\forall h\in\N\,.
$$
  By \eqref{heyjoe} and either by Theorem \ref{thm 1} if $s_h<1$, or by \cite[Theorem 1.1]{FMP} in the case $s_h=1$, we have that, for a suitable positive constant $C(n,s_0)$,
  \[
  \frac{A(E_h)^2}{C_0(n,s_0)}\le D_{s_h}(E_h)\le\frac{2\,V_{\a_h}(B)}{\per_{s_h}(B)}\,\Big(\frac{m_h}{|B|}\Big)^{(\a_h+s_h)/n}\,,
  \]
  so that
$$
      A(E_{h,*})\le C(n,s_0,\a_0)\,m_h^{(\a_0+s_0)/2n}\,,\qquad\forall h\in\N\,.
$$
  Up to translations, we may thus assume
$$
    \lim_{h\to\infty}|E_{h,*}\Delta B|=0\,.
$$
  By Corollary \ref{corollary cruciale}, we thus have
  \[
  \pa E_{h,*}=\Big\{(1+u_h(x))\,x:x\in\pa B\Big\}\,,\qquad u_h\in C^1(\pa B)\,,\qquad\forall h\in\N\,,
  \]
  where $\|u_h\|_{C^1(\pa B)}\to0$ as $h\to\infty$. Since $|E_{h,*}|=|B|$, by Lemma \ref{lemma fuglede alpha} below we find that
  \[
  V_{\a}(B)-V_{\a}(E_{h,*})\le C(n)\,\Big([u_h]^2_{\frac{1-\alpha}{2}}+\|u_h\|_{L^2(\pa B)}^2\Big)\,,\qquad\forall \a\in(0,n)\,,
  \]
  where
  \[
  [u]^2_{\frac{1-\alpha}{2}}:=\iint_{\pa B\times \pa B}\frac{|u(x)-u(y)|^2}{|x-y|^{n-\a}}d\H_x\,d\H_y\,.
  \]
  Notice, in particular, that
\begin{equation}
\label{eq:norm s a}
  [u]^2_{\frac{1-\alpha}{2}}\le 2^{\a+s}\,[u]^2_{\frac{1+s}{2}}\,,\qquad\forall \a\in(0,n)\,,s\in(0,1)\,.
  \end{equation}
  At the same time, by $\per_{s_h}(E_h)+V_{\a_h}(E_h)\le \per_{s_h}(B_{r_h})+V_{\a_h}(B_{r_h})$, where $|B_{r_h}|=m_h$, we have
  \begin{eqnarray*}
  \de_{s_0}(E_h)&\le& D_{s_h}(E_h)\le \frac{V_{\a_h}(B_{r_h})-V_{\a_h}(E_h)}{\per_{s_h}(B_{r_h})}
  \\
  &\le&
  m_h^{(\a_h+s_h)/n}\,\frac{C(n)\,\Big([u_h]^2_{\frac{1-\alpha}{2}}+\|u_h\|_{L^2(\pa B)}^2\Big)}{\inf_{s\in[s_0,1)}\per_{s}(B)}
  \\
  &\leq &C(n,s_0)\,m_h^{(\a_h+s_h)/n}\,\Big([u_h]^2_{\frac{1+s_0}{2}}+\|u_h\|_{L^2(\pa B)}^2\Big)\,,
  \end{eqnarray*}
  where we used \eqref{eq:norm s a}.
  On the other hand, by Theorem \ref{fuglede} (notice that we can assume without loss of generality that $\int_{E_h}x\,dx=0$ for every $h\in\N$)
  \[
  \de_{s_0}(E_h)\ge \frac{s_0}{C(n)}\,\Big([u_h]^2_{\frac{1+s_0}{2}}+\|u_h\|_{L^2(\pa B)}^2\Big)\,.
  \]
  We have thus proved
  \[
  \frac{s_0}{C(n)}\le C(n,s_0)\,m_h^{(\a_h+s_h)/n}\,,
  \]
  and since $\a_h\ge\a_0$, $s_h\ge s_0$, and $m_h \to 0$, this inequality leads to a contradiction for $h$ sufficiently large.
\end{proof}

Let us recall that, by Riesz's rearrangement inequality, for every $\a\in(0,n)$
\begin{equation}\label{Valpha max}
V_\a(B)\ge V_\a(E)\qquad\mbox{ whenever $|E|=|B|$}\,,
\end{equation}
with equality if and only if $E=x+B$ for some $x\in\R^n$. (Indeed, the radial convolution kernel $|z|^{\a-n}$ is strictly decreasing.) Due to the maximality property of balls expressed in \eqref{Valpha max}, one expect the quantity $V_\a$ to satisfy an estimate of the form $V_\a(E)\ge V_\a(B)-C(n,\a)\,\|u\|^2$ on nearly spherical sets of volume $|B|$, for some suitable norm $\|\cdot\|$. This is exactly the content of the following lemma.

\begin{lemma}
\label{lemma fuglede alpha}
  There exist positive constants $\e_0$ and $C_0$, depending on $n$ only, with the following property: If $E\subset\R^n$ is an open set such that $|E|=|B|$ and
$$
  \pa E=\Big\{(1+u(x))\,x:x\in\pa B\Big\}\,,
$$
  for some function  $u\in C^1(\pa B)$ with $\|u\|_{C^1(\pa B)}\le\e_0$, then
$$
  V_\a(B)-V_\a(E)\leq C_0\,\Big([u]^2_{\frac{1-\a}{2}}+\a V_\a (B)\|u\|_{L^2(\Sp)}^2\Big)\,,\qquad\forall \a\in(0,n)\,.
$$
\end{lemma}

\begin{proof} The proof of this result is very similar to the one of Theorem \ref{fuglede}.

 As in that proof, we slightly change notation and assume that $E_t$ is an open set with $|E_t|=|B|$ and
  \[
  E_t=\Big\{(1+t\,u(x))\,x:x\in\pa B\Big\}\,,\qquad \|u\|_{C^1(\pa B)}\le\frac12\,,\qquad t\in(0,2\e_0)\,.
  \]
  Given $r,\rho,\theta\ge0$ we now set
  \[
  f_{\theta}(r,\rho):=\frac{r^{n-1}\,\rho^{n-1}}{(|r-\rho|^2+r\,\rho\,\theta^2)^{(n-\a)/2}}\,,
  \]
  so that
  \[
  V_\a(E_t)=\int_{\pa B}d\H_x\int_{\pa B}d\H_y\int_0^{1+t\,u(x)}dr\int_0^{1+t\,u(y)}f_{|x-y|}(r,\rho)\,d\rho\,.
  \]
  By exploiting the identity
  \[
  2\,\int_0^a\int_0^b=\int_0^a\int_0^a+\int_0^b\int_0^b-\int_a^b\int_a^b\,,\qquad a,b\in\R\,,
  \]
  we find that
  \begin{eqnarray}\label{Valpha Ft}
  V_\a(E_t)&=&\int_{\pa B}d\H_x\int_{\pa B}d\H_y\int_0^{1+t\,u(x)}dr\int_0^{1+t\,u(x)}f_{|x-y|}(r,\rho)\,d\rho
  \\\nonumber
  &&-\frac12\int_{\pa B}d\H_x\int_{\pa B}d\H_y\int_{1+t\,u(y)}^{1+t\,u(x)}dr\int_{1+t\,u(y)}^{1+t\,u(x)}f_{|x-y|}(r,\rho)\,d\rho\,.
  \end{eqnarray}
  By a change of variable, for every $x\in\pa B$ we find
  \begin{eqnarray*}
  &&\int_{\pa B}d\H_y\int_0^{1+t\,u(x)}dr\int_0^{1+t\,u(x)}f_{|x-y|}(r,\rho)\,d\rho
  \\&=&(1+t\,u(x))^{n+\a}\,\int_{\pa B}d\H_y\int_0^{1}dr\int_0^{1}f_{|x-y|}(r,\rho)\,d\rho
  =(1+t\,u(x))^{n+\a}\,\frac{V_\a(B)}{P(B)}\,,
  \end{eqnarray*}
  where in the last identity we have used \eqref{Valpha Ft} with $u=0$. Hence,
  \begin{eqnarray}\nonumber
  V_\a(E_t)&=&-\frac12\int_{\pa B}d\H_x\int_{\pa B}d\H_y\int_{1+t\,u(y)}^{1+t\,u(x)}dr\int_{1+t\,u(y)}^{1+t\,u(x)}f_{|x-y|}(r,\rho)\,d\rho
  \\\nonumber
  &&+\frac{V_\a(B)}{P(B)}\,\int_{\pa B}(1+t\,u)^{n+\a}\,d\H\,,
  \end{eqnarray}
  from which we conclude that
$$
  V_\a(B)-V_\a(E_t)=\frac{t^2}2\,g(t)+\frac{V_\a(B)}{P(B)}\,(h(0)-h(t))\,,
$$
  provided we set $h(t):=\int_{\pa B}(1+t\,u)^{n+\a}\,d\H$ and
  \begin{eqnarray*}
    g(t):=\int_{\pa B}d\H_x\int_{\pa B}d\H_y\int_{u(y)}^{u(x)}dr\int_{u(y)}^{u(x)}f_{|x-y|}(1+t\,r,1+t\,\rho)\,d\rho\,.
  \end{eqnarray*}
  Since $|B|=|E_t|$ implies $\int_{\pa B}(1+t\,u)^n=n\,|E_t|=n\,|B|=P(B)=h(0)$, we get
  \begin{align*}
  h(0)-h(t)&=\int_{\pa B}(1+tu)^n\big(1-(1+tu)^{\a}\big)\,d\H\\
  &\le -\a\,t\,\int_{\pa B}u\,d\H-\a\,(2n+\a-1)\,\frac{t^2}2\int_{\pa B}u^2\,d\H+C(n)\,\a\,t^3\,\|u\|_{L^2}^2\,.
  \end{align*}
  In addition, because $|B|=|E_t|$ also gives $0=\int_{\pa B}\big((1+t\,u)^n-1\big)$, we can likewise deduce that
  \begin{eqnarray*}
    -t\int_{\pa B}u\,d\H\le (n-1)\frac{t^2}2\int_{\pa B}u^2\,d\H + C(n)\,t^3\,\|u\|_{L^2}^2\,,
  \end{eqnarray*}
  therefore
  \[
  h(0)-h(t)\le -\a\,(n+\a)\,\frac{t^2}2\,\int_{\pa B}u^2\,d\H+\a\,C(n)\,t^3\,\|u\|_{L^2}^2\,.
  \]
Furthermore, we notice that
$$
    g(0)=\iint_{\pa B\times\pa B}\frac{|u(x)-u(y)|^2}{|x-y|^{n-\a}}d\H_x\,d\H_y=[u]^2_{\frac{1-\alpha}{2}}\,.
$$
  Arguing as in the proof of Theorem \ref{fuglede}, we infer that $g(t)=g(0)+t\,g'(\tau)$ for some $\tau\in(0,t)$ and with $|g'(\tau)|\le C(n)\,g(0)$. Hence,
 \begin{multline}\label{quantpotinterm}
 V_\a(B)-V_\a(E_t)\leq \frac{t^2}{2} \left([u]^2_{\frac{1-\alpha}{2}}-\alpha(n+\alpha)\frac{V_\a(B)}{P(B)}\,\|u\|^2_{L^2} \right)\\
 +C(n)\,t^3\Big([u]^2_{\frac{1-\alpha}{2}}+\alpha V_\alpha(B) \|u\|_{L^2}^2\Big)\,.
 \end{multline}
This last estimate obviously implies the announced result.
\end{proof}


\section{First and second variation formulae and local minimizers}\label{section variation formulae} In this section we provide first and second variation formulae for the functionals $P_s$ (compare with \cite[Section 4]{DdPW}) and $V_\alpha$, and actually for generic nonlocal functionals behaving like $P_s$ and $V_\a$. Before introducing our precise setting, let us recall what is the situation in the case of the classical perimeter functional (see, e.g., \cite[Section 9]{S}, \cite[Chapter 10]{Giusti} or \cite[Sections 17.3 and 17.6]{Ma}), and set some useful terminology.

Given an open set $\Om$ and a vector field $X\in C^\infty_c(\Om;\R^n)$, we denote by $\{\Phi_t\}_{t\in\R}$ the {\it flow induced by $X$}, that is the smooth map $(t,x)\in\R\times\R^n\mapsto \Phi_t(x)\in\R^n$ defined by solving the family of ODEs (parameterized by $x\in\R^n$)
\begin{equation}
\label{ODE}
\begin{cases}
\pa_t\Phi_t(x)=X(\Phi_t(x))\,,& t\in\R\,, \\
\Phi_0(x)=x\,.
\end{cases}
\end{equation}
By the implicit function theorem, there always exists $\e>0$ such that $\{\Phi_t\}_{|t|<\e}$ is a smooth family of diffeomorphisms. Given $E\subset\R^n$ with $|E|<\infty$, one says that $X$ {\it induces a volume-preserving flow on $E$} if $|\Phi_t(E)|=|E|$ for every $|t|<\e$.

If $E$ is a set of finite perimeter in $\Om$ and $E_t:=\Phi_t(E)$, then $\{E_t\}_{|t|<\e}$ is a family of sets of finite perimeter in $\Om$, $t\mapsto P(E_t;\Om)$ is a smooth function on $|t|<\e$ (thanks to the area formula for rectifiable sets), and it makes sense to define the first and second variations of the perimeter at $E$ along $X$ (or, more precisely, along the flow induced by $X$ via \eqref{ODE}) as
\[
\de P(E;\Om)[X]:=\frac{d}{dt}P(E_t;\Om)_{\bigl|t=0}\,,\qquad \de^2P(E;\Om)[X]:=\frac{d^2}{dt^2}P(E_t;\Om)_{\bigl|t=0}\,.
\]
One says that $E$ is a {\it volume-constrained stationary set for the perimeter in $\Om$} if $\de P(E;\Om)[X]=0$ whenever $X$ induces a volume-preserving flow on $E$; if in addition $\de^2 P(E;\Om)[X]\ge0$ for every $X$ inducing a volume-preserving flow on $E$, then $E$ is said to be a {\it volume-constrained stable set for the perimeter in $\Om$}. The interest into these properties stems from the immediate fact that if $E$ is a {\it local volume-constrained perimeter minimizer in $\Om$}, that is, if $P(E;\Om)<\infty$ and, for some $\de>0$,
\begin{equation}\label{local volume constrained minimizer P}
P(E;\Om)\le P(F;\Om)\,,\qquad\forall F\subset\Om\,,\quad |E|=|F|\,,\quad |E\Delta F|<\de\,,
\end{equation}
then $E$ is automatically a volume-constrained stable set for the perimeter in $\Om$. In order to effectively exploit stability one needs explicit formulas for $\de P(E;\Om)[X]$ and $\de^2 P(E;\Om)[X]$ in terms of $X$. When $\pa E\cap\Om$ is a $C^2$-hypersurface one can obtain such formulas by using the area formula, Taylor's expansions, and the divergence theorem on $\pa E\cap\Om$. Denoting by ${\mathrm H}_{\pa E}$ the scalar mean curvature of $\pa E\cap\Omega$ (with respect to the orientation induced by the outer unit normal $\nu_E$ to $E$), by ${\mathrm c}^2_{\pa E}$  the sum of the squares of the principal curvatures of $\pa E\cap\Omega$, and setting $\zeta=X\cdot\nu_E$ for the normal component of $X$ with respect to $\nu_E$, one gets the classical formulae
\begin{eqnarray}
 \label{first variation classical perimeter}
 \de P(E;\Om)[X]&=&\int_{\pa E\cap\Omega}{\rm H}_{\pa E}\,\zeta\,d\H\,,
 \\\label{second variation classical perimeter}
 \de^2 P(E;\Om)[X]&=&\int_{\pa E} |\nabla_\tau\zeta|^2-{\rm c}^2_{\pa E}\,\zeta^2\,d\H
 \\\nonumber
 &&+\int_{\pa E} {\rm H}_{\pa E} \big( (\Div X)\,\zeta-\Div_{\tau}\bigl(\zeta\,X_\tau\bigl)\big)\,d\H\,.
\end{eqnarray}
(Here, $X_\tau=X-\zeta\,\nu_E$ is the tangential projection of $X$ along $\pa E$, while $\nabla_\tau$ and $\Div_\tau$ denote the tangential gradient and the tangential divergence operators to $\pa E$.) If $E$ is a volume-constrained stationary set for the perimeter in $\Om$, then ${\rm H}_{\pa E}$ is constant on $\pa E\cap\Om$ and
\begin{equation}
  \label{second variation classical perimeter vera}
  \de^2 P(E;\Om)[X]=\int_{\pa E} |\nabla_\tau\zeta|^2-{\rm c}^2_{\pa E}\,\zeta^2\,d\H
\end{equation}
whenever $X$ induces a volume-preserving flow on $E$. Indeed, $|E_t|=|E|$ for every $|t|<\e$ implies
\begin{equation}
  \label{svf6.5}
  0=\frac{d}{dt}|E_t|_{\bigl|t=0}=\int_{\pa E} \zeta\,d\H\,,\qquad 0=\frac{d^2}{dt^2}|E_t|_{\bigl|t=0}=\int_{\pa E}\,(\Div X)\,\zeta\,d\H\,.
\end{equation}
By combining the first condition in \eqref{svf6.5} with $\de P(E;\Om)[X]=0$ and \eqref{first variation classical perimeter}, one finds that ${\rm H}_{\pa E}$ is constant on $\pa E\cap\Om$. By combining \eqref{second variation classical perimeter}, the second condition in \eqref{svf6.5}, the fact that ${\rm H}_{\pa E}$ is constant on $\pa E\cap\Om$, and the identity $\int_{\pa E}\Div_{\tau}\bigl(\zeta\,X_\tau\bigl)\,d\H=0$ (which follows by the tangential divergence theorem), one deduces \eqref{second variation classical perimeter vera}.

We now want to obtain these kind of variation formulas for the nonlocal functionals considered in this paper. We shall actually work in a broader framework. Precisely, given $s\in (0,1)$ and $\alpha\in(0,n)$, we fix thorough this section two convolution kernels $K,G\in C^1(\R^n\setminus \{0\};[0,\infty))$ which are symmetric by the origin (i.e., $K(-z)=K(z)$ and $G(-z)=G(z)$ for every $z\in\R^n\setminus\{0\}$) and satisfy the pointwise bounds
\begin{equation}\label{kernel}
K(z)\leq \frac{C_K}{|z|^{n+s}}\,, \qquad G(z)\leq \frac{C_G}{|z|^{n-\alpha}}\,,\qquad\forall z\in\R^n\setminus\{0\}\,,
\end{equation}
for some constants $C_K$ and $C_G$. Correspondingly, given $E\subset\R^n$, we consider the nonlocal functionals (defined in $[0,\infty]$)
$$
P_K(E)=\iint_{E\times E^c}K(x-y)\,dx\,dy\,,\qquad V_G(E)=\iint_{E\times E}G(x-y)\,dx\,dy\,.
$$
Notice that the two functionals are substantially different only in presence of the singularities allowed in \eqref{kernel}. Indeed, by virtue of \eqref{kernel}, $K$ is possibly singular only close to the origin, while $G$ is possibly singular only at infinity (in the sense that the integral of $G$ may diverge at infinity). When no singularity is present, then the two functionals are essentially equivalent in the sense that one has
\begin{equation}
  \label{relazione}
  P_K(E)=|E|\,\|K\|_{L^1(\R^n)}-V_K(E)\,,\qquad\mbox{if $K\in L^1(\R^n)$ and $|E|<\infty$.}
\end{equation}
We next introduce the restrictions of $P_K$ and $V_G$ to a given open set $\Om$. Following \cite{CRS}, we set
\begin{eqnarray*}
P_K(E,\Omega)&:=&\int_{E\cap\Omega}\int_{E^c\cap\Omega}K(x-y)\,dx\,dy+\int_{E\cap\Omega}\int_{E^c\setminus\Omega}K(x-y)\,dx\,dy
\\&&+\int_{E\setminus\Omega}\int_{E^c\cap\Omega}K(x-y)\,dx\,dy\,,
\\
V_G(E,\Omega)&:=&\int_{E\cap\Omega}\int_{E\cap\Omega}G(x-y)\,dx\,dy+2\int_{E\cap\Omega}\int_{E\setminus\Omega}G(x-y)\,dx\,dy\,.
\end{eqnarray*}
If $P_K(E;\Om)<\infty$, $X\in C^\infty_c(\Om;\R^n)$, and $E_t=\Phi_t(E)$ as before, then one finds from the area formula that $t\mapsto P_K(E_t;\Om)$ is a smooth function for $|t|<\e$, and correspondingly is able to define the first and second variations of $P_K(\cdot,\Om)$ at $E$ along $X$ as
$$
\de P_K(E;\Om)[X]= \frac{d }{dt}P_K(E_t;\Om)_{\bigl|t=0}\,,\qquad \de^2 P_K(E;\Om)[X]= \frac{d^2 }{dt^2}P_K(E_t;\Om)_{\bigl|t=0}\,.
$$
Identical definitions are adopted when $V_G$ is considered in place of $P_K$ and $E$ is such that $V_G(E;\Om)<\infty$ (as it is the case, for example, whenever $E$ is bounded).

Having set our terminology, we now turn to the problem of expressing first and second variations along $X$ in terms of boundary integrals involving $X$ and its derivatives, in the spirit of \eqref{first variation classical perimeter} and \eqref{second variation classical perimeter}. These formulas involve some ``nonlocal'' variants of the quantities ${\rm H}_{\pa E}$ and ${\rm c}^2_{\pa E}$, that are introduced as follows. Given $E\subset\R^n$, $x\in\R^n$, and a non-negative Borel function $J$ on $\R^n$, we define (as elements of $[-\infty,\infty]$)
\begin{eqnarray}\label{nlmc}
   {\rm H}_{J,\pa E}(x)&:=&{\rm p.v.}\left(\int_{\R^n}\bigl(\chi_{E^c}(y)-\chi_E(y)\bigr)\,J(x-y)\,dy\right)
   \\\nonumber
   &=&\limsup_{\e\to0^+}\int_{\R^n\setminus B(x,\e)}\bigl(\chi_{E^c}(y)-\chi_E(y)\bigr)\,J(x-y)\,dy\,,
   \\
   \label{nlmc theta}
   {\rm H}^*_{J,\pa E}(x)&:=&2\int_{E} J(x-y)\,dy\,.
\end{eqnarray}
Moreover, given an orientable hypersurface $M$ of class $C^1$ in $\R^n$, and denoting by $\nu_M$ an orientation of $M$, we define ${\rm c}^2_{J,M}:M\to[0,\infty]$ by setting
\begin{eqnarray}
\label{nlspc}
   {\rm c}^2_{J,M}(x)&:=&\int_{M}\,J(x-y)|\nu_M(x)-\nu_M(y)|^2\,d\H_y\,,\qquad\forall x\in M\,.
\end{eqnarray}
The functions ${\rm H}_{J,\pa E}$ and ${\rm H}^*_{J,\pa E}$ will play the role of nonlocal mean curvatures for $P_K$ when $J=K$ and for $V_G$ when $J=G$, respectively. As it turns out, if $J\in L^1(\R^n)$ then the two quantities are equivalent up to a constant and a change of sign, that is,
\[
{\rm H}_{J,\pa E}(x)=\|J\|_{L^1(\R^n)}-{\rm H}^*_{J,\pa E}(x)\,,\qquad\forall x\in\R^n\,,
\]
a result that, of course, is in accord with \eqref{relazione}. We are now in the position to the state the main theorem of this section.

 \begin{theorem}\label{secondvariation}
 Let $K,G\in C^1(\R^n\setminus\{0\};[0,\infty))$ be even functions satisfying \eqref{kernel}  for some $s\in(0,1)$ and $\a\in(0,n)$, let $\Om$ be an open set in $\R^n$, let $E\subset\R^n$ be an open set with $C^{1,1}$-boundary such that $\pa E\cap\Om$ is a $C^2$-hypersurface, and, given $X\in C^\infty_c(\Om;\R^n)$, set $\zeta=X\cdot\nu_E$. If $P_K(E;\Om)<\infty$ and $\int_{\pa E}(1+|z|)^{-n-s}\,d\H_z<\infty$, then
 \begin{eqnarray}\label{fvf}
 \de P_K(E;\Om)[X]&=&\int_{\pa E}{\rm H}_{K,\pa E}\,\zeta\,d\H\,,
 \\\nonumber
 \de^2 P_K(E;\Om)[X]
 &=&\iint_{\partial E\times\pa E}K(x-y)|\zeta(x)-\zeta(y)|^2\,d\H_x\,d\H_y-\int_{\pa E}{\rm c}^2_{K,\pa E}\,\zeta^2\,d\H
 \\\label{svf}
 &&+\int_{\pa E}{\rm H}_{K,\pa E}\,\Big((\Div X)\,\zeta-\Div_{\tau}\bigl(\zeta\,X_\tau\bigr)\Big)\,d\H\,.
 \end{eqnarray}
 If $V_G(E;\Om)<\infty$ and $\int_E |z|^{-n+\a}\,dz<\infty$, then
 \begin{eqnarray}
 \label{fvfV}\nonumber
 \de V_G(E;\Om)[X]&=&\int_{\pa E}{\rm H}^*_{G,\pa E}\,\zeta\,d\H\,.
 \\\nonumber
 \de^2 V_G(E;\Om)[X]
 &=&-\iint_{\partial E\times\pa E}G(x-y)|\zeta(x)-\zeta(y)|^2\,d\H_x\,d\H_y+\int_{\pa E}{\rm c}^2_{G,\pa E}\,\zeta^2\,d\H
 \\\label{svfV}
 &&+\int_{\pa E}{\rm H}^*_{G,\pa E}\,\Big((\Div X)\,\zeta-\Div_{\tau}\bigl(\zeta\,X_\tau\bigr)\Big)\, d\H\,.
 \end{eqnarray}
 \end{theorem}

\begin{remark}\label{remark variazione seconda stazionari}
  Let $E$ be as in Theorem \ref{secondvariation}. By arguing as in the deduction of \eqref{second variation classical perimeter vera} from \eqref{first variation classical perimeter} and \eqref{second variation classical perimeter}, we see that if $E$ is a volume-constrained stationary set for $P_K$, then
  \[
  \de^2 P_K(E;\Om)[X]=\iint_{\partial E\times\pa E}K(x-y)|\zeta(x)-\zeta(y)|^2\,d\H_x\,d\H_y-\int_{\pa E}{\rm c}^2_{K,\pa E}\,\zeta^2\,d\H\,.
  \]
  whenever $X$ is volume-preserving on $E$. Similarly, if $E$ is a volume-constrained stationary set for $V_G$, then
  \[
  \de^2 V_G(E;\Om)[X]=-\iint_{\partial E\times\pa E}G(x-y)|\zeta(x)-\zeta(y)|^2\,d\H_x\,d\H_y+\int_{\pa E}{\rm c}^2_{G,\pa E}\,\zeta^2\,d\H\,,
  \]
  whenever $X$ is volume-preserving on $E$.
\end{remark}

The fact that $\pa E$ is of class $C^{1,1}$ guarantees that ${\rm c}^2_{K,\pa E}(x)\in\R$ for every $x\in\pa E$. It also implies that $\zeta=X\cdot\nu_E$ is a Lipschitz function, which in turn guarantees that the first-integral on the right-hand side of \eqref{svf} converge. The convergence of ${\rm c}^2_{G,\pa E}$ and of the first integral on the right-hand side of \eqref{svfV} is trivial. In the next two propositions we address the continuity properties of ${\rm H}_{K,\pa E}$ and ${\rm H}^*_{G,\pa E}$.

 \begin{proposition}\label{cont}
 If $s\in(0,1)$, $K\in C^1(\R^n\setminus\{0\};[0,\infty))$ is even and satisfies $K(z)\le C_K/|z|^{n+s}$ for every $z\in\R^n\setminus\{0\}$, $\Omega$ and $E$ are open sets, and $\pa E\cap\Omega$ is an hypersurface of class $C^{1,\s}$ for some $\s\in(s,1)$, then \eqref{nlmc} defines a continuous real-valued function ${\rm H}_{K,\pa E}$ on $\pa E\cap\Om$.
 \end{proposition}

\begin{proof} Given $\de\in [0,1/2)$, let $\eta_\de\in C^\infty([0,\infty);[0,1])$ be such that $\eta_\de=1$ on $[0,\de)\cup(1/\de,\infty)$, $\eta_\de=0$ on $[2\de,1/2\de)$, and $|\eta_\de'|\le 2/\de$ on $[0,\infty)$, and $\eta_\de(s)\downarrow 0$ for every $s>0$ as $\de\to 0^+$. If we set $K_\de(z)=(1-\eta_\de(|z|))\,K(z)$, $z\in\R^n$, then $K_\de\in  C^1_c(\R^n)\subset L^1(\R^n)$, so that
\[
{\rm H}_{K_\de,\pa E}(x)=\int_{E^c}K_\de(x-y)\,dy-\int_{E}K_\de(x-y)\,dy\,,\qquad\forall x\in\R^n\,,
\]
and thus ${\rm H}_{K_\de,\pa E}$ is a continuous function on $\R^n$ for every $\de>0$. In fact, we notice for future use that ${\rm H}_{K_\de,\pa E}\in C^1(\R^n)$, with
\begin{equation}
   \label{nabla H K delta volume}
 \nabla {\rm H}_{{K_{\de}},\pa E}(x)=\int_{E^c}\nabla K_\de(x-y)\,dy-\int_{E}\nabla K_\de(x-y)\,dy\,,\qquad\forall x\in\R^n\,.
\end{equation}
Let us now decompose $x\in\R^n$ as $(x',x_n)\in\R^{n-1}\times\R$, and set
\[
C_r=\Big\{x\in\R^n:|x'|<r\,,|x_n|<r\Big\}\,,\qquad P_{r,\g}=\Big\{x\in C_r:\g\,|x'|^{1+\s}<x_n\Big\}\,,
\]
for $r>0$ and $\g>0$. If $\Om'\subset\subset\Om$, then we can find $r>0$ and $\g>0$ such that for every $x\in\pa E\cap\Om'$ there exists a rotation around the origin followed by a translation, denoted by $Q_x$, such that
\begin{equation}
  \label{inclusioni}
  \Big(C_r\setminus P_{r,\g}\Big)\cap\{x_n>0\}\subset Q_x(E^c)\,,\qquad
\Big(C_r\setminus P_{r,\g}\Big)\cap\{x_n<0\}\subset Q_x(E)\,,
\end{equation}
see Figure \ref{fig cr}.
\begin{figure}
  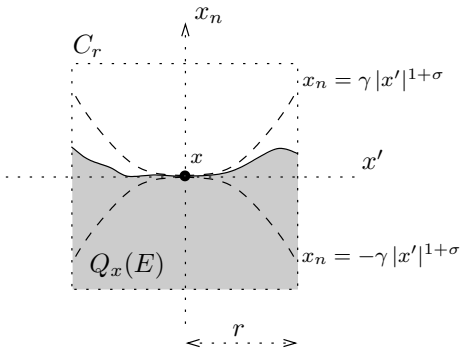\caption{{\small The sets defined in \eqref{inclusioni}. The region $P_{r,\g}$ is that part of $C_r$ encolosed by the graphs $x_n=\pm\g|x'|^{1+\s}$.}}\label{fig cr}
\end{figure}
Provided $\e<\de<2\de<r$, we thus find that
\begin{eqnarray*}
  &&\Big|\int_{\R^n\setminus B(x,\e)}(\chi_{E^c}(y)-\chi_{E}(y))\,K(x-y)\,dy-\int_{\R^n\setminus B(x,\e)}(\chi_{E^c}(y)-\chi_{E}(y))\,K_\de(x-y)\,dy\Big|
  \\&=&
  \Big|\int_{E^c\setminus B(x,\e)}\eta_\de(|x-y|)\,K(x-y)\,dy-\int_{E\setminus B(x,\e)}\eta_\de(|x-y|)\,K(x-y)\,dy\Big|
  \\&\le&
  \Big|\int_{(C_r\cap E^c)\setminus B(x,\e)}\eta_\de(|x-y|)\,K(x-y)\,dy-\int_{(C_r\cap E)\setminus B(x,\e)}\eta_\de(|x-y|)\,K(x-y)\,dy\Big|
  \\
  &&+2\,\int_{\R^n\setminus B_{1/2\de}}K(z)\,dz
  \\&\le&
  \int_{Q_x^{-1}(P_{r,\g})\setminus B(x,\e)}\eta_\de(|x-y|)\,K(x-y)\,dy+2\,\int_{\R^n\setminus B_{1/2\de}}K(z)\,dz\,,
\end{eqnarray*}
where in the last inequality we have used \eqref{inclusioni} and the symmetry of $K$ to cancel out opposite contributions from the points in $E^c$ and in $E$ lying in $Q_x^{-1}(C_r\setminus P_{r,\g})$. We now notice that
\begin{eqnarray*}
  \omega(\delta)&:=&\int_{Q_x^{-1}(P_{r,\g})\setminus B(x,\e)}\eta_\de(|x-y|)\,K(x-y)\,dy\\
  &\le&\int_{Q_x^{-1}(P_{r,\g})}\eta_\de(|x-y|)\,K(x-y)\,dy
  \\
  &=&
  \int_{|z'|<r}dz'\int_{-\g\,|z'|^{1+\s}}^{\g\,|z'|^{1+\s}}\eta_\de(|z|)\,K(z)\,dz_n
  \\
  &\le&
  C_K\,\int_{|z'|<r}dz'\int_{-\g\,|z'|^{1+\s}}^{\g\,|z'|^{1+\s}}\frac{dz_n}{(|z'|^2+|z_n|^2)^{(n+s)/2}}\,.
\end{eqnarray*}
Since $\eta_\de(z)\to 0$ for every $z\in\R^n\setminus\{0\}$ as $\de\to 0^+$, and since
\[
\int_{|z'|<r}dz'\int_{-\g\,|z'|^{1+\s}}^{\g\,|z'|^{1+\s}}\frac{dz_n}{(|z'|^2+|z_n|^2)^{(n+s)/2}}<\infty\,,
\]
we conclude that $\om(\de)\to 0$ as $\de\to 0$ (with a velocity that depends on $C_K$, $s$, $r$, $\g$ and $\s$ only). Since $\int_{\R^n\setminus B_{1/2\de}}K(z)\,dz\to 0$ as $\de\to 0^+$ (with a velocity that depends on $C_K$ and $s$ only), we conclude that, if $\om_0(\de)=\om(\de)+2\int_{\R^n\setminus B_{1/2\de}}K(z)\,dz$, then
\[
\Big|\int_{\R^n\setminus B(x,\e)}(\chi_{E^c}(y)-\chi_{E}(y))\,K(x-y)\,dy-\int_{\R^n\setminus B(x,\e)}(\chi_{E^c}(y)-\chi_{E}(y))\,K_\de(x-y)\,dy\Big|\le\om_0(\de)\,,
\]
for every $x\in\pa E\cap\Om'$ and every $\e<\de<2\de<r$. We thus conclude that ${\rm H}_{K,\pa E}(x)\in\R$ for every $x\in\pa E\cap\Om'$, and that ${\rm H}_{K_\de,\pa E}\to {\rm H}_{K,\pa E}$ uniformly on $\pa E\cap\Om'$. In particular, ${\rm H}_{K,\pa E}$ is real-valued and continuous on $\pa E\cap\Om$.
 \end{proof}

Since the function $z \mapsto  |z|^{-n+\a}$ belongs to $L^1_{\rm loc}(\R^n)$,
we also have the following result:

\begin{proposition}
   If $G\in C^1(\R^n\setminus\{0\};[0,\infty))$ is even and satisfies \eqref{kernel} for some $\a\in(0,n)$ and $\int_E |z|^{-n+\a}\,dz<\infty$ (this is the case for instance if $E$ is bounded), then \eqref{nlmc theta} defines a continuous real-valued function ${\rm H}^*_{G,\pa E}$ on $\R^n$.
\end{proposition}

 \begin{proof}[Proof of Theorem \ref{secondvariation}] We shall detail the proof of the theorem only in the case of $P_K$, being the discussion for $V_G$ similar. We denote by $\e$ the positive number such that $\{\Phi_t\}_{|t|<\e}$ is a smooth family of diffeomorphisms of $\R^n$.

 \medskip

 \noindent {\it Step one:} Given $\de\ge0$, we define ${K_{\de}}$ as in the proof of Proposition \ref{cont}. Our goal here is proving \eqref{fvf} and \eqref{svf} with $K_\de$ in place of $K$. We first claim that ${\rm H}_{{K_{\de}},\pa E}\in C^1(\R^n)$, and that $\nabla {\rm H}_{{K_{\de}},\pa E}$ can be expressed both as in \eqref{nabla H K delta volume} and as in \eqref{nabla H K delta bordo} below. Since $E$ is an open set with Lipschitz boundary and $K_\de\in C^1_c(\R^n)$, by the Gauss--Green theorem, the symmetry of $K_\de$, and \eqref{nabla H K delta volume}, we find that
 \begin{equation}
   \label{nabla H K delta bordo}
    \nabla {\rm H}_{{K_{\de}},\pa E}(x)=2\,\int_{\pa E}K_{\de}(y-x)\,\nu_E(y)\,d\H_y\,,\qquad\forall x\in\R^n\,.
 \end{equation}
 We now notice that, since $E_{t+h}=\Phi_h(E_t)$, by the area formula we get, whenever $|t|<\e$ and $|t+h|<\e$,
 \begin{multline*}
 P_{K_{\de}}(E_{t+h},\Omega)=\int_{E_t\cap\Omega}\int_{E^c_t\cap\Omega}{K_{\de}}(\Phi_h(x)-\Phi_h(y))J_{\Phi_h}(x)J_{\Phi_h}(y)\,dx\,dy\\
 +\int_{E_t\cap\Omega}\int_{E^c_t\setminus\Omega}{K_{\de}}(\Phi_h(x)-y)J_{\Phi_h}(x)\,dx\,dy
 +\int_{E_t\setminus\Omega}\int_{E^c_t\cap\Omega}{K_{\de}}(x-\Phi_h(y))J_{\Phi_h}(y)\,dx\,dy\,,
 \end{multline*}
 where $J_{\Phi_h}$ stands for the Jacobian of the map $\Phi_h$. Since $\Phi_h=\Id+h\,X+O(h^2)$ and $J_{\Phi_h}=1+h\,\Div X+O(h^2)$ uniformly on $\R^n$ as $h\to 0$, we deduce from
 \[
  \frac{d }{dt}P_{K_{\de}}(E_t,\Omega)= \frac{d }{dh}P_{K_{\de}}(E_{t+h},\Omega)_{\bigl|h=0}\,,
 \]
 and by the smoothness of ${K_{\de}}$ that
 \begin{align}
  \frac{d }{dt}P_{K_{\de}}(E_t,\Omega)&=\int_{E_t\cap\Omega}\int_{E^c_t\cap\Omega}\nabla {K_{\de}}(x-y)\cdot(X(x)-X(y))\,dx\,dy \nonumber \\
  &\qquad+\int_{E_t\cap\Omega}\int_{E^c_t\cap\Omega} {K_{\de}}(x-y)(\Div X(x)+\Div X(y))\,dx\,dy \nonumber\\
  &\qquad+\int_{E_t\cap\Omega}\int_{E^c_t\setminus\Omega}\Big(\nabla {K_{\de}}(x-y)\cdot X(x)+{K_{\de}}(x-y)\Div X(x)\Big)	\,dx\,dy\nonumber\\
  &\qquad+\int_{E_t\setminus\Omega}\int_{E^c_t\cap\Omega}\Big(\nabla {K_{\de}}(x-y)\cdot X(y)+{K_{\de}}(x-y)\Div X(y)\Big)	\,dx\,dy\nonumber\\
 \label{fvf0} &=I_1+I_2+I_3+I_4\,.\nonumber \end{align}
 By symmetry of ${K_{\de}}$ and by the divergence theorem, we find
 \begin{align*}
 I_1 &=\displaystyle\int_{E_t^c\cap\Omega}\biggl(\int_{E_t}\nabla {K_{\de}}(x-y)\cdot X(x)\,dx\biggr)\,dy+\displaystyle\int_{E_t\cap\Omega}\biggl(\int_{E_t^c}\nabla {K_{\de}}(y-x)\cdot X(y)\,dy\biggr)\,dx \\
 &=-\displaystyle\int_{E_t^c\cap\Omega}\biggl(\int_{E_t} {K_{\de}}(x-y)\Div X(x)\,dx\biggr)\,dy+ \displaystyle\int_{E_t^c\cap\Omega}\biggl(\int_{\pa E_t} {K_{\de}}(x-y)X(x)\cdot\nu_{E_t}(x)\,d\H_x\biggr)\,dy\\
& \quad -\displaystyle\int_{E_t\cap\Omega}\biggl(\int_{E_t^c} {K_{\de}}(x-y)\Div X(y)\,dy\biggr)\,dx-\displaystyle\int_{E_t\cap\Omega}\biggl(\int_{\pa E_t} {K_{\de}}(x-y)X(y)\cdot\nu_{E_t}(y)\,d\H_y\biggr)\,dx\,,
 \end{align*}
which leads to
 $$
 I_1+I_2=\!\displaystyle\int_{E_t^c\cap\Omega}\biggl(\int_{\pa E_t}\! {K_{\de}}(x-y)X(x)\cdot\nu_{E_t}(x)d\H_x\!\biggr)dy-\displaystyle\int_{E_t\cap\Omega}\biggl(\int_{\pa E_t}\! {K_{\de}}(x-y)X(y)\cdot\nu_{E_t}(y)d\H_y\!\biggr)dx\,.
 $$
 Similarly, we get that
\begin{align*}
 I_3&=\displaystyle\int_{E_t^c\setminus\Omega}\biggl(\int_{\pa E_t} {K_{\de}}(x-y)X(x)\cdot\nu_{E_t}(x)\,d\H_x\biggr)\,dy,\\
  I_4&=-\displaystyle\int_{E_t\setminus\Omega}\biggl(\int_{\pa E_t} {K_{\de}}(x-y)X(y)\cdot\nu_{E_t}(y)\,d\H_y\biggr)\,dx\,.
\end{align*}
By exploiting once more the symmetry of ${K_{\de}}$ we thus conclude that (for every $t$ small enough)
\begin{equation}
  \label{eh}
  \frac{d }{dt}P_{K_{\de}}(E_t,\Omega)=\int_{\pa E_t}\,{\rm H}_{{K_{\de}},\pa E_t}\,(X\cdot\nu_{E_t})\,d\H\,,
\end{equation}
which of course implies \eqref{fvf} with $K_\de$ in place of $K$ by setting $t=0$. Having in mind to differentiate \eqref{eh}, we now notice that, by the area formula,
$$
\int_{\pa E_t}\,{\rm H}_{{K_{\de}},\pa E_t}\,(X\cdot\nu_{E_t})\,d\H
=\int_{\pa E}\,{\rm H}_{{K_{\de}},\pa E_t}(\Phi_t)\,(X(\Phi_t)\cdot\nu_{E_t}(\Phi_t))\,J_{\Phi_t}^{\pa E}\,d\H\,,
$$
where $J_{\Phi_t}^{\pa E}$ denotes the tangential Jacobian of $\Phi_t$ with respect to $\pa E$. Therefore,
 \begin{eqnarray}\label{svf2}
 \frac{d^2 }{dt^2}P_{K_{\de}}(E_t, \Omega)_{\bigl|t=0}&=& \frac{d}{dt}\biggl(\int_{\pa E_t}{\rm H}_{{K_{\de}},\pa E_t}\,(X\cdot\nu_{E_t})\,d\H\biggr)_{\bigl|t=0}
 \\
 &=&\int_{\pa E}\frac{d}{dt}\bigl({\rm H}_{{K_{\de}},\pa E_t}(\Phi_t)\bigr)_{\bigl|t=0}\,(X\cdot\nu_E)\,d\H
 \nonumber
 \\
 &&+\int_{\pa E}{\rm H}_{{K_{\de}},\pa E}\,\frac{d}{dt}\bigl(X(\Phi_t)\cdot\nu_{E_t}(\Phi_t))\,J^{\pa E_t}_
 {\Phi_t}\bigr)_{\bigl|t=0}\,d\H\nonumber \\
 &=&J_1+J_2\,. \nonumber
 \end{eqnarray}
 In order to compute $J_1$ we begin noticing that, by the area formula and since ${K_{\de}}\in L^1(\R^n)$,
  $$
  {\rm H}_{{K_{\de}},\pa E_t}(\Phi_t(x))=\int_{\R^n}\bigl(\chi_{E^c}(y)-\chi_E(y)\bigr){K_{\de}}(\Phi_t(x)-\Phi_t(y))\,
  J_{\Phi_t}(y)\,dy\,.
  $$
  By symmetry and smoothness of ${K_{\de}}$, by the Taylor's expansions in $t$ of $\Phi_t$ and $J_{\Phi_t}$ mentioned above, by recalling that ${\rm H}_{{K_{\de}},\pa E}\in C^1(\R^n)$ and
  \eqref{nabla H K delta volume}, and by the divergence theorem, we get
  \begin{eqnarray*}
   &&\frac{d}{dt}\bigl({\rm H}_{{K_{\de}},\pa E_t}(\Phi_t(x))\bigr)_{\bigl|t=0}=\int_{\R^n}(\chi_{E^c}(y)-\chi_E(y))\nabla {K_{\de}}(x-y)\cdot(X(x)-X(y))\,dy \\
   && \qquad\qquad\qquad \qquad\qquad\quad +\int_{\R^n}(\chi_{E^c}(y)-\chi_E(y)){K_{\de}}(x-y)\Div X(y)\,dy
   \nonumber \\
   && \qquad\qquad=\nabla {\rm H}_{{K_{\de}},\pa E}(x)\cdot X(x)+\int_{E^c}\nabla {K_{\de}}(y-x)\cdot X(y)\,dy-\int_{E}\nabla {K_{\de}}(y-x)\cdot X(y)\,dy\nonumber \\
   && \qquad\qquad\qquad \qquad\qquad\quad +\int_{\R^n}(\chi_{E^c}(y)-\chi_E(y)){K_{\de}}(x-y)\Div X(y)\,dy \nonumber \\
    && \qquad\qquad=\nabla {\rm H}_{{K_{\de}},\pa E}(x)\cdot X(x)-2\int_{\pa E}{K_{\de}}(x-y)X(y)\cdot\nu_E(y)\,dy\,. \nonumber
     \end{eqnarray*}
  By this last identity and by the symmetry of ${K_{\de}}$, setting $\zeta=X\cdot\nu_E$ we find that
\begin{eqnarray}\nonumber
J_1&=&-2\iint_{\pa E\times \pa E}{K_{\de}}(x-y)\,\zeta(x)\,\zeta(y)\,d\H_x\,d\H_y+\int_{\pa E}\big(\nabla {\rm H}_{{K_{\de}},\pa E}\cdot X\big)\,\zeta\,d\H
\\
&=& \iint_{\pa E\times \pa E}{K_{\de}}(x-y)|\zeta(x)-\zeta(y)|^2\,d\H_x\,d\H_y-2\iint_{\pa E\times\pa E} {K_{\de}}(x-y)\,\zeta(x)^2\,d\H_x\,d\H_y \nonumber
\\
&& \qquad+\int_{\pa E}\big(\nabla {\rm H}_{{K_{\de}},\pa E}\cdot \nu_E\big)\,\zeta^2\,d\H +\int_{\pa E}\big(\nabla_\tau {\rm H}_{{K_{\de}},\pa E}\cdot X_\tau\big)\,\zeta\,d\H\,, \label{svf3}
\end{eqnarray}
where in the last identities we have simply completed a square and used the identity $X=\zeta\,\nu+X_\tau$. By \eqref{nabla H K delta bordo} we also get
\begin{eqnarray*}
\nabla {\rm H}_{{K_{\de}},\pa E}(x)\cdot \nu_E(x)&=&-\int_{\pa E} {K_{\de}}(x-y)|\nu_E(x)-\nu_E(y)|^2\,d\H_y+2\int_{\pa E} {K_{\de}}(x-y)\,d\H_y
\\
&=&-{\rm c}^2_{{K_{\de}},\pa E}(x)+2\int_{\pa E} {K_{\de}}(x-y)\,d\H_y\,,
\end{eqnarray*}
and thus we conclude from \eqref{svf3} that
 \begin{eqnarray}\nonumber
 J_1&=&\iint_{\partial E\times \pa E}{K_{\de}}(x-y)\,|\zeta(x)-\zeta(y)|^2\,d\H_x\,d\H_y-\int_{\pa E}\,{\rm c}^2_{{K_{\de}},\pa E}\,\zeta^2\,d\H
 \\\label{svf4}
 &&+\int_{\pa E}\big(\nabla_\tau {\rm H}_{{K_{\de}},\pa E}\cdot X_\tau\big)\,\zeta\,d\H\,.
 \end{eqnarray}
In order to compute $J_2$, we notice that, by arguing as in \cite[Step three, proof of Proposition 3.9]{CMM} (see also \cite[Section 9]{S}), one finds
  $$
  \frac{d}{dt}\Big(X(\Phi_t)\cdot\nu_{E_t}(\Phi_t)J^{\pa E}_{\Phi_t}\Big)_{\bigl|t=0}=Z\cdot\nu_E-2X_\tau\cdot\nabla_\tau\zeta+{\rm B}_{\pa E}[X_\tau,X_\tau]+\Div_\tau\bigl(\zeta\,X\bigr)\,,
  $$
  where $Z$ is the vector field defined by
  \[
  Z(x)=\pa_{tt}^2\Phi_t(x)_{\bigl|t=0}\,,\qquad x\in\R^n\,,
  \]
  and where ${\rm B}_{\pa E}$ denotes the second fundamental form of $\pa E$. Hence,
\begin{eqnarray}
  \label{svf4.25}
  J_2&=&\int_{\pa E}{\rm H}_{{K_{\de}},\pa E}\Big(Z\cdot\nu_E-2X_\tau\cdot\nabla_\tau\zeta+{\rm B}_{\pa E}[X_\tau,X_\tau]\Big)\,d\H
  \\\nonumber
  &&+\int_{\pa E}{\rm H}_{{K_{\de}},\pa E}\,\Div_\tau\bigl(\zeta\,X\bigr)\,d\H\,.
\end{eqnarray}
By the tangential divergence theorem
\[
\int_{\pa E}\Div_\tau Y\,d\H=\int_{\pa E}Y\cdot\nu_E\,{\rm H}_E\,d\H\qquad\forall Y\in C^1_c(\Om;\R^n)
\]
(recall that ${\rm H}_{\pa E}$ denotes the scalar mean curvature of $\pa E$ taken with respect to $\nu_E$), so that the sum of the second lines of \eqref{svf4} and \eqref{svf4.25} is equal to
\begin{multline*}
\int_{\pa E}{\rm H}_{{K_{\de}},\pa E}\,\Div_\tau\bigl(\zeta\,X\bigr)\,d\H+\int_{\pa E}\big(\nabla_\tau {\rm H}_{{K_{\de}},\pa E}\cdot X_\tau\big)\,\zeta\,d\H\\
=\int_{\pa E}\Div_\tau\bigl({\rm H}_{{K_{\de}},\pa E}\zeta\,X\bigr)\,d\H=\int_{\pa E}\,{\rm H}_{{K_{\de}},\pa E}\,{\rm H}_{\pa E}\,\zeta^2\,d\H\,.
\end{multline*}
We thus deduce from \eqref{svf2}, \eqref{svf4}, and \eqref{svf4.25}, that
\begin{multline*}
\frac{d^2 }{dt^2}P_{K_{\de}}(E_t, \Omega)_{\bigl|t=0}=\iint_{\partial E\times \pa E}{K_{\de}}(x-y)\,|\zeta(x)-\zeta(y)|^2\,d\H_x\,d\H_y-\int_{\pa E}\,{\rm c}^2_{{K_{\de}},\pa E}\,\zeta^2\,d\H
\\
  +\int_{\pa E}{\rm H}_{{K_{\de}},\pa E}\Big(Z\cdot\nu_E-2X_\tau\cdot\nabla_\tau\zeta+{\rm B}_{\pa E}[X_\tau,X_\tau]+{\rm H}_{\pa E}\,\zeta^2\Big)\,d\H\,.
\end{multline*}
By exploiting the identity
\[
Z\cdot\nu_E-2X_\tau\cdot\nabla_\tau\zeta+{\rm B}_{\pa E}[X_\tau,X_\tau]+{\rm H}_{\pa E}\zeta^2=-\Div_{\tau}\bigl(\zeta\,X_\tau)+(\Div X)\,\zeta
\]
(see, for example, \cite[Proof of Theorem 3.1]{AFM}), we thus come to prove \eqref{svf} with $K_\de$ in place of $K$.

\medskip

\noindent {\it Step two}: We now prove \eqref{fvf} and \eqref{svf} by taking the limit as $\de\to 0^+$ in \eqref{fvf} and \eqref{svf} with $K_\de$ in place of $K$. Let us set $\vphi_\de(t):=P_{K_\de}(E_t;\Om)$ and $\vphi(t):=P_K(E_t;\Om)$, so that $\vphi_\de$ and $\vphi$ are smooth functions on $(-\e,\e)$ with
\begin{equation}
  \label{convergenza funzioni}
  \lim_{\de\to 0^+}\vphi_\de(t)=\vphi(t)\,,\qquad\forall |t|<\e\,.
\end{equation}
(This follows by monotone convergence, as $\eta_\de\downarrow 0^+$ as $\de\to 0^+$ on $(0,\infty)$.) Let $\Om'\subset\subset\Om$ be an open set such that ${\rm spt}X\subset\subset\Om'$. Thanks to the smoothness of $\{\Phi_t\}_{|t|<\e}$, the argument in the proof of Proposition \ref{cont} can be repeated for every set $E_t=\Phi_t(E)$ corresponding to $|t|<\e$ with the same constants $r$ and $\g$, thus showing that
\begin{equation}
  \label{limite uniforme}
  \lim_{\de\to 0^+}\sup_{|t|<\e}\sup_{\pa E_t\cap\Om'}|{\rm H}_{K_\de,\pa E_t}-{\rm H}_{K,\pa E_t}|=0\,.
\end{equation}
At the same time, by step one,
\begin{eqnarray}\label{fvf delta}
 \vphi_\de'(t)=\int_{\pa E_t}{\rm H}_{K_\de,\pa E_t}\,\zeta\,d\H\,,\qquad\forall |t|<\e\,,
\end{eqnarray}
so that \eqref{limite uniforme} and \eqref{fvf delta} imply that
\begin{equation}
  \label{derivata prima}
  \lim_{\de\to 0^+}\sup_{|t|<\e}\Big|\vphi_\de'(t)-\int_{\pa E_t}\,{\rm H}_{K,\pa E_t}\,\zeta\,d\H\Big|=0\,.
\end{equation}
By the mean value theorem, \eqref{convergenza funzioni} and \eqref{derivata prima} give
$$
 \vphi'(t)=\int_{\pa E_t}{\rm H}_{K,\pa E_t}\,\zeta\,d\H\,,\qquad\forall |t|<\e\,,
$$
which implies \eqref{fvf} for $t=0$. In order to prove \eqref{svf}, we first notice that, by step one,
\begin{eqnarray}\nonumber
 \vphi_\de''(t)
 &=&\iint_{\partial E_t\times\pa E_t}K_\de(x-y)|\zeta(x)-\zeta(y)|^2\,d\H_x\,d\H_y-\int_{\pa E_t}{\rm c}^2_{K_\de,\pa E_t}\,\zeta^2\,d\H
 \\\label{svf delta}
 &&+\int_{\pa E_t}{\rm H}_{K_\de,\pa E_t}\,\Big((\Div X)\,\zeta-\Div_{\tau}\bigl(\zeta\,X_\tau\bigr)\Big)\,d\H\,,\qquad\forall |t|<\e\,.
 \end{eqnarray}
Let $A_1(t,\de)$, $A_2(t,\de)$ and $A_3(t,\de)$ denote the three integrals on the right-hand side of \eqref{svf delta}, and let $A_1(t)$, $A_2(t)$ and $A_3(t)$ stand for the corresponding integrals obtained by replacing $K_\de$ with $K$. By arguing as above, we just need to prove that for $i=1,2,3$ we have $A_i(t,\de)\to A_i(t)$ uniformly on $|t|<\e$ as $\de\to 0^+$. The fact that $A_3(t,\de)\to A_3(t)$ uniformly on $|t|<\e$ as $\de\to 0^+$ follows from \eqref{limite uniforme} and of the smoothness of $X$.
Finally, when $i=1,2$,
the uniform convergence of $A_i(t,\de)\to A_i(t)$ for $|t|<\e$ as $\de\to 0^+$
is a simple consequence of the fact that $\zeta$ is Lipschitz and compactly supported in $\Omega'$,
and that $\{\Om'\cap\pa E_t\}_{|t|<\e}$ is a uniform family of $C^2$-hypersurfaces.
This completes the proof of the theorem.
\end{proof}

\section{The stability threshold}\label{section stability threshold} In this section we consider the family of functionals $\per_s+\beta\,V_\a$ ($\beta>0$) and discuss in terms of the value of $\beta$ the volume-constrained stability of  $\per_s+\beta\,V_\a$ around the unit ball $B$. Our interest in this problem lies in the fact that, as we shall prove in section \ref{section stability implies local minimality}, stability is actually a necessary and sufficient condition for volume-constrained local minimality. Therefore the analysis carried on in this section will provide the basis for the proof of Theorem \ref{thm 3}. We set
\begin{equation}\label{defbetstar}
  \beta_\star(n,s,\alpha):=
  \begin{cases}
  \displaystyle \frac{1-s}{\omega_{n-1}}\inf_{k\geq 2}\, \frac{\lambda_k^s-\lambda_1^s}{\mu_k^\alpha-\mu_1^\alpha}\,, &\hspace{1cm} \text{if $s\in(0,1)$}\,,\\[10pt]
  \displaystyle \inf_{k\geq 2} \,\frac{\lambda_k^1-\lambda_1^1}{\mu_k^\alpha-\mu_1^\alpha}\,, &\hspace{1cm} \text{if $s=1$}\,,
  \end{cases}
  \end{equation}
  where, for every $k\in\N\cup\{0\}$,
  \begin{eqnarray}\label{eigvLapBel}
  \l_k^1&=&k(k+n-2)\,,
  \\
  \label{def lambda k s}
  \l_k^s&=&\frac{2^{1-s}\,\pi^{\frac{n-1}2}}{1+s}\,\frac{\Gamma(\frac{1-s}2)}{\Gamma(\frac{n+s}2)}\,
  \bigg(\frac{\Gamma(k+\frac{n+s}2)}{\Gamma(k+\frac{n-2-s}2)}-\frac{\Gamma(\frac{n+s}2)}{\Gamma(\frac{n-2-s}{2})}\bigg)\,,\qquad \hspace{0.2cm}s\in(0,1)\,,
  \\
  \label{formmualpha<1}
  \mu_k^\a&=&\frac{2^{1+\a}\,\pi^{\frac{n-1}2}}{1-\a}\,\frac{\Gamma(\frac{1+\a}2)}{\Gamma(\frac{n-\a}2)}\,
  \bigg(\frac{\Gamma(k+\frac{n-\a}2)}{\Gamma(k+\frac{n-2+\a}2)}-\frac{\Gamma(\frac{n-\a}2)}{\Gamma(\frac{n-2+\a}{2})}\bigg)\,,\qquad \a\in(0,1)\,,
  \\
  \label{formmualpha>1}
  \mu_k^\a&=&2^\a\,\pi^{\frac{n-1}2}\,\frac{\Gamma(\frac{\a-1}2)}{\Gamma(\frac{n-\a}2)}\,
  \bigg(\frac{\Gamma(\frac{n-\a}2)}{\Gamma(\frac{n-2+\a}{2})}-\frac{\Gamma(k+\frac{n-\a}2)}{\Gamma(k+\frac{n-2+\a}2)}\bigg)\,,
  \qquad\hspace{0.4cm}\a\in(1,n)\,,
  \\
  \label{formmualpha=1}
  \mu_k^1&=&\frac{4\,\pi^{\frac{n-1}{2}}}{\G(\frac{n-1}2)}\,\left(\frac{\G^\prime(k+\frac{n-1}{2})}{\G(k+\frac{n-1}{2})} - \frac{\G^\prime(\frac{n-1}{2})}{\G(\frac{n-1}{2})} \right)\,.
  \end{eqnarray}
  Here $\Gamma$ denotes the Euler's Gamma function, while $\Gamma'$ is the derivative of $\Gamma$, so that $\Gamma'/\Gamma$ is the digamma function. By exploiting basic properties of the Gamma function, it is straightforward to check that $\lambda_k^s/\mu_k^\alpha\to\infty$ as $k\to\infty$, so that the infimum in \eqref{defbetstar} is achieved, and $\beta_\star>0$. We shall actually prove that the infimum is always achieved at $k=2$
and the formula for $\beta_\star$ considerably simplifies (see Proposition \ref{proposition betastar}).

\begin{theorem}\label{thmstability}
The unit ball $B$ is a volume-constrained stable set for $\per_s+\beta\,V_\a$ if and only if $\beta\in(0,\beta_\star]$.
\end{theorem}

Let us first of all explain the origin of the formula \eqref{defbetstar} for $\beta_\star$. Since $B$ is a volume-constrained stationary set for $P$, $P_s$, and $V_\a$ (indeed, $B$ is a global volume-constrained minimizer of $P$ and $P_s$, and a global volume-constrained maximizer of $V_\a$), by Remark \ref{remark variazione seconda stazionari} we find that (setting $K_s(z)=|z|^{-(n+s)}$ and $G_\a(z)=|z|^{-(n-\a)}$ for every $z\in\R^n\setminus\{0\}$)
\begin{eqnarray}\label{vv1}
  \de^2P(B)[X]&=&\iint_{\pa B}|\nabla_\tau \zeta|^2 d\H  -\int_{\pa B}{\rm c}^2_{\pa B}\,\zeta^2\,d\H\,,
  \\\label{vv2}
  \de^2P_s(B)[X]&=&\iint_{\partial B\times\pa B}\frac{|\zeta(x)-\zeta(y)|^2}{|x-y|^{n+s}}\,d\H_x\,d\H_y
  -\int_{\pa B}{\rm c}^2_{K_s,\pa B}\,\zeta^2\,d\H\,,
  \\\label{vv3}
  \de^2 V_\a(B)[X]&=&-\iint_{\partial B\times\pa B}\frac{|\zeta(x)-\zeta(y)|^2}{|x-y|^{n-\a}}\,d\H_x\,d\H_y
  +\int_{\pa B}{\rm c}^2_{G_\a,\pa B}\,\zeta^2\,d\H\,,
\end{eqnarray}
for every $X$ inducing a volume-preserving flow on $B$ (here, $\zeta=X\cdot\nu_B$). The reason why we are able to discuss the volume-constrained stability of $\per_s+\beta\,V_\a$ at $B$ is that the Sobolev semi-norms $[u]_{H^1(\pa B)}$, $[u]_{H^{(1+s)/2}(\pa B)}$, and $[u]_{H^{(1-\a)/2}(\pa B)}$, can all be decomposed in terms of the Fourier coefficients of $u$ with respect to a orthonormal basis of spherical harmonics.

Indeed, recalling our notation $\{Y_k^i\}_{i=1}^{d(k)}$ for an orthonormal basis in $L^2(\pa B)$ of the space $\mathcal{S}_k$ of spherical harmonics of degree $k$, we have proved in \eqref{fou1} that
\begin{equation}\label{fourier s}
\iint_{\partial B\times\pa B}\frac{|u(x)-u(y)|^2}{|x-y|^{n+s}}\,d\H_x\,d\H_y=\sum_{k=0}^\infty\sum_{i=1}^{d(k)}\,\l_k^s\,a_k^i(u)^2\,,
\end{equation}
where $a_k^i(u)=\int_{\pa B}u\,Y_k^i\,d\H$.  Similarly, it is well-known that
\begin{equation}
 \label{decompDir}
 \int_{\pa B} |\nabla_\tau u|^2\,d\mathcal{H}^{n-1}= \sum_{k=0}^\infty\sum_{i=1}^{d(k)}\,\l_k^1\,a_k^i(u)^2\,,
 \end{equation}
with $\lambda_k^1$ defined as in \eqref{eigvLapBel}; see, for example, \cite{M}. We finally claim that for every $\a\in(0,n)$ we have
\begin{equation}\label{fourier alpha}
\iint_{\pa B\times\pa B}
\frac{|u(x)-u(y)|^2}{|x-y|^{n-\alpha}}\,d\mathcal{H}^{n-1}_x\mathcal{H}^{n-1}_y=\sum_{k=0}^\infty\sum_{i=1}^{d(k)}\,\mu_k^\a\,a_k^i(u)^2\,,
\end{equation}
for $\mu_k^\a$ defined as in \eqref{formmualpha<1}, \eqref{formmualpha>1}, and \eqref{formmualpha=1}. Indeed, following \cite[p. 151]{Sa}, one defines the {\it Riesz operator on the sphere of order $\gamma\in(0,n-1)$} as
\[
\mathcal{R}^\gamma u(x):=\frac{1}{2^\g\,\pi^{\frac{n-1}{2}}}\,\frac{\G(\frac{n-1-\gamma}2)}{\G(\frac\gamma2)}\,\int_{\pa B}\frac{u(y)}{|x-y|^{n-1-\gamma}}\,d\H_y\,,\qquad x\in\Sp\,.
\]
By \cite[Lemma 6.14]{Sa}, the $k$-th eigenvalue of ${\mathcal R}^\gamma$ is given by
\begin{equation}
 \label{se1R}
\mu_k^*(\gamma)=\frac{\G(k+\frac{n-1-\gamma}2)}{\G(k+\frac{n-1+\gamma}2)} \,,\qquad k\in\N\cup\{0\}\,,
\end{equation}
so that {$\mu_k^*(\gamma)> 0$, $\mu_k^*(\gamma)$ is strictly decreasing in $k$, and $\mu^*_k(\gamma)\downarrow 0$ as $k\to\infty$}. Moreover
\begin{equation}
 \label{se2R}
{\mathcal R}^\gamma Y_k=\mu_k^*(\gamma)\,Y_k\,,\qquad\forall k\in\N\cup\{0\}\,,
\end{equation}
where $Y_k$ denotes a generic spherical harmonic of degree $k$. In particular
\begin{equation}\label{valueinteg}
\frac{1}{2^\g\,\pi^{\frac{n-1}{2}}}\,\frac{\G(\frac{n-1-\gamma}2)}{\G(\frac\gamma2)}\,\int_{\pa B}\frac{d\H_y}{|x-y|^{n-1-\gamma}}=\mu^*_0(\gamma)\,\qquad\text{for every $x\in\pa B$}\,.
\end{equation}
Next, similarly to what we have done in section \ref{section fuglede}, we introduce for every $\alpha\in (0,n)$ the operator
$$
{\mathscr R}_\alpha u(x):=2\int_{\Sp}\frac{u(x)-u(y)}{|x-y|^{n-\alpha}}\,d\H_y\,,\qquad u\in C^1(\Sp)\,,
$$
so that, for every $u\in C^1(\pa B)$,
\begin{equation}\label{lb1R}
[u]_{\frac{1-\alpha}{2}}^2= \iint_{\pa B\times\pa B}\frac{|u(x)-u(y)|^2}{|x-y|^{n-\alpha}}\,d\mathcal{H}^{n-1}_x\mathcal{H}^{n-1}_y=\int_{\pa B}u\,{\mathscr R}_\alpha u\,d\H\,.
\end{equation}
If $\alpha\in(1,n)$ then $\g=\a-1\in(0,n-1)$, and thus we can deduce from \eqref{valueinteg} and \eqref{lb1R} that
\[
{\mathscr R}_\alpha=2^{\alpha}\,\pi^{\frac{n-1}{2}}\,\frac{\G(\frac{\alpha-1}2)}{\G(\frac{n-\alpha}2)}\,
\Big(\mu_0^*(\alpha-1){\rm Id}-{\mathcal R}^{\alpha-1}\Big)\,,\qquad \a\in(1,n)\,.
\]
In particular, we deduce from \eqref{se1R} and \eqref{se2R} that \eqref{fourier alpha} holds true with $\mu_k^\a$ defined as in \eqref{formmualpha>1} whenever $\a\in(1,n)$. If $\a\in(0,1)$, then ${\mathscr R}_\alpha$ becomes singular and by applying \eqref{defD} with $\g=1-\a\in(0,1)$ we have
$$
{\mathscr R}_\alpha =\frac{2^{1+\alpha}\pi^{\frac{n-1}{2}}}{1-\alpha}\,\frac{\G(\frac{1+\alpha}{2})}{\G(\frac{n-\alpha}{2})}\,\mathcal{D}^{1-\alpha}\,,\qquad \a\in(0,1)\,.
$$
In particular, it follows from \eqref{se1} and \eqref{se2} that \eqref{fourier alpha} holds true with $\mu_k^\a$ defined as in \eqref{formmualpha<1}. Finally, to prove \eqref{fourier alpha} in the case $\a=1$, it just suffice to notice that ${\mathscr R}_\alpha Y\to{\mathscr R}_1 Y$ as $\a\to 1$ for every spherical harmonic $Y$: therefore the eigenvalue $\mu_k^1$ of ${\mathscr R}_1$ can be simply computed by taking the limit of $\mu_k^\a$ as $\a\to 1^+$ in \eqref{formmualpha>1} or as $\a\to 1^-$ in \eqref{formmualpha<1}. In both ways one verifies the validity of \eqref{fourier alpha} with $\a=1$ and with $\mu_k^1$ defined as in \eqref{formmualpha=1}.

As a last preparatory remark to the proof of Theorem \ref{thmstability}, let us notice that by \eqref{formmualpha<1}, \eqref{formmualpha>1}, and \eqref{formmualpha=1} (and by exploiting some classical properties of the Gamma and digamma functions), one has
\begin{equation}
  \label{okR}
  \mu_0^\alpha=0\,,\qquad \mu_{k+1}^\alpha>\mu_k^\alpha\,,\qquad {\mathscr R}_\alpha Y_k=\mu_k^\alpha\, Y_k\,,\qquad\forall k\in\N\cup\{0\}\,,\quad\forall \alpha\in(0,n)\,.
\end{equation}
In addition, $\{\mu_k^\alpha\}$ is bounded for $\alpha\in(1,n)$, and  $\mu_k^\alpha\uparrow\infty$ as $k\to\infty$ for $\alpha\in(0,1]$. Finally, we notice that since the coordinate functions $x_i$, $i=1,\dots,n$, belong to ${\mathcal S}_1$, we have ${\mathscr R}_\alpha x_i=\mu_1^\alpha x_i$ by \eqref{okR}.  Inserting $x_i$ in \eqref{lb1R} and adding up over $i$, yields
 \begin{equation}\label{millot2V}
 \mu_1^\alpha=\frac{1}{P(B)}\iint_{\Sp\times\Sp}\frac{d\H_x\,d\H_y}{|x-y|^{n-2-\alpha}}=\int_{\Sp}\frac{\,d\H_y}{|z-y|^{n-2-\alpha}}\,,\qquad \forall z\in\pa B\,.
 \end{equation}
We can thus conclude that
$$
  {\rm c}_{\pa B}^2=n-1\,,\qquad {\rm c}^2_{K_s,\pa B}=\l_1^s\,,\qquad {\rm c}^2_{V_\a,\pa B}=\mu_1^\a\,,
$$
for every $s\in(0,1)$ and $\a\in(0,n)$: indeed, the first identity is trivial, while the second and the third one follow from \eqref{nlspc}, \eqref{millot2}, and \eqref{millot2V}.

Starting from the above considerations, given $s\in(0,1]$ and $\a\in(0,n)$ we are led to consider the following quadratic functionals
\begin{eqnarray*}
\QP_1(u)&:=&\int_{\pa B}|\nabla_\tau u|^2\,d\mathcal{H}^{n-1}-(n-1)\int_{\pa B}u^2\,d\mathcal{H}^{n-1}\,,
\\
\QP_s(u)&:=&\frac{1-s}{\omega_{n-1}}\bigg(\iint_{\pa B\times\pa B}\frac{|u(x)-u(y)|^2}{|x-y|^{n+s}}\,d\mathcal{H}^{n-1}_x\mathcal{H}^{n-1}_y
- \lambda_1^s\int_{\pa B}u^2\,d\H\bigg)\,,
\\
\QV_\a(u)&:=&\iint_{\pa B\times\pa B}\frac{|u(x)-u(y)|^2}{|x-y|^{n-\alpha}}\,d\mathcal{H}^{n-1}_x\mathcal{H}^{n-1}_y
-\mu_1^\a\,\int_{\pa B}u^2\,d\mathcal{H}^{n-1}\,.
\end{eqnarray*}
We set
\[
\widetilde{H}^{\frac{1+s}2}(\pa B):=\Big\{u\in H^{\frac{1+s}2}(\pa B):\int_{\pa B}u\,d\H=0\Big\}\,,
\]
and notice the validity of the following proposition.

\begin{proposition}\label{FourQ}
If $s\in(0,1]$, $\alpha\in(0,n)$, and $\beta>0$, then
$$
\QP_s(u)-\beta\,\QV_\a(u)=
\begin{cases}
\displaystyle \sum_{k=2}^\infty\sum_{i=1}^{d(k)}\bigg(\frac{1-s}{\omega_{n-1}}(\lambda_k^s-\lambda_1^s)-\beta(\mu_k^\alpha-\mu_1^\alpha)\bigg)\,a_k^i(u)^2\,, & \text{if $s\in(0,1)$}\,,\\[12pt]
\displaystyle \sum_{k=2}^\infty\sum_{i=1}^{d(k)}\bigg((\lambda_k^1-\lambda_1^1)-\beta(\mu_k^\alpha-\mu_1^\alpha)\bigg)\,a_k^i(u)^2\,, & \text{if $s=1$}\,.
\end{cases}
$$
for every $u\in\widetilde{H}^{\frac{1+s}2}(\pa B)$. In particular, $\QP_s-\beta\,\QV_\a\ge0$ on $\widetilde{H}^{\frac{1+s}2}(\pa B)$ if and only if $\beta\in(0,\beta_\star]$.
\end{proposition}

\begin{proof}
This is immediate from the definition of $\beta_\star$ and from \eqref{fourier s}, \eqref{decompDir}, and \eqref{fourier alpha}, once one takes into account that $a_0(u)=0$ for every $u\in L^2(\pa B)$ with $\int_{\pa B}u\,d\H=0$. (Indeed, $\mathcal{S}_0$ is the space of constant functions on $\pa B$.)
\end{proof}

We premise a final lemma to the proof of Theorem \ref{thmstability}.

\begin{lemma}\label{prop:connectlem}
Given $n\ge 2$, there exist positive constants $C_0$ and $\delta_0$, depending on $n$ only, with the following property: If $v\in C^\infty(\pa  B)$ and $\|v\|_{C^1(\pa B)}\leq \delta_0$, then there exists $X\in C^\infty_c(\R^n;\R^n)$ such that
\begin{enumerate}
\item[(i)] $\Div X=0$ on  $B_2\setminus B_{1/2}$;
\item[(ii)] the flow $\Phi_t$ induced by $X$ satisfies $\Phi_1(x)=(1+v(x)) x$ for every $x\in \pa B$;
\item[(iii)] $\|X\cdot\nu_B-v\|_{C^1(\pa B)}\le C_0\,\|v\|^2_{C^1(\pa B)}$.
\end{enumerate}
If in addition $|\Phi_1(B)|=|B|$, then $|\Phi_t(B)|=|B|$ for every $t\in(-1,1)$.
\end{lemma}

\begin{proof}
Let $\chi:[0,\infty)\to [0,1]$ be a smooth cut-off function such that $\chi(r)=1$ for $r\in[1/2,2]$ and $\chi(r)=0$ for $r\in [0,1/4]\cup[3,\infty)$, and define $X\in C^\infty_c(\R^n;\R^n)$ by setting
\[
X(x)=\frac{\chi(|x|)}{n}\bigg(\Big(1+ v\Big(\frac{x}{|x|}\Big)\Big)^n-1\bigg) \frac{x}{|x|^n}\,,\qquad x\in\R^n\,.
\]
Direct computations show the validity of (i) and (iii) (the latter with a constant $C_0$ that depends on $\de_0$). Up to further decrease the value of $\de_0$ we can ensure that $\Phi_t$ is a diffeomorphism for every $|t|\le 1$. By a direct computation we see that
\[
\Phi_t(x)=\bigg(1+ t\big(\big(1+ v(x)\big)^n-1\big)\bigg)^{\frac{1}{n}}x\,,
\]
for every $x\in\pa B$ and $|t|\le1$. In particular, (ii) holds true. By \eqref{svf6.5} and by (i) we infer that
\[
\frac{d^2}{dt^2}|E_t|=\int_{\pa E_t}(\Div X)(X\cdot\nu_{ E_t})\,d\mathcal{H}^{n-1}=0\,\qquad\forall |t|\le 1\,,
\]
that is, $t\mapsto |E_t|$ is affine on $[-1,1]$. In particular, if $|E_1|=|B|=|E_0|$, then $|E_t|=|B|$ for every $t\in[-1,1]$.
\end{proof}

\begin{proof}
  [Proof of Theorem \ref{thmstability}] We fix $\beta>0$ and claim that $B$ is a volume-constrained stable set for $\per_s+\beta\,V_\a$ if and only if
  \begin{equation}
    \label{catino}
  \QP_s(u)-\beta\,\QV_\a(u)\ge0\,,\qquad\mbox{$\forall u\in C^\infty(\pa B)$ with $\int_{\pa B}u\,d\H=0$}\,;
  \end{equation}
  the theorem will then follow by a standard density argument and by Proposition \ref{FourQ}. By \eqref{vv1}, \eqref{vv2}, and \eqref{vv3}, we see that $B$ is a volume-constrained stable set for $\per_s+\beta\,V_\a$ if and only if
  \begin{eqnarray}
    \label{catino2}
  &&\QP_s(X\cdot\nu_B)-\beta\,\QV_\a(X\cdot\nu_B)\ge0\,,\qquad\mbox{$\forall X\in C^\infty_c(\R^n;\R^n)$ inducing }
  \\\nonumber
  &&\hspace{6cm}\mbox{ a volume-preserving flow on $B$}\,.
  \end{eqnarray}
  Now, the fact that \eqref{catino} implies \eqref{catino2} is obvious: indeed, recall \eqref{svf6.5}, $u=X\cdot\nu_B$ satisfies $\int_{\pa B}u\,d\H=0$ whenever $X$ induces a volume-preserving flow on $B$. To prove the reverse implication, let us fix $u\in C^\infty(\pa B)$ with $\int_{\pa B}u\,d\H=0$, and consider the open sets
  \[
  E_\de=\Big\{(1+\de\,u(x))\,x:x\in\pa B\Big\}\,,\qquad\de\in(0,1)\,.
  \]
  Since $\int_{\pa B}u\,d\H=0$, we have that $\bigl| |E_\de|-|B|\bigr| \leq C\de^2$ for some constant $C$ depending on $u$ only. Therefore, if $F_\de=(|B|/|E_\de|)^{1/n}\,E_\de$, then we have
  \[
  F_\de=\Big\{(1+v_\de(x))\,x:x\in\pa B\Big\}\,,\qquad\de\in(0,1)\,,
  \]
  for some $v_\de\in C^\infty(\pa B)$ with $\|v_\de\|_{C^1(\pa B)}\le C\,\de$ and $\|v_\de-\de\,u\|_{C^1(\pa B)}\le C\,\de^2$ (again, the constant $C$ does not depend on $\de$). Provided $\de$ is small enough we can apply Lemma \ref{prop:connectlem} to find a vector field $X_\de\in C^\infty_c(\R^n;\R^n)$ inducing a volume-preserving flow on $B$, and with the property that
  \[
  \|X_\de\cdot\nu_B-v_\de\|_{C^1(\pa B)}\le C\,\|v_\de\|_{C^1(\pa B)}^2\le C\,\de^2\,.
  \]
  In particular, $\|X_\de\cdot\nu_B-\de\,u\|_{C^1(\pa B)}\le C\,\de^2$, and thus by \eqref{catino2}  we have (recall that $\QP$ and $\QV$ are quadratic forms)
  \[
  0\le \QP_s(X_\de\cdot\nu_B)-\beta\,\QV_\a(X_\de\cdot\nu_B)\le \QP_s(\de\,u)-\beta\,\QV_\a(\de\,u)+C\,\de^3\,.
  \]
  We divide by $\de^2$ and let $\de\to 0^+$ to find that $\QP_s(u)-\beta\,\QV_\a(u)\ge 0$. This shows that \eqref{catino2} implies \eqref{catino}, and thus completes the proof of the theorem.
\end{proof}

We close this section with the following result.

\begin{proposition}\label{proposition betastar} For every $n\ge 2$, $s\in(0,1]$ and $\a\in(0,n)$ one has
\begin{equation}
  \label{betastar}
  \beta_\star(n,s,\a)=
  \begin{cases}
  \displaystyle \frac{n+s}{n-\a}\,\frac{s\,(1-s)\,P_s(B)}{\a\,\omega_{n-1}V_\a(B)}\,, &\hspace{1cm} \text{if $s\in(0,1)$}\,,\\[10pt]
  \displaystyle \frac{n+1}{n-\a}\,\frac{P(B)}{\a\,V_\a(B)}\,, &\hspace{1cm} \text{if $s=1$}\,.
  \end{cases}
\end{equation}
\end{proposition}

\begin{proof}
  By appendix \ref{section betastar}
  \begin{equation}\nonumber
  \beta_\star(n,s,\alpha)=
  \begin{cases}
  \displaystyle (1-s)\,\frac{\lambda_2^s-\lambda_1^s}{\mu_2^\alpha-\mu_1^\alpha}\,, &\hspace{1cm} \text{if $s\in(0,1)$}\,,\\[10pt]
  \displaystyle \frac{\lambda_2^1-\lambda_1^1}{\mu_2^\alpha-\mu_1^\alpha}\,, &\hspace{1cm} \text{if $s=1$}\,,
  \end{cases}
  \end{equation}
  We then find \eqref{betastar} by Proposition \ref{millot} and by Proposition \ref{millotV} below.
\end{proof}


\section{Proof of Theorem \ref{thm 3}}\label{section stability implies local minimality} We are now in the position of proving Theorem \ref{thm 3}. We begin with the following result, which extends Theorem \ref{fuglede} to the family of functionals $\per_s+\beta\,V_\a$ with $\beta\in(0,\beta_\star)$.

\begin{theorem}\label{fugledeV}
For every $s\in(0,1)$, $\alpha\in(0,n)$, and $\beta\in(0,\beta_\star(n,s,\a)),$ there exist positive constants $c_0=c_0(n)$ and $\e_\beta=\e_\b(n,s,\a)$ with the following property: If $E$ is a nearly spherical set as in \eqref{nearly} with $|E|=|B|$, $\int_Ex\,dx=0$, and $\|u\|_{C^1(\Sp)}<\e_\beta$, then
  \begin{equation}\label{fugbeta}
\big({\rm Per}_s+\beta V_\alpha\big)(E)-\big({\rm Per}_s+\beta V_\alpha\big)(B)\geq c_0\Big(1-\frac{\beta}{\beta_\star}\Big)\,\Big((1-s)[u]_{\frac{1+s}{2}}^2+\|u\|_{L^2(\pa B)}^2\Big)\,.
 \end{equation}
Moreover, we can take $\e_\beta$ of the form
\begin{equation}
  \label{eps beta forma}
  \e_\beta=\Big(1-\frac{\beta}{\beta_\star}\Big)\,\e_0(n)\,,
\end{equation}
for a suitable positive constant $\e_0(n)$.
\end{theorem}

\begin{remark} If $\beta\in(0,\beta_\star(n,1,\a))$ and $u$ satisfies the assumptions of Theorem \ref{fugledeV}, then
\begin{equation}\label{fugbetafors=1}
\big(P+\beta V_\alpha\big)(E)-\big(P +\beta V_\alpha\big)(B)\geq c_0\Big(1-\frac{\beta}{\beta_\star(n,1,\a)}\Big)\, \|u\|_{H^1(\pa B)}^2\,.
\end{equation}
To prove this observe that, by a standard approximation argument, it suffices to consider the case when $u\in C^{1,\g}(\pa B)$ for some $\g\in(0,1)$, and thus $\per_s(E)\to P(E)$ as $s\to 1^-$ by \eqref{limit s to 1}. By \eqref{betastar} and again by \eqref{limit s to 1}, $\beta_\star(n,s,\a)\to\beta_\star(n,1,\a)$ as $s\to 1^-$. In particular, we can find $\tau>0$ such that $\,\beta<\beta_\star(n,s,\a)$ and $\e_\beta(n,1,\a)<\e_{\beta}(n,s,\a)$ for every $s\in(1-\tau,1)$. We may thus apply \eqref{fugbeta} with $s\in(1-\tau,1)$ and then let $\tau\to 0^+$, to find that
\[
\big(P+\beta V_\alpha\big)(E)-\big(P +\beta V_\alpha\big)(B)\geq c_0\Big(1-\frac{\beta}{\beta_\star(n,1,\a)}\Big)\,\limsup_{s\to 1^{-1}}\,(1-s)[u]_{\frac{1+s}{2}}^2\,.
\]
Finally, by \eqref{eigvLapBel} and \eqref{def lambda k s} we find that $\lambda_k^s\to\omega_{n-1}\lambda_k^1$ as $s\to 1^-$, hence recalling \eqref{fourier s}
and \eqref{decompDir} we get
\begin{equation}
\label{eq:norms}
\lim_{s\to 1^-}\,(1-s)[u]^2_{\frac{1+s}{2}}=\omega_{n-1}\int_{\pa B}|\nabla_\tau u|^2\,
\end{equation}
and  \eqref{fugbetafors=1} is proved.
\end{remark}

\begin{remark} Theorem \ref{fuglede} follows from Theorem \ref{fugledeV} by letting $\a\to n^-$ in \eqref{fugbeta}. Indeed, denoting by $C$ a generic constant depending on $n$ only, we notice that \eqref{formmualpha<1}, \eqref{formmualpha>1}, and \eqref{formmualpha=1} give $\mu_k^\a-\mu_1^\a\leq C\,(n-\a)$ for all $k\geq 2$. At the same time, by exploiting  \eqref{se1}, \eqref{lb000}, and \eqref{ok} we find that
$$
  (1-s)\,\l_1^s\ge \frac1{C}\,,\qquad\forall s\in(0,1)\,,
$$
so that by Proposition \ref{millot}, again for every $k\ge 2$,
$$
(1-s)(\lambda_k^s-\lambda_1^s)\geq (1-s)(\lambda_2^s-\lambda_1^s)=\frac{n+s}{n-s} (1-s)\lambda_1^s\geq \frac1{C}\,.
$$
We thus conclude from \eqref{defbetstar} that
\[
\beta_\star(n,s,\a)\geq \frac{c(n)}{n-\a}\,,
\]
for a suitable positive constant $c(n)$. In particular, $\beta_\star(n,s,\a)\to \infty$ as $\a\to n^-$ uniformly with respect to $s\in (0,1)$, and \eqref{fug0} follows by letting $\a\to n^-$ in \eqref{fugbeta}.
\end{remark}

Before discussing the proof of Theorem \ref{fugledeV} we need the following observation, which parallels Proposition \ref{millot}.

\begin{proposition}\label{millotV}
For every $\alpha\in (0,n)$, one has
\begin{eqnarray}
\label{millot1V}
 \mu_1^\alpha&=&\alpha(n+\alpha)\frac{V_\alpha(B)}{P(B)}\,,
 \\
 \label{remeige1V}
 \mu_2^\alpha&=&\frac{2n}{n+\alpha}\,\mu_1^\alpha \,.
\end{eqnarray}
\end{proposition}

\begin{proof}
 By scaling, $V_\alpha(B_r)=r^{n+\alpha}V_\alpha(B)$. Hence,
 \[
 (n+\alpha)V_\alpha(B)=\frac{d}{dr}\Big|_{r=1}V_\alpha(B_r)=2\int_{B}\,dx\int_{\pa B}\frac{d\H_y}{|x-y|^{n-\alpha}}\,.
 \]
 Since
 $$
\frac{1}{|x-y|^{n-\alpha}}=\frac{1}{\alpha} \,{\rm div}_x\biggl(\frac{x-y}{|x-y|^{n-\alpha}}\biggr)
 $$
 by the divergence theorem we get
 $$
 \alpha (n+\alpha)V_\alpha(B)=2\iint_{\Sp\times\Sp}\frac{(x-y)\cdot x}{|x-y|^{n-\alpha}}\,d\H_x\,d\H_y\,.
 $$
 By symmetry, the right-hand side of the last identity is equal to
\begin{eqnarray*}
 &&\iint_{\Sp\times\Sp}\frac{(x-y)\cdot x}{|x-y|^{n+s}}\,d\H_x\,d\H_y
 + \iint_{\Sp\times\Sp}\frac{(y-x)\cdot y}{|x-y|^{n+s}} \,d\H_y\,d\H_x
 \\
 &&=\iint_{\Sp\times\Sp}\frac{d\H_x\,d\H_y}{|x-y|^{n+s-2}}\,,
\end{eqnarray*}
so that  \eqref{millot1V} follows from \eqref{millot2V}. One can deduce \eqref{remeige1V} from \eqref{formmualpha<1}, \eqref{formmualpha=1}, and \eqref{formmualpha>1} (depending on whether $\a\in(1,n)$, $\a=1$ or $\a\in(0,1)$) by exploiting the factorial property of the Gamma function. Since a similar argument was presented in Proposition \ref{millot}, we omit the details.
\end{proof}

\begin{proof}[Proof of Theorem~\ref{fugledeV}]
We consider $u\in C^1(\pa B)$ with  $\|u\|_{C^1(\pa B)}\le 1/2$ and assume the existence of $t\in(0,2\,\e_\b)$ such that the open set $E_t$ whose boundary is given by
$$
 \pa E_t=\Big\{(1+t\,u(x))\,x:x\in\pa B\Big\}
$$
satisfies $|E_t|=|B|$ and $\int_{E_t}x\,dx=0$. If $\e_\beta$ is small enough then \eqref{ober2}, \eqref{quantpotinterm}, and \eqref{millot1V} imply that
\begin{multline}\label{devinfenerg}
\big({\rm Per}_s+\beta V_\alpha\big)(E_t)-\big({\rm Per}_s+\beta V_\alpha\big)(B)\geq \frac{t^2}{2}\Big(\mathcal{QP}_s(u)-\beta\mathcal{QV}_\a(u)\Big)\\
-C(n)t^3\bigg(\frac{1-s}{\omega_{n-1}}\big([u]^2_{\frac{1+s}{2}}+\lambda_1^s\|u\|^2_{L^2}\big)+\beta\big([u]^2_{\frac{1-\a}{2}}+\mu_1^\a\|u\|^2_{L^2}\big)\bigg)\,.
\end{multline}
By Proposition \ref{FourQ} and by definition of $\beta_\star$ we have
\begin{align*}
 \mathcal{QP}_s(u)-\beta\mathcal{QV}_\a(u) & =  \sum_{k=2}^\infty\sum_{i=1}^{d(k)}\bigg(\frac{1-s}{\omega_{n-1}}(\lambda_k^s-\lambda_1^s)-\beta(\mu_k^\alpha-\mu_1^\alpha)\bigg)\,|a_k^i|^2 \\
& \geq  \frac{1-s}{\omega_{n-1}}\Big(1-\frac{\beta}{\beta_\star} \Big)\sum_{k=2}^\infty\sum_{i=1}^{d(k)}(\lambda_k^s-\lambda_1^s)|a_k^i|^2\\
& = \frac{1-s}{\omega_{n-1}}\Big(1-\frac{\beta}{\beta_\star} \Big) \Big([u]^2_{\frac{1+s}{2}} - \lambda_1^s\|u\|^2_{L^2}\Big)\,,
\end{align*}
thus using \eqref{fug10} and \eqref{ober3} we find
\begin{equation}\label{devinfenerg2}
\mathcal{QP}_s(u)-\beta\mathcal{QV}_\a(u) \geq \frac{1-s}{4}\Big(1-\frac{\beta}{\beta_\star} \Big)\Big([u]^2_{\frac{1+s}{2}} + \lambda_1^s\|u\|^2_{L^2}\Big)\, .
\end{equation}
Choosing $\e_\beta$ small enough, we can apply \eqref{ober3} and \eqref{remeige1V} to estimate
\begin{equation}
  \label{lezione}
  \mu_1^\a \|u\|^2_{L^2}\leq 2\mu_1^\a \sum_{k=2}^\infty\sum_{i=1}^{d(k)}|a_k^i|^2\leq \frac{2(n+\a)}{n-\a} \sum_{k=2}^\infty\sum_{i=1}^{d(k)}(\mu_k^\a-\mu^\a_1)|a_k^i|^2\leq C(n)\mathcal{QV}_\a(u)\,,
\end{equation}
where in the last inequality we have used the temporary assumption that
\begin{equation}
  \label{alpha less}
  \a\le n-\frac12\,.
\end{equation}
By \eqref{lezione} and by \eqref{devinfenerg2} (which gives, in particular, $\mathcal{QP}_s(u)\ge\beta\mathcal{QV}_\a(u)$), we find
\begin{equation}\label{devinfenerg3}
\beta\Big([u]^2_{\frac{1-\a}{2}}+\mu_1^\a\|u\|^2_{L^2}\Big) =  \beta\mathcal{QV}_\a(u) +2\beta\mu_1^\alpha \|u\|^2_{L^2} \leq C(n)\beta\, \mathcal{QV}_\a(u) \leq C(n)  \mathcal{QP}_s(u)\,.
%
\end{equation}
By gathering \eqref{devinfenerg}, \eqref{devinfenerg2}, and \eqref{devinfenerg3} we end up with
$$
\big({\rm Per}_s+\beta V_\alpha\big)(E_t)-\big({\rm Per}_s+\beta V_\alpha\big)(B)\geq
\frac{1-s}{\omega_{n-1}}
\Big(\frac{t^2}{8}\Big(1-\frac{\beta}{\beta_\star} \Big)-C(n)t^3\Big)\,
\Big([u]^2_{\frac{1+s}{2}} +\l_1^s \|u\|^2_{L^2}\Big).
$$
By choosing $\e_0(n)$ suitably small in \eqref{eps beta forma}, and by exploiting  \eqref{se1}, \eqref{lb000}, and \eqref{ok} to deduce that $(1-s)\,\l_1^s\ge c(n)>0$ for a suitable positive constant $c(n)$, we deduce that
$$
\big({\rm Per}_s+\beta V_\alpha\big)(E_t)-\big({\rm Per}_s+\beta V_\alpha\big)(B)\geq
c_0 t^2\Big(1-\frac{\beta}{\beta_\star} \Big)\Big((1-s)[u]^2_{\frac{1+s}{2}} + \|u\|^2_{L^2}\Big)\,,
$$
for a constant $c_0$ which only depends on $n$.  This completes the proof of the theorem in the case \eqref{alpha less} holds true. Let us now assume that $\a\in(n-1/2,n)$, and prove a stronger version of \eqref{quantpotinterm}. Since $|E_t|=|B|$, we can write
$$
V_\a(B)-V_\a(E_t)=\big(V_\a(B)-|B|^2\big)-\big(V_\a(E_t)-|E_t|^2\big)\,.
$$
If we set
$$
f_{\theta}(r,\rho):=\frac{r^{n-1}\r^{n-1}}{(|r-\r|^2+r\,\r\,\theta^2)^{\frac{n-\a}{2}}}-r^{n-1}\rho^{n-1}\,,\qquad r,\rho,\theta\geq 0\,,
$$
then we find
\[
V_\a(E_t)-|E_t|^2=\iint_{\Sp\times\Sp}\biggl(\int_{0}^{1+tu(x)}\int_0^{1+tu(y)} f_{|x-y|}(r,\rho)  \,dr\,d\r\biggr)\,d\H_x\,d\H_y\,,
\]
Arguing as in the proof of Lemma \ref{lemma fuglede alpha}, we derive that
\begin{align*}
V_\a(E_t)-|E_t|^2=&-\frac{t^2}{2}\tilde g(t) +\frac{V_\a(B)}{P(B)}\int_{\pa B}(1+t u)^{n+\a}\,d\H -\frac{|B|^2}{P(B)}\int_{\pa B}(1+t u)^{2n}\,d\H\\
=& -\frac{t^2}{2}\tilde g(t) +\frac{V_\a(B)-|B|^2}{P(B)}\int_{\pa B}(1+t u)^{n+\a}\,d\H \\
&\hskip100pt-\frac{|B|^2}{P(B)}\int_{\pa B}(1+t u)^{2n}\Big(1-(1+t u)^{\a-n}\Big)\,d\H\,,
\end{align*}
with
$$\tilde g(t):=\iint_{\pa B\times\pa B} \left( \int_{u(y)}^{u(x)} \int_{u(y)}^{u(x)} f_{|x-y|}(1+tr,1+t\rho)\,drd\rho  \right)\,d\H_x\,d\H_y\,. $$
Setting $h(t):=\int_{\pa B}(1+tu)^{n+\a}$ and
$$\ell(t):= \int_{\pa B}(1+t u)^{2n}\Big(1-(1+t u)^{\a-n}\Big)\,d\H\,,$$
we conclude that
$$ V_\a(B)-V_\a(E_t)=\frac{t^2}{2}\tilde g(t) + \frac{V_\a(B)-|B|^2}{P(B)}\big(h(0)-h(t) \big) +\frac{|B|^2}{P(B)}\,\ell(t)\,.$$
In the proof of Lemma \ref{lemma fuglede alpha} we showed that
 \begin{equation}\label{23:25}
 h(0)-h(t)\le -\a\,(n+\a)\,\frac{t^2}2\,\int_{\pa B}u^2\,d\H+C(n)\,t^3\,\|u\|_{L^2}^2\,.
 \end{equation}
In the same way (using Taylor expansion and $|E_t|=|B|$) we obtain that
 \begin{equation}\label{23:26}
\ell(t)\leq (n-\a)(2n+\a)\frac{t^2}{2}\|u\|^2_{L^2}+(n-\a)C(n)t^3\|u\|^2_{L^2}\,.
 \end{equation}
Then, noticing that
$$\a(n+\a)=2n^2-(n-\a)(2n+\a) $$
and using  \eqref{millot1V}, we compute
\begin{multline}\label{23:27}
 \frac{V_\a(B)-|B|^2}{P(B)}=\frac{1}{\a(n+\a)}\Big(\mu_1^\a-\a(n+\a)\frac{|B|^2}{P(B)}\Big)\\
 =\frac{1}{\a(n+\a)}\Big(\mu_1^\a-2n^2\frac{|B|^2}{P(B)}\Big) +(n-\a)\frac{(2n+\a)|B|^2}{\a(n+\a)P(B)}\,.
 \end{multline}
On the other hand, \eqref{millot1V} implies
$$\mu_1^\a\,\mathop{\longrightarrow}\limits_{\a\to n}\, 2n^2\frac{|B|^2}{P(B)}=:\mu_1^n \,.$$
From the explicit value of $\mu_1^\a$ given by \eqref{formmualpha>1}, we easily infer that $|\mu_1^\a-\mu_1^n|\leq (n-\a)C(n)$. Hence,
\begin{equation}\label{23:28}
 \Big| \frac{V_\a(B)-|B|^2}{P(B)}\Big|\leq (n-\a)C(n)\,.
 \end{equation}
 Gathering \eqref{23:25}, \eqref{23:26}, \eqref{23:27}, and \eqref{23:28}, we are led to
 $$\frac{V_\a(B)-|B|^2}{P(B)}\big(h(0)-h(t) \big) +\frac{|B|^2}{P(B)}\,\ell(t)\leq -(\mu_1^\a-\mu_1^n)\frac{t^2}{2}\|u\|^2_{L^2} +(n-\a)C(n)t^3\|u\|^2_{L^2}\,. $$
  Next, from the smooth dependence $\tilde g$ on $t$, we can find $\tau\in(0,t)$ such that $\tilde g(t)=\tilde g(0)+t\,\tilde g'(\tau)$.
 Since $\a\in (n-1/2,n)$, we have the estimate
 $$
 \Big|r\,\frac{\pa f_\theta}{\pa r}(1+\tau\,r,1+\tau\,\r)+\r\,\frac{\pa f_\theta}{\pa\r}(1+\tau\,r,1+\tau\,\r)\Big|\le (n-\a)\frac{C(n)}{\theta^{n-\a}}\big(1+|\log(\theta)|\big)
 \le (n-\a)\frac{C(n)}{\theta^{3/4}}\,,
 $$
for all  $r,\r\in(-\frac12,\frac12)$, all $\theta\in(0,2]$, and a suitable constant $C(n)$.  In turn, the sequence $\{\mu_k^{n-3/4}\}$ is bounded and one can estimate
 \[
 |g'(\tau)|\le (n-\a)C(n)\iint_{\pa B\times\pa B}\frac{|u(x)-u(y)|^2}{|x-y|^{3/4}}\,d\H_x\,d\H_y\leq (n-\a)C(n)\,\|u\|_{L^2}^2\,,
 \]
therefore
$$V_\a(B)-V_\a(E_t)\leq\frac{t^2}{2}\tilde g(0) -(\mu_1^\a-\mu_1^n)\frac{t^2}{2}\|u\|^2_{L^2}+(n-\a)C(n)t^3\|u\|^2_{L^2}\,. $$
Then, we notice that
$$\tilde g(0)=[u]^2_{\frac{1-\a}{2}} -\iint_{\pa B\times\pa B}|u(x)-u(y)|^2\,d\H_x\,d\H_y\,. $$
Also, from  \eqref{formmualpha>1} we infer that
$$\lim_{\a\to n}\mu_k^\a=\mu_1^n\,,\qquad \forall k\geq 1\,. $$
Hence, by dominated convergence we have
$$[u]^2_{\frac{1-\a}{2}}=\sum_{k=1}^\infty\sum_{i=1}^{d(k)}\mu_k^\a |a_k^i|^2\mathop{\longrightarrow}\limits_{\a\to n} \mu_1^n\sum_{k=1}^\infty\sum_{i=1}^{d(k)}|a_k^i|^2\,.$$
Since we obviously have
$$[u]^2_{\frac{1-\a}{2}} \mathop{\longrightarrow}\limits_{\a\to n} \iint_{\pa B\times\pa B}|u(x)-u(y)|^2\,d\H_x\,d\H_y\,,$$
we have thus proved that
$$ \iint_{\pa B\times\pa B}|u(x)-u(y)|^2\,d\H_x\,d\H_y =  \mu_1^n\sum_{k=1}^\infty\sum_{i=1}^{d(k)}|a_k^i|^2\,.$$
As a consequence,
$$V_\a(B)-V_\a(E_t)\leq \frac{t^2}{2} \sum_{k=2}^\infty\sum_{i=1}^{d(k)}(\mu_k^\a-\mu_1^\a) |a_k^i|^2 - (\mu_1^\a-\mu_1^n)\frac{t^2}{2}|a_0|^2+(n-\a)C(n)t^3\|u\|^2_{L^2}\,. $$
Recalling \eqref{fug11} and the fact that $|\mu_1^\a-\mu_1^n|\leq (n-\a)C(n)$, we conclude that
 \begin{equation}\label{finalasymptV}
 V_\a(B)-V_\a(E_t)\leq \frac{t^2}{2}\mathcal{QV}_\a(u)+ (n-\a)C(n)t^3\|u\|^2_{L^2}\,,
\end{equation}
that is the required strengthening of \eqref{quantpotinterm}. Now we can apply \eqref{ober2} together with \eqref{finalasymptV} to find that
\begin{multline}\label{devinfenergbis}
\big({\rm Per}_s+\beta V_\alpha\big)(E_t)-\big({\rm Per}_s+\beta V_\alpha\big)(B)\geq \frac{t^2}{2}\Big(\mathcal{QP}_s(u)-\beta\mathcal{QV}_\a(u)\Big)\\
-C(n)t^3\bigg(\frac{1-s}{\omega_{n-1}}\big([u]^2_{\frac{1+s}{2}}+\lambda_1^s\|u\|^2_{L^2}\big)+(n-\a)\beta\|u\|^2_{L^2}\bigg)\,.
\end{multline}
Arguing as in the previous case, it yields
\begin{multline*}
\big({\rm Per}_s+\beta V_\alpha\big)(E_t)-\big({\rm Per}_s+\beta V_\alpha\big)(B)\geq
\frac{t^2}{8}\Big(1-\frac{\beta}{\beta_\star} \Big)\bigg(\frac{1-s}{\omega_{n-1}}[u]^2_{\frac{1+s}{2}} + (1-s)\lambda_1^s\|u\|^2_{L^2}\bigg) \\
-C(n)t^3\bigg(\frac{1-s}{\omega_{n-1}}[u]^2_{\frac{1+s}{2}}+(1-s)\lambda_1^s\|u\|^2_{L^2}+(n-\a)\beta_\star\|u\|^2_{L^2}\bigg)\,.
\end{multline*}
 Since $(n-\a)\beta_\star\leq C(n)$ by \eqref{betastar}, we conclude as in the previous case.
\end{proof}

As a last tool in the proof of Theorem \ref{thm 3} we prove the following lemma.

\begin{lemma}\label{lemma infine}
  Let $s\in(0,1]$ and $\a\in(0,n)$. If $\beta<\beta_\star$, then $B$ is a local volume-constrained minimizer of $\per_s+\beta\,V_\a$. If $\beta>\beta_\star$, then $B$ is not a local volume-constrained minimizer of $\per_s+\beta\,V_\a$.
\end{lemma}

\begin{proof}
  If $B$ is a local volume-constrained minimizer of $\per_s+\beta\,V_\a$, then $B$ is automatically a volume-constrained stable set for $\per_s+\beta\,V_\a$, and thus $\beta\le\beta_\star$ by Theorem \ref{thmstability}. We are thus left to prove that if $\beta<\beta_\star$, then $B$ is a local volume-constrained minimizer of $\per_s+\beta\,V_\a$. To this end, we argue by contradiction and assume the existence of some $\beta<\beta_*$ such that there exists a sequence $\{E_h\}_{h\in\N}$ with
  \begin{equation}
    \label{14 Eh}
    |E_h|=|B|\,,\qquad \lim_{h\to\infty}|E_h\Delta B|=0\,,\qquad \per_s(E_h)+\beta\,V_\a(E_h)<\per_s(B)+\beta\,V_\a(B)\,,\quad\forall h\in\N\,.
  \end{equation}
  We divide the proof in two steps.

  \medskip

  \noindent {\it Step one}: We show the existence of a radius $R>0$ (depending on $n$, $s$ and $\a$ only) such that the sequence $E_h$ in \eqref{14 Eh} can actually be assumed to satisfy the additional constraint
  \begin{equation}
    \label{14 R}
    E_h\subset B_R\,,\qquad\forall h\in\N\,.
  \end{equation}
  To show this, let us introduce a parameter $\eta<1$ (whose precise value will be chosen shortly depending on $n$, $s$ and $\a$) and let us assume without loss of generality and thanks to \eqref{14 Eh} that $|E_h\Delta B|<\eta$ for every $h\in\N$. By Lemma \ref{lemma truncation}, see in particular \eqref{tr1}, there exists a sequence $\{r_h\}_{h\in\N}$ with $1\leq r_h\le 1+ C_1\,\eta^{1/n}$ such that
  \begin{equation}
   \label{14 tr1x}
  \per_s(E_h\cap B_{r_h})\le \per_s(E_h)-\frac{|E_h\setminus B_{r_h}|}{C_2\,\eta^{1/n}}\,,
  \end{equation}
  where $C_1$ and $C_2$ depend on $n$ and $s$ only. Next, we consider $\mu_h>0$ such that $F_h:=\mu_h\, (E_h\cap B_{r_h})$ satisfies $|F_h|=|B|$. Since $|E_h\Delta B|\to 0$ as $h\to\infty$, it must be that $\mu_h\to 1$ and $|F_h\Delta B|\to 0$ as $h\to\infty$. In particular, we can assume without loss of generality that $F_h\subset B_R$ for every $h\in\N$, provided we set $R:=2+ C_1\,\eta^{1/n}$. We finally show that
  \begin{equation}
    \label{14 neve}
      \per_s(F_h)+\beta\,V_\a(F_h)\le\per_s(E_h)+\beta\,V_\a(E_h)\,.
  \end{equation}
  Indeed, by setting $u_h:=|E_h\setminus B_{r_h}|$ we find that
  \begin{eqnarray*}
    \per_s(F_h)+\beta\,V_\a(F_h)&=&\mu_h^{n-s}\per_s(E_h\cap B_{r_h})+\mu_h^{n+\a}\beta\,V_\a(E_h\cap B_{r_h})
    \\
    &\le&(1+C\,u_h)\,\Big(\per_s(E_h\cap B_{r_h})+\beta\,V_\a(E_h\cap B_{r_h})\Big)\,,
  \end{eqnarray*}
  where $C=C(n,s,\a)$. By $V_\a(E_h\cap B_{r_h})\le V_\a(E_h)$, \eqref{14 tr1x}, and \eqref{14 Eh}, we conclude that
  \begin{eqnarray*}
    \per_s(F_h)+\beta\,V_\a(F_h)&\le&\per_s(E_h)+\beta\,V_\a(E_h)
    \\
    &&+\bigg(C\,\Big(\per_s(B)+\beta_\star\,V_\a(B)\Big)-\frac1{C_2\,\eta^{1/n}}\bigg)\,u_h\,,
  \end{eqnarray*}
  so that \eqref{14 neve} follows provided $\eta$ was suitably chosen in terms of $n$, $s$ and $\a$ only.

  \medskip

  \noindent {\it Step two:} Given $M>0$ and a sequence $E_h$ satisfying \eqref{14 Eh} and \eqref{14 R} we now consider the variational problems
  \begin{equation}
    \label{14 problem}
      \g_h:=\inf\bigg\{\per_s(E)+\beta\,V_\a(E)+M\,|E\Delta E_h|:E\subset\R^n\bigg\}\,,\qquad h\in\N\,,
  \end{equation}
  and prove the existence of minimizers. Indeed, if $R$ is as in \eqref{14 R}, then $V_\a(E\cap B_R)\le V_\a(E)$ by set inclusion, $\per_s(E\cap B_R)\le \per_s(E)$ by Lemma \ref{lemma convessi}, while, if we set $F=E\cap B_{R}$, then by \eqref{14 R},
  \begin{eqnarray}\nonumber
  |F\Delta E_h|&=&|F\setminus E_h|+|E_h\setminus F|
  \le|E\setminus E_h|+|(E_h\cap B_{R})\setminus F|+|E_h\setminus B_{R}|
  \\\nonumber
  &=&|E\setminus E_h|+|(E_h\cap B_{R})\setminus E|
  \\\label{14 giovanni}
  &\le&|E\Delta E_h|\,.\nonumber
  \end{eqnarray}
  Thus the value of $\g_h$ is not changed if we restrict the minimization class by imposing $E\subset B_R$. By the Direct Method, there exists a minimizer $F_h$ in \eqref{14 problem} for every $h\in\N$, with $F_h\subset B_R$. We now claim that there exists $\Lambda>0$ such that
  \begin{equation}
    \label{14 giovanni h}
      \per_s(F_h)\le\per_s(E)+\Lambda\,|E\Delta F_h|\,,\qquad\forall E\subset\R^n\,,
  \end{equation}
  for every $h\in\N$. Indeed, by minimality of $F_h$ in \eqref{14 problem} and by \eqref{rigot}, we find that for every bounded set $E\subset\R^n$ one has
  \begin{eqnarray*}
    \per_s(F_h)-\per_s(E)\le\frac{2\,P(B)\,\beta}\a\,\Big(\frac{|E|}{|B|}\Big)^{\a/n}\,|E\setminus F_h|
    +M\,\Big(|E\Delta E_h|-|F_h\Delta E_h|\Big)\,.
  \end{eqnarray*}
  In particular, \eqref{14 giovanni h} follows provided
  \begin{equation}
    \label{14 lambda1}
      \Lambda\ge\frac{2^{1+\a}\,P(B)\,\beta\,R^\a}\a+M\,,
  \end{equation}
  whenever $|E|\le |B_{2R}|$. To address the complementary case, we just notice that, setting for the sake of brevity $\F:=\per_s+\beta\,V_\a$, by \eqref{14 Eh} and by minimality of $F_h$ one has
  \begin{eqnarray}\label{14 giovanni2}
  \F(B)>\F(E_h)\ge\F(F_h)+M\,|F_h\Delta E_h|\,.
  \end{eqnarray}
  In particular $\per_s(F_h)\le\F(B)$ for every $h\in\N$. Hence, if $|E|\ge|B_{2R}|$ then by $F_h\subset B_R$ we have
  \[
  \per_s(E)+\Lambda\,|E\Delta F_h|\ge\Lambda(|E|-|F_h|)\ge \Lambda\,|B|(2^n-1)\,R^n\ge\per_s(F_h)\,,
  \]
  provided
  \begin{equation}
    \label{14 lambda2}
      \Lambda\ge\frac{\F(B)}{|B|(2^n-1)\,R^n}\,.
  \end{equation}
  We choose $\Lambda$ to be the maximum between the right-hand sides of \eqref{14 lambda1} and \eqref{14 lambda2}, and in this way \eqref{14 giovanni h} is proved. We now notice that by \eqref{14 giovanni2}, \eqref{14 Eh}, and up to discard finitely many $h$'s, we can assume that
  \begin{equation}
    \label{14 giovanni freddo}
      |F_h\Delta B|\le\frac{2\,\F(B)}M\,,\qquad\forall h\in\N\,.
  \end{equation}
  Let now $\e_\b$ be defined as in Theorem \ref{fugledeV}. By Corollary \ref{corollary cruciale} there exist $\a\in(0,1)$ and $\de>0$ (depending on $n$, $s$ and $\a$ only) such that the following holds: If $F$ is a $\Lambda$-minimizer of the $s$-perimeter with $|F\Delta B|<\de$ ($\Lambda$ as in \eqref{14 giovanni h}), then there is $u\in C^{1,\a}(\pa B)$ such that
  \[
  \pa F=\Big\{(1+u(x))\,x:x\in\pa B\Big\}\,,\qquad \|u\|_{C^1(\pa B)}<\e_\beta\,.
  \]
  Hence, by \eqref{14 giovanni h} and \eqref{14 giovanni freddo}, we can choose $M$ large enough (depending on $n$, $s$ and $\a$) in such a way that, for every $h\in\N$, there exists $u_h\in C^{1,\a}(\pa B)$ with
  \[
  \pa F_h=\Big\{(1+u_h(x))\,x:x\in\pa B\Big\}\,,\qquad \|u_h\|_{C^1(\pa B)}<\e_\beta\,.
  \]
  Let us set $t_h:=(|F_h|/|B|)^{1/n}$ and $G_h:=x_h+t_h\,F_h$ for $x_h$ such that $\int_{G_h}x\,dx=0$. By \eqref{14 giovanni freddo}, we can make $|t_h-1|$ small enough in terms of $\e_\beta$ to entail that for every $h\in\N$ there exists  $v_h\in C^{1,\a}(\pa B)$ with
  \[
  \pa G_h=\Big\{(1+v_h(x))\,x:x\in\pa B\Big\}\,,\qquad \|v_h\|_{C^1(\pa B)}<\e_\beta\,.
  \]
  By Theorem \ref{fugledeV} we conclude that
  \begin{equation}
    \label{14 fagioli}
  \F(B)\le\F(G_h)=t_h^{n-s}\,\per_s(F_h)+t_h^{n+\a}\,\beta\,V_\a(F_h)
  \le\max\{t_h^{n-s},t_h^{n+\a}\}\,\F(F_h)\,,
  \end{equation}
  which in turn gives, in combination with \eqref{14 giovanni2},
  \begin{equation}
    \label{14 fagioli2}
  \frac{\F(B)}{\max\{t_h^{n-s},t_h^{n+\a}\}}+M\,|F_h\Delta E_h|\le\F(B)\,.
  \end{equation}
  If $t_h=1$ for a value of $h$, then by \eqref{14 fagioli2} we find $F_h=E_h$ and thus $\F(F_h)=\F(E_h)<\F(B)$, a contradiction to \eqref{14 fagioli}. At the same time, since $\F(B)>0$, \eqref{14 fagioli2} implies that $t_h\ge1$ for every $h\in\N$. We may thus assume that $t_h>1$ for every $h\in\N$. Since $|F_h\Delta E_h|\ge ||F_h|-|B||=|B|\,(t_h^n-1)$, by \eqref{14 fagioli2} we find
  \[
  M\,|B|\,(t_h^n-1)\le\F(B)\,\Big(1-\frac{1}{t_h^{n+\a}}\Big)\,,
  \]
  where, say, $t_h\in(1,3/2)$ for every $h\in\N$. However, if $M$ is large enough depending on $n$, $s$, and $\a$ only, we actually have that
  \[
  M\,|B|\,(t^n-1)>\F(B)\,\Big(1-\frac{1}{t^{n+\a}}\Big)\,,\qquad\forall t\in(1,3/2)\,.
  \]
  We thus find a contradiction also in the case that $t_h>1$ for every $h\in\N$. This completes the proof of the lemma.
\end{proof}

\begin{proof}
  [Proof of Theorem \ref{thm 3}] Given $m>0$ let us define $\beta>0$ by setting
  \[
  \beta=\Big(\frac{m}{|B|}\Big)^{(n+\a)/n}\,\Big(\frac{|B|}{m}\Big)^{(n-s)/n}=\Big(\frac{m}{|B|}\Big)^{(s+\a)/n}\,.
  \]
  (Notice that $\beta<\beta_\star$ if and only if $m<m_\star$, since by \eqref{mstar} and \eqref{betastar} we have $m_\star=|B|\,\beta_\star^{n/(s+\a)}$.) By exploiting this identity and the scaling properties of $\per_s$ and $V_\a$, and denoting by $B[m]$ a ball of volume $m$, given $\de>0$ we notice that
  \[
  \per_s(B)+\beta\,V_\a(B)\le\per_s(F)+\beta\,V_\a(F)\,,\qquad\mbox{whenever $|F|=|B|$ and $|F\Delta B|<\de$}
  \]
  if and only if
  \[
  \per_s(B[m])+V_\a(B[m])\le\per_s(E)+V_\a(E)\,,\qquad\mbox{whenever $|E|=m$ and $|E\Delta B[m]|<\frac{m}{|B|}\,\de$}\,.
  \]
  As a consequence, Theorem \ref{thm 3} is equivalent to Lemma \ref{lemma infine}.
\end{proof}


\appendix

\section{A simple $\Gamma$-convergence result}\label{appendix gamma} Here we prove the $\Gamma$-convergence of $P_s$ to $P_{s_*}$ in the limit $s\to s_*$, with $s_*\in(0,1)$ fixed. Of course, if $|(E_h\Delta E)\cap K|\to 0$ for every $K\subset\subset\R^n$ and $s_h\to s_*\in(0,1)$ as $h\to\infty$, then by Fatou's lemma one easily obtains
$$
  P_{s_*}(E)\le \liminf_{h\to\infty}P_{s_h}(E_h)\,,
$$
that is the $\Gamma$-liminf inequality. The proof of the $\Gamma$-limsup inequality is only slightly longer. For the sake of simplicity, we shall limit ourselves to work with bounded sets (this is the case we need in the paper). Precisely, given a bounded set $F\subset\R^n$, we want to construct a sequence $\{F_h\}_{h\in\N}$ of bounded sets such that $|F_h\Delta F|\to\infty$ as $h\to\infty$ and
\begin{equation}
  \label{gamma limsup}
  \limsup_{h\to\infty}P_{s_h}(F_h)\le P_{s_*}(F)\,.
\end{equation}
We now prove \eqref{gamma limsup}. We start by recalling the following nonlocal coarea formula due to Visintin \cite{V},
\begin{equation}
  \label{coarea nonlocal}
  \int_{\R^n}dx\int_{\R^n}\frac{|u(x)-u(y)|}{|x-y|^{n+s}}\,dy=2\,\int_0^1\,P_s(\{u>t\})\,dt\,,\qquad s\in(0,1)\,,
\end{equation}
that holds true (as an identity in $[0,\infty]$) whenever $u:\R^n\to[0,1]$ is Borel measurable; see \cite[Lemma 10]{ADPM}. Next we use \cite[Proposition 14.5]{L} to infer that if $P_{s_*}(F)<\infty$ and we set $u_\e=1_F\star\rho_\e$, $\rho_\e$ a standard $\e$-mollifier, then
\begin{equation}\label{appro}
\lim_{\e\to 0^+}\int_{\R^n}dx\int_{\R^n}\frac{|u_\e(x)-u_\e(y)|}{|x-y|^{n+s_*}}\,dy=2\,P_{s_*}(F)\,.
\end{equation}
Combining \eqref{coarea nonlocal} and \eqref{appro} with a classical argument by De Giorgi, see, e.g. \cite[Theorem 13.8]{Ma}, we reduce the proof of \eqref{gamma limsup} to the case that $F$ is a bounded, smooth set. This implies that $P_s(F)<\infty$ for every $s\in(0,1)$. In particular, if we let $s_{**}\in(0,1)$ be such that $s_h<s_{**}$ for every $h\in\N$, then we trivially find that, for every $(x,y)\in\R^n\times\R^n$,
\[
\frac{1_{F\times F^c}(x,y)}{|x-y|^{n+s_h}}\le 1_{(F\times F^c)\cap\{|x-y|>1\}}(x,y)+\frac{1_{F\times F^c\cap\{|x-y|\le 1\}}(x,y)}{|x-y|^{n+s_{**}}}=:g(x,y)\,,
\]
where $g\in L^1(\R^n\times\R^n)$ thanks to the fact that $P_{s_{**}}(F)<\infty$. In particular,
\[
\lim_{h\to\infty}P_{s_h}(F)=P_{s_*}(F)\,,
\]
whenever $s_h\to s_*\in(0,1)$ as $h\to\infty$ and $F$ is a smooth bounded set. This proves \eqref{gamma limsup}.

\section{A geometric lemma} The following natural fact, which is well-known in the case of the classical perimeter, was used in the proof of Lemma \ref{lemma infine}. We give a proof since it may be useful elsewhere.

\begin{lemma}\label{lemma convessi}
  If $s\in(0,1]$ and $E\subset\R^n$ is such that $P_s(E)<\infty$, then $P_s(E\cap K)\le P_s(E)$ for every convex set $K\subset\R^n$.
\end{lemma}

\begin{proof}The case $s=1$ being classical, we can assume $s<1$.
Since any convex set can be written as a countable intersections of half-space,
it is enough to prove that $P_s(E\cap H)\le P_s(E)$ whenever $H$ is an half-space. By approximation, it suffices to prove this estimate when $E$ is bounded.
We now observe that, if we set $F:=E\cup H$,
using that $E\subset F$, $E\setminus H=F\setminus H$, and $F\cap H=H$, we get
\begin{eqnarray*}
P_s(E)-P_s(E\cap H)&=&\int_E \int_{E^c} \frac{dx\,dy}{|x-y|^{n+s}}
-\int_{E\cap H}\int_{(E\cap H)^c}\frac{dx\,dy}{|x-y|^{n+s}}\\
&=&\biggl( \int_{E\cap H} + \int_{E\setminus H}\biggr)\int_{E^c}\frac{dx\,dy}{|x-y|^{n+s}}
- \biggl( \int_{E^c} + \int_{E\setminus H}\biggr)\int_{E\cap H}\frac{dx\,dy}{|x-y|^{n+s}}\\
&=&\biggl(\int_{E^c} - \int_{E\cap H} \biggr)\int_{E\setminus H} \frac{dx\,dy}{|x-y|^{n+s}}\\
&\geq&\biggl(\int_{F^c} - \int_{F\cap H} \biggr)\int_{F\setminus H} \frac{dx\,dy}{|x-y|^{n+s}}.
\end{eqnarray*}
We now observe that (just by doing the above steps backward) the last term is formally equal to $P_s(F)-P_s(H)$. However, this does not really make sense as both $P_s(F)$ and $P_s(H)$
are actually infinite. For this reason, we have to consider a local version of $P_s$:
given a set $G$ and a domain $A$, we define the $s$-perimeter of $G$ inside $A$ as
$$
P_s(G;A):=\biggl(\int_{G\cap A}\int_{G^c\cap A}+\int_{G\cap A}\int_{G^c\cap A^c}+\int_{G\cap A^c}\int_{G^c\cap A}\biggr)\frac{dx\,dy}{|x-y|^{n+s}}.
$$
With this notation, if $B_R$ is a large ball which contains $E$ (recall that $E$ is bounded), since $F$ is equal to $H$ outside $B_R$ it is easy to check that
$$
\biggl(\int_{F^c} - \int_{F\cap H} \biggr)\int_{F\setminus H} \frac{dx\,dy}{|x-y|^{n+s}}
= P_s(F;B_R)-P_s(H;B_R).
$$
Applying
\cite[Proposition 17]{ADPM} we deduce that $P_s(F;B_R)-P_s(H;B_R)\geq 0$, concluding the proof.
\end{proof}

\section{About the constant $\beta_\star$}\label{section betastar} We have already noticed that, in order to show the equivalence between the two formulas \eqref{mstar} and \eqref{defbetstar} for $\beta_\star$, it suffices to show that, for every  $s\in(0,1]$ and $\a\in(0,n)$, one has
\begin{equation}
  \label{due}
  \frac{\lambda_k^s-\lambda_1^s}{\mu_k^\alpha-\mu_1^\alpha}\ge \frac{\lambda_2^s-\lambda_1^s}{\mu_2^\alpha-\mu_1^\alpha}\qquad \forall k \geq 2\,.
\end{equation}
{\it Proof of \eqref{due} in the case that $s\in(0,1)$ and $\a\in(0,1)$.} In this case, the repeated application of the factorial property of the gamma function shows that \eqref{due} is equivalent in proving that the quantity
$$
X_k:=\frac{\frac{\prod_{j=1}^k\left(j+\frac{n+s}2\right)}{\prod_{j=1}^k\left(j+\frac{n-2-s}2\right)} -1}
{\frac{\prod_{j=1}^k\left(j+\frac{n-\alpha}2\right)}{\prod_{j=1}^k\left(j+\frac{n-2+\alpha}2\right)} -1}
$$
attains its minimal value on $k\ge 1$ at $k=1$. To this end it is convenient to rewrite $X_k$ as follows: first, we notice that
$$
X_k=\frac{\frac{\prod_{j=1}^k\left(j+\frac{n-1}2+t\right)}{\prod_{j=1}^k\left(j+\frac{n-1}2-t\right)} -1}
{\frac{\prod_{j=1}^k\left(j+\frac{n-1}2+\tau\right)}{\prod_{j=1}^k\left(j+\frac{n-1}2-\tau\right)} -1},
\qquad
\text{where}
\quad
t:=\frac{1+s}2,\quad \tau:=\frac{1-\alpha}{2}\,,
$$
(and thus, $0<\tau<t$); second, we set
\begin{equation}
  \label{akbkckdk}
  a_k:=\prod_{j=2}^k\a_j,\qquad
b_k:=\prod_{j=2}^k\b_j,
\qquad
c_k:=\prod_{j=2}^k\g_j,\qquad
d_k:=\prod_{j=2}^k\de_j\,,
\end{equation}
where
$$
\a_k:=k+\frac{n-1}2+t,\qquad \b_k:=k+\frac{n-1}2-t,
$$
$$
\g_k:=k+\frac{n-1}2+\tau,\qquad \de_k:=k+\frac{n-1}2-\tau.
$$
In this way, $X_k\ge X_1$ for every $k\ge 2$ can be rephrased into
\begin{equation}
  \label{due ridotta}
  \frac{\frac{a_k\a_1 }{b_k\b_1} - 1}
{\frac{c_k\g_1}{d_k\de_1}-1} \geq \frac{\frac{\alpha_1 }{\beta_1} - 1}
{\frac{\gamma_1}{\delta_1}-1}\,,\qquad\forall k\ge 2\,.
\end{equation}
It is useful to rearrange the terms in \eqref{due ridotta} and rewrite it as
%
\begin{equation}
  \label{due ridotta x}
a_k d_k \a_1(\g_1-\de_1)+b_kd_k(\a_1\de_1 -\b_1\g_1)+b_kc_k\g_1(\b_1 - \a_1) \geq 0\,,\qquad\forall k\ge 2\,.
\end{equation}
We now observe that, setting $\ell:=(n+1)/2$, we have
$$
\a_1=\ell+t,\quad \b_1=\ell-t,\quad \g_1=\ell+\tau,\quad\delta_1=\ell-\tau,\quad \a_1\de_1 -\b_1\g_1=2\ell (t-\tau).
$$
Hence, substituting these formulas into the above expression we find that
\begin{eqnarray*}
  \mbox{left-hand side of \eqref{due ridotta x}}&=&2a_kd_k(\ell+t)\tau + 2 b_kd_k\ell (t-\tau) -2b_kc_k(\ell+\tau)t
  \\
  &=&2(a_kd_k - b_kc_k)t\tau+2(a_k-b_k)d_k\ell \tau - 2 (c_k-d_k)b_k\ell t\,.
\end{eqnarray*}
Therefore \eqref{due ridotta x} follows by showing that
\begin{eqnarray}
\label{eq:induction1}
&a_kd_k \geq  b_kc_k\,,&\qquad \forall\,k \geq 2\,,
\\
\label{eq:induction2}
&(a_k-b_k)d_k \tau \geq  (c_k-d_k)b_k t\,,&\qquad \forall\,k \geq 2\,.
\end{eqnarray}
To prove \eqref{eq:induction1} it suffices to observe that
$$
\a_j \de_j -\b_j \g_j=2\Bigl(j+\frac{n-1}{2}\Bigr)(t-\tau) \geq 0\qquad \forall\,j \geq 1,
$$
so that
$$
a_kd_k =\prod_{j=2}^k\a_j\de_j \geq \prod_{j=2}^k\b_j\g_j=b_kc_k\,,\qquad\forall k\ge 2\,,
$$
as desired. We now prove \eqref{eq:induction2} by induction. A simple manipulation shows that \eqref{eq:induction2} in the case $k=2$ is equivalent to $d_2\ge b_2$, which is true, so that we directly focus on the inductive hypothesis. By noticing that $a_{k+1}=a_k\a_{k+1}$, and that analogous identities hold for $\beta_k$, $\g_k$ and $\de_k$, we can equivalently reformulate the $(k+1)$-case of \eqref{eq:induction2} as
$$
(a_k\a_{k+1}-b_k\b_{k+1})d_k\de_{k+1}\tau \geq  (c_k\g_{k+1}-d_k\de_{k+1})b_k\b_{k+1} t\,.
$$
This last inequality can be conveniently rewritten as
\begin{eqnarray*}
&&a_k(\a_{k+1}-\beta_{k+1}) d_k\de_{k+1}\tau+\b_{k+1}\de_{k+1}(a_k-b_k)d_k\tau
\\
&&\hspace{3cm}\geq c_k(\gamma_{k+1}-\de_{k+1})b_k\beta_{k+1}t + \b_{k+1}\de_{k+1}(c_k-d_k)b_k t\,.
\end{eqnarray*}
Indeed, by the inductive hypothesis $(a_k-b_k)d_k \tau \geq  (c_k-d_k)b_k t$, it is clear that a sufficient condition for this last inequality (and thus, for \eqref{eq:induction2}) to hold true, is that
\begin{eqnarray}\label{due ridotta xx}
a_k(\a_{k+1}-\beta_{k+1}) d_k\de_{k+1}\tau\ge c_k(\gamma_{k+1}-\de_{k+1})b_k\beta_{k+1}t\,.
\end{eqnarray}
By $\a_{k+1}-\beta_{k+1}=2t$ and $\gamma_{k+1}-\de_{k+1}=2\tau$, \eqref{due ridotta xx} is equivalent to
$$
2 (a_kd_k \delta_{k+1} - b_k c_k \beta_{k+1})t \tau\ge0\,.
$$
Finally, this inequality holds true because of \eqref{eq:induction1} and the fact that $\delta_{k+1}\geq \beta_{k+1}$. This complete the proof of \eqref{eq:induction2}, and thus of \eqref{due} in the case that $\s\in(0,1)$ and $\a\in(0,1)$.

\medskip

\noindent {\it Proof of \eqref{due} in the case that $s\in(0,1)$ and $\a\in(1,n)$.} By the factorial property of the gamma function \eqref{due} is  now equivalent in proving that $X_k\ge X_1$ for every $k\ge 2$, where we have now set
$$
X_k:=\frac{\frac{\prod_{j=1}^k\left(j+\frac{n+s}2\right)}{\prod_{j=1}^k\left(j+\frac{n-2-s}2\right)} -1}
{1-\frac{\prod_{j=1}^k\left(j+\frac{n-\alpha}2\right)}{\prod_{j=1}^k\left(j+\frac{n-2+\alpha}2\right)}}\,.
$$
We notice that
$$
X_k=\frac{\frac{\prod_{j=1}^k\left(j+\frac{n-1}2+t\right)}{\prod_{j=1}^k\left(j+\frac{n-1}2-t\right)} -1}
{1-\frac{\prod_{j=1}^k\left(j+\frac{n-1}2-\tau\right)}{\prod_{j=1}^k\left(j+\frac{n-1}2+\tau\right)}}\,,
\qquad
\text{where}
\quad
t:=\frac{1+s}2,\quad \tau:=\frac{\alpha-1}{2}\,.
$$
We next define $a_k$, $b_k$, $c_k$ and $d_k$ as in \eqref{akbkckdk}, with $\a_k$, $\b_k$, $\g_k$ and $\de_k$ given by
$$
\a_k:=k+\frac{n-1}2+t,\qquad \b_k:=k+\frac{n-1}2-t,
$$
$$
\g_k:=k+\frac{n-1}2-\tau,\qquad \de_k:=k+\frac{n-1}2+\tau.
$$
We have thus reformulated \eqref{due} as
\[
  \frac{\frac{a_k\a_1 }{b_k\b_1} - 1}
{1-\frac{c_k\g_1}{d_k\de_1}} \geq \frac{\frac{\alpha_1 }{\beta_1} - 1}
{1-\frac{\gamma_1}{\delta_1}}\,,\qquad\forall k\ge 2\,,
\]
which is in turn equivalent to
\begin{equation}
  \label{due ridotta xxx}
a_k d_k \a_1(\de_1-\g_1)+b_kd_k(\b_1\g_1-\a_1\de_1)+b_kc_k\g_1(\a_1-\b_1) \geq 0\,,\qquad\forall k\ge 2\,.
\end{equation}
If we set $\ell=(n+1)/2$, then we find
$$
\a_1=\ell+t,\quad \b_1=\ell-t,\quad \g_1=\ell-\tau,\quad\delta_1=\ell+\tau,\quad \a_1\de_1 -\b_1\g_1=2\ell (t+\tau)\,,
$$
so that
\begin{eqnarray*}
  \mbox{left-hand side of \eqref{due ridotta xxx}}&=&2a_kd_k(\ell+t)\tau - 2 b_kd_k\ell (t+\tau) +2b_kc_kt(\ell-\tau)
  \\
  &=&2(a_kd_k - b_kc_k)t\tau+2(a_k-b_k)d_k\ell \tau + 2 (c_k-d_k)b_k\ell t\,.
\end{eqnarray*}
We are thus left to prove that
\begin{eqnarray}
\label{eq:induction1x}
&a_kd_k \geq  b_kc_k\,,&\qquad \forall\,k \geq 2\,,
\\
\label{eq:induction2x}
&(a_k-b_k)d_k \tau \geq  (d_k-c_k)b_k t\,,&\qquad \forall\,k \geq 2\,.
\end{eqnarray}
To prove \eqref{eq:induction1x} it suffices to observe that
$$
\a_j \de_j -\b_j \g_j=2\Bigl(j+\frac{n-1}{2}\Bigr)(t+\tau) \geq 0\qquad \forall\,j \geq 1,
$$
where $t>0$ and $\tau>0$. To prove  \eqref{eq:induction2x} we argue once again by induction. One easily sees that \eqref{eq:induction2} in the case $k=2$ is equivalent to say that $d_2\ge b_2$, which is true also in the present case. We now check the inductive hypothesis. The $(k+1)$-case of \eqref{eq:induction2} is now equivalent to
$$
(a_k\a_{k+1}-b_k\b_{k+1})d_k\de_{k+1}\tau \geq  (d_k\de_{k+1}-c_k\g_{k+1})b_k\b_{k+1} t\,.
$$
We reformulate this as
\begin{eqnarray*}
&&a_k(\a_{k+1}-\beta_{k+1}) d_k\de_{k+1}\tau+\b_{k+1}\de_{k+1}(a_k-b_k)d_k\tau
\\
&&\hspace{3cm}\geq c_k(\de_{k+1}-\gamma_{k+1})b_k\beta_{k+1}t + \b_{k+1}\de_{k+1}(d_k-c_k)b_k t\,.
\end{eqnarray*}
By the inductive hypothesis $(a_k-b_k)d_k \tau \geq  (d_k-c_k)b_k t$, thus we are left to check that
\begin{eqnarray}\label{due ridotta xxxx}
a_k(\a_{k+1}-\beta_{k+1}) d_k\de_{k+1}\tau\ge c_k(\de_{k+1}-\gamma_{k+1})b_k\beta_{k+1}t\,.
\end{eqnarray}
By $\a_{k+1}-\beta_{k+1}=2t$ and $\de_{k+1}-\gamma_{k+1}=2\tau$, \eqref{due ridotta xxxx} is equivalent to $2 (a_kd_k \delta_{k+1} - b_k c_k \beta_{k+1})t \tau\ge0$, which is true thanks to \eqref{eq:induction1x} and  $\delta_{k+1}\geq \beta_{k+1}$. The proof of \eqref{eq:induction2x}, thus of \eqref{due} in the case that $\s\in(0,1)$ and $\a\in(1,n)$, is now complete.

\medskip

\noindent {\it Proof of \eqref{due} in the remaining cases.} The case that $s\in(0,1)$ and $\a=1$ is covered by taking the limit as $\a\to 1^-$ with $s$ fixed in \eqref{due} for $\a\in(0,1)$. This proves \eqref{due} for every $s\in(0,1)$ and $\a\in(0,n)$. The case $s=1$ is recovered by multiplying \eqref{due} by $1-s$ when $s\in(0,1)$ and then taking the limit as $s\to 1^-$ with $\a$ fixed. The proof of \eqref{due} is now complete.

\section{Corrections}
The pointwise estimate which appears five lines before \eqref{fug9} is wrong. A similar pointwise estimate used to deduce \eqref{quantpotinterm} is also wrong.

\medskip

Let us first show how to correct the proof of Theorem~\ref{fuglede}. A direct computation shows that
\begin{equation}\label{uno}
|g'(t)|=
	C\Bigg|\int_{\partial B}d\mathcal H^{n-1}_x\!\int_{\partial B}d\mathcal H^{n-1}_y\!\int_{u(x)}^{u(y)}dr \int_{u(x)}^{u(y)}\frac{2t(r-\rho)^2 + \big[(1+t\rho)r + (1+tr)\rho\big]|x-y|^2}{\big(t^2(r-\rho)^2 + (1+tr)(1+t\rho)|x-y|^2\big)^{\frac{n+s}{2}+1}}\,d\varrho\Bigg|\,. 
\end{equation}
Since $r,\varrho\in[-1/2,1/2]$ and $0<t<1$, we have
\begin{equation}\label{duedue}
\frac{ \big[(1+t\rho)r + (1+tr)\rho\big]|x-y|^2}{\big(t^2(r-\rho)^2 + (1+tr)(1+t\rho)|x-y|^2\big)^{\frac{n+s}{2}+1}}\leq \frac{C}{|x-y|^{n+s}}\,,
\end{equation}
while at the same time we can estimate
$$
	\Bigg|\int_{\partial B}d\mathcal H^{n-1}_x\!\int_{\partial B}d\mathcal H^{n-1}_y\int_{u(x)}^{u(y)}dr \int_{u(x)}^{u(y)}\frac{2t(r-\rho)^2}{\big(t^2(r-\rho)^2 + (1+tr)(1+t\rho)|x-y|^2\big)^{\frac{n+s}{2}+1}}\,d\varrho\Bigg|
	$$
by
$$
C\int_{\partial B}d\mathcal H^{n-1}_x\int_{\partial B}d\mathcal H^{n-1}_y\int_{u(x)\land u(y)}^{u(x)\lor u(y)}dr \int_{u(x)\land u(y)}^{u(x)\lor u(y)}\frac{2t(r-\rho)^2}{\big(t^2(r-\rho)^2 + |x-y|^2\big)^{\frac{n+s}{2}+1}}\,d\varrho\,.
$$
Since the function
$$
t\mapsto\frac{t}{(t^2(r-\rho)^2 + |x-y|^2)^p}\,,\qquad p>\frac12\,,
$$
attains its maximum on $[0,\infty)$ at $t=|x-y|/[(2p-1)^{1/2}|r-\varrho|]$,  setting $p=(n+s+2)/2$, the integral above can be bounded by
\begin{equation}\label{tre}
C\,\int_{\partial B}d\mathcal H^{n-1}_x\int_{\partial B}d\mathcal H^{n-1}_y\int_{u(x)\land u(y)}^{u(x)\lor u(y)}dr \int_{u(x)\land u(y)}^{u(x)\lor u(y)}\frac{|r-\varrho|}{|x-y|^{n+s+1}}d\varrho
\end{equation}
Now, if $a<b$, then
$$
\int_a^b\int_a^b|r-\varrho|\,drd\varrho=2\int_a^b\int_{\varrho}^b(r-\varrho)\,drd\varrho=\frac{(b-a)^3}{3}\,;
$$
therefore, the integral in \eqref{tre} is finally controlled by
$$
\int_{\partial B}d\mathcal H^{n-1}_x\int_{\partial B}\frac{|u(x)-u(y)|^3}{|x-y|^{n+s+1}}\,d\mathcal H^{n-1}_y\leq \frac12\int_{\partial B}d\mathcal H^{n-1}_x\int_{\partial B}\frac{|u(x)-u(y)|^2}{|x-y|^{n+s}}\,d\mathcal H^{n-1}_y\,,
$$
where we used the fact that $|u(x)-u(y)|\leq\frac12 |x-y|$. From this estimate, recalling \eqref{uno} and \eqref{duedue}, we conclude that
$$
|g'(t)|\leq C\int_{\partial B}d\mathcal H^{n-1}_x\int_{\partial B}\frac{|u(x)-u(y)|^2}{|x-y|^{n+s}}\,d\mathcal H^{n-1}_y
$$
and the proof goes as before.

\medskip

The same argument can be used to correct the proof of Lemma~\ref{lemma fuglede alpha}.

\end{document}